\documentclass[a4paper,11pt,reqno,noindent]{amsart}
\usepackage[centertags]{amsmath}
\usepackage{amsfonts,amssymb,amsthm,dsfont,cases,amscd,esint,enumerate,mathabx}
\usepackage[T1]{fontenc}
\usepackage[english]{babel}
\usepackage[applemac]{inputenc}
\usepackage{newlfont}
\usepackage{color}
\usepackage[body={15cm,21.5cm},centering]{geometry} 
\usepackage{fancyhdr}
\usepackage{enumerate}
\usepackage{tikz,tikzsymbols}
\usetikzlibrary{arrows.meta,patterns}

\theoremstyle{plain}

\newtheorem{theor}{Theorem}[section]
\newtheorem{lem}[theor]{Lemma}
\newtheorem{prop}[theor]{Proposition}
\newtheorem{cor}[theor]{Corollary}
\newtheorem{defin}[theor]{Definition}
\theoremstyle{definition}

\newtheorem{rem}[theor]{Remark}

\numberwithin{equation}{section}

\newcommand{\N}{\mathbb N}
\newcommand{\R}{\mathbb R}
\newcommand{\Z}{\mathbb Z}

\newcommand{\e}{\varepsilon}

\newcommand{\Lc}{\mathcal{L}}
\newcommand{\Bc}{\mathcal{B}}
\newcommand{\Dm}{\mathbb{D}}

\newcommand{\Kk}{\boldsymbol K}
\newcommand{\Ll}{\boldsymbol L}
\newcommand{\Uu}{\boldsymbol U}
\newcommand{\Gc}{\mathcal G}
\newcommand{\Ac}{\mathcal A}
\newcommand{\Fc}{\mathcal F}
\newcommand{\Rc}{\mathcal R}
\newcommand{\Hf}{\mathfrak H}
\newcommand{\pv}{\operatorname{p.v.}}
\newcommand{\Sym}{{\operatorname{sym}}}
\newcommand{\dist}{{\operatorname{dist}}}
\newcommand{\supp}{{\operatorname{supp}}}
\newcommand{\Id}{\operatorname{Id}}
\newcommand{\E}{\mathbb{E}}
\newcommand{\ee}{e}
\newcommand{\Aa}{\boldsymbol a}
\newcommand{\Bb}{\boldsymbol b}
\newcommand{\Ld}{\operatorname{L}}
\newcommand{\step}[1]{\noindent \textit{Step} #1.}
\newcommand{\substep}[1]{\noindent \textit{Substep} #1.}
\newcommand{\Pm}{\mathbb{P}}

\newcommand{\pr}[1]{\mathbb{P}\left[#1\right]}
\newcommand{\cro}[1]{[\![#1]\!]}
\newcommand{\expec}[1]{\mathbb{E}\left[ #1 \right]}
\newcommand{\expecm}[1]{\mathbb{E}\big[ #1 \big]}

\newcommand{\var}[1]{\mathrm{Var}\left[#1\right]}
\newcommand{\cov}[2]{\operatorname{Cov}\left[{#1};{#2}\right]}

\usepackage[colorlinks,linkcolor=black,citecolor=black,urlcolor=black]{hyperref}

\title{On Bourgain's approach to stochastic homogenization}

\author[M. Duerinckx]{Mitia Duerinckx}
\address[Mitia Duerinckx]{Universit\'e Libre de Bruxelles, D\'epartement de Math\'ematique, 1050~Brussels, Belgium}
\email{mitia.duerinckx@ulb.be}
\author[M. Lemm]{Marius Lemm}
\address[Marius Lemm]{University of T\"ubingen, Department of Mathematics, 72076~T\"ubingen, Germany}
\email{marius.lemm@uni-tuebingen.de}
\author[F. Pagano]{Fran\c cois Pagano}
\address[Fran\c cois Pagano]{Universit\'e de Gen\`eve, D\'epartement de Math\'ematiques, 1205, Geneva, Switzerland}
\email{francois.pagano@unige.ch}

\begin{document}
\selectlanguage{english}

\begin{abstract}
In 2018, Bourgain pioneered a novel {perturbative} harmonic-analytic approach to the stochastic homogenization theory of discrete elliptic equations with weakly random i.i.d.\@ coefficients. The approach was subsequently refined to {show that {homogenized approximations}
of ensemble averages can be derived to a precision four times better than almost sure homogenized approximations,}
which was unexpected by the state-of-the-art homogenization theory. In this paper, we grow this budding theory in various directions: First, we prove that the approach is robust by extending it to the continuum setting with exponentially mixing random coefficients. Second, we give a new proof via Malliavin calculus in the case of Gaussian coefficients, which avoids the main technicality of Bourgain's original approach. This new proof also  applies to strong Gaussian correlations {with power-law decay.}
Third, we extend Bourgain's approach to the study of fluctuations by constructing weak correctors up to order~$2d$, which also clarifies the link between Bourgain's approach and the standard corrector approach to homogenization.
Finally, we draw several consequences from those different results, both for quantitative homogenization of ensemble averages and for asymptotic expansions of the annealed Green's function.
\end{abstract}

\maketitle

\setcounter{tocdepth}{1}
\tableofcontents

\section{Introduction}

{In spatial dimension $d\ge1$, given an underlying probability space $(\Omega,\Pm)$, consider a stationary measurable random coefficient field $\Aa:\R^d\times\Omega\to\R^{d\times d}$} satisfying the following uniform ellipticity and boundedness assumptions, $\Pm$-almost surely,
\begin{equation}\label{eq:ellipt-a}
e\cdot\Aa(x,\omega)e\,\ge\,\tfrac1{C_0}|e|^2,\qquad|\Aa(x,\omega)e|\,\le\,C_0|e|,\qquad\text{for all $x,e\in\R^d$}.
\end{equation}
{Here, stationarity means that the (finite-dimensional) law of $\Aa$ is shift-invariant; we refer to Section~\ref{sec:assumptions} for the detailed assumptions used throughout this work.}
Given a deterministic field $f\in\Ld^2(\R^d)^d$, for almost all $\omega\in\Omega$, we consider the following heterogeneous elliptic problem in $\R^d$,
\begin{equation}\label{eq:ellipt}
-\nabla\cdot\Aa(\tfrac\cdot\e,\omega)\nabla u_{\e,f}(\cdot,\omega)=\nabla\cdot f,\qquad u_{\e,f}(\cdot,\omega)\in \dot H^1(\R^d),
\end{equation}
where $\e>0$ stands for the length scale of heterogeneities. For simplicity, we shall drop the $\omega$-dependence in the notation, simply writing for instance $\Aa(x)=\Aa(x,\omega)$ and $u_{\e,f}=u_{\e,f}(\cdot,\omega)$, and we shall further abbreviate $\Aa_\e(x):=\Aa(\frac x\e)$.
To have a physical picture in mind, we may think of the solution field $\nabla u_{\e,f}$ as an electric field generated by a given charge distribution $\nabla\cdot f$ in an heterogeneous material with conductivity $\Aa_\e$.

In this setting, the goal of \textit{homogenization theory} is to describe as accurately as possible the solution field~$\nabla u_\e$ in the ``macroscopic'' limit $\e\downarrow0$.
The present contribution builds on a new approach to stochastic homogenization pioneered by Bourgain~\cite{Bourgain-18} in~2018.
As described in Section~\ref{sect:background} below,
the starting point of the approach is inspired by the Fourier method developed in the early works of Conlon and Naddaf~\cite{CN,conlon2000green}, and  also studied by Sigal~\cite{Sigal}: it starts by dividing the description of the solution field~$\nabla u_{\e,f}$ between its ensemble average~$\expec{\nabla u_{\e,f}}$ and its fluctuation $\nabla u_{\e,f}-\expec{\nabla u_{\e,f}}$. Then, it proceeds by viewing the homogenization problem as a regularity question for suitable Fourier symbols. In the weakly random regime, this regularity question can be efficiently tackled by perturbative methods. More precisely, in~\cite{Bourgain-18}, Bourgain focused on the ensemble-averaged solution field and investigated the regularity of the corresponding Fourier symbol in the simplified setting of discrete elliptic equations with weakly random i.i.d.\ coefficients on~$\Z^d$. In an improved version obtained in~\cite{Lemm-18}, in the same discrete setting, Bourgain's regularity result has led to a refined homogenized description of ensemble averages, which took the homogenization community by surprise~\cite{DGL,D-21a}: in a nutshell, it was shown that the ensemble-averaged field $\expec{\nabla u_{\e,f}}$ allows for an homogenized approximation with an accuracy \emph{four times better} than the field $\nabla u_{\e,f}$ itself. This goes far beyond what could be obtained from the standard corrector approach to homogenization~\cite{Gu-17,DO1}.
In the present work, we extend this new theory in three main directions.
\begin{enumerate}[---]
\item First, while only the case of {\it discrete} elliptic equations with i.i.d.\ coefficients was originally considered in~\cite{Bourgain-18,Lemm-18}, we show the robustness of the approach by extending it to the continuum setting and to coefficient fields with stretched exponential $\alpha$-mixing rate; see Theorem~\ref{th:exp-mix}. This is achieved by a suitable coarse-graining argument.
\smallskip\item {Second, we investigate the possible extension to strongly-correlated coefficient fields:
as an illustrative model, we focus on
the Gaussian setting with power-law decaying correlations
and we show that a transition occurs at the power-law exponent~$2d$ for the accuracy of the homogenized description of ensemble averages; see Theorem~\ref{th:correl}.
To this aim, we appeal to Malliavin calculus and discover how, when available, stochastic calculus leads to a completely new route to harness both oscillatory and probabilistic cancellations in the perturbation series: this reduces the main difficulty in Bourgain's original approach, avoiding any use of the so-called disjointification lemma.\footnote{For discrete elliptic equations with i.i.d.\@ coefficients, as originally studied in~\cite{Bourgain-18,Lemm-18}, we could similarly use the so-called Glauber calculus developed in~\cite{DGO1,D-20a} to avoid any need for Bourgain's disjointification lemma. This is an immediate consequence of our use of Malliavin calculus in the proof of Theorem~\ref{th:correl} and we skip the detail for conciseness.}}
\smallskip\item Third, we show how the approach can be extended from the description of ensemble averages to further describe fluctuations {of the solution field}; see Theorem~\ref{th:weak-cor}. This connects to the question of existence of the so-called weak correctors introduced in~\cite{D-21a} and it sheds a new light on the topic, in particular revealing the possible limitations of the non-perturbative approach initiated in~\cite{D-21a}.
\end{enumerate}
We refer to Section~\ref{sec:mainres} for precise statements of the main results, as well as for their consequences on the quantitative homogenization of ensemble averages and on asymptotic expansions of the annealed Green's function.
While showing that the approach pioneered by Bourgain is more robust and powerful than first realized,
we however emphasize that as in~\cite{Bourgain-18,Lemm-18} we are still restricted to a \emph{weakly} disordered regime
(that is, small ellipticity contrast).
The validity of corresponding non-perturbative results beyond those of~\cite{D-21a} remains a wide open question and is known as the \textit{Bourgain--Spencer conjecture}~\cite{DGL}; see Sections~\ref{sec:Bourgain-res} and~\ref{sec:weak-cor}.

\subsection*{Plan of the paper}
In Section~\ref{sect:background}, we revisit the background around Bourgain's approach and we prove in particular new results on the link between homogenized approximations and regularity questions for suitable Fourier symbols.
Our main results and some applications are stated in Section~\ref{sec:mainres}. The proofs are split into the six subsequent sections.
In Section~\ref{sect:detest}, we develop a continuum analog of the deterministic estimates from~\cite{DGL,Lemm-18} by using suitable mixed Lebesgue spaces, where local averaging is designed to handle singularities. In Section~\ref{sect:stretchedexponentialmixingproof}, we prove Theorem~\ref{th:exp-mix} about the stretched exponential mixing setting by using a coarse-graining argument and the deterministic estimates of Section~\ref{sect:detest}. In Section~\ref{sect:correlatedGaussianproof}, we prove Theorem~\ref{th:correl} about the Gaussian setting with power-law correlations by means of Malliavin calculus.
In Section~\ref{sect:weakcorrectors}, we prove Theorem~\ref{th:weak-cor} on the construction of weak correctors.
Finally, Sections~\ref{sect:massive} and~\ref{sect:corproofs} are devoted to some applications of our main results (Corollaries~\ref{cor:applyDGL}--\ref{cor:higherorder}).

\section{Alternative perspective on homogenization}\label{sect:background}

In this section, we recall the standard corrector approach to homogenization and compare it to an alternative approach first initiated by Conlon and Naddaf~\cite{CN,conlon2000green} and rediscovered by Sigal~\cite{Sigal}.
We explain how the latter leads us to viewing the existence of two-scale expansions in homogenization theory as a regularity question for suitable Fourier symbols.
This new perspective on homogenization provides the starting point of Bourgain's analysis~\cite{Bourgain-18}, which we shall further develop in the present work.

\subsection{Standard corrector approach to homogenization}\label{sec:standard-hom}
Building on the expected separation of scales, a standard approach to describe the solution $u_{\e,f}$ of the heterogeneous elliptic problem~\eqref{eq:ellipt} is based on postulating a formal two-scale expansion, see e.g.~\cite{BLP-78},
\begin{equation}\label{eq:2sc}
u_{\e,f}~\sim~\bar u_{\e,f}+\sum_{{n \geq 1}}\e^n\varphi^n_{j_1\ldots j_n}(\tfrac\cdot\e)\,\nabla^n_{j_1\ldots j_n}\bar u_{\e,f},\qquad\text{in $\Ld^2(\Omega;\dot H^1(\R^d))$},
\end{equation}
where we use Einstein's summation convention on repeated indices $1\le j_1,\ldots,j_n\le d$.
This amounts to approaching $u_{\e,f}$ as a sum of small modulations at scale $O(\e)$ around a deterministic profile~$\bar u_{\e,f}$. Modulations are given by so-called {\it correctors} $\{\varphi^n\}_{n\ge1}$, which we would expect to construct as stationary random fields just like the underlying coefficient field $\Aa$ itself.
If such an expansion is possible, then we find that it must necessarily be characterized as follows:
\begin{enumerate}[---]
\item The (higher-order) correctors $\{\varphi^n\}_{n\ge1}$ must be defined iteratively by letting $\varphi^n_{j_1\ldots j_n}$ be the stationary random field that has vanishing expectation, finite second moments, and satisfies almost surely the following equation in the weak sense on $\R^d$,
\begin{equation}\label{eq:correctors}
\quad-\nabla\cdot\Aa\nabla\varphi^n_{j_1\ldots j_n}\,=\,\nabla\cdot\big(\Aa\varphi^{n-1}_{j_1\ldots j_{n-1}}\ee_{j_n}\big)
+\ee_{j_n}\cdot P^\bot\Aa\big(\nabla\varphi^{n-1}_{j_1\ldots j_{n-1}}+\varphi^{n-2}_{j_1\ldots j_{n-2}}\ee_{j_{n-1}}\big),
\end{equation}
with the conventions $\varphi^0\equiv1$ and $\varphi^{-1}\equiv0$, and with the notation $P^\bot=\Id-\E$.
\smallskip\item The (higher-order) homogenized solution $\bar u_{\e,f}$ satisfies
\begin{equation}\label{eq:homog-eqn}
\qquad-\nabla\cdot\Big(\sum_{n\ge1}\bar\Aa^n_{j_1\ldots j_{n-1}}(\e\nabla)^{n-1}_{j_1\ldots j_{n-1}}\Big)\nabla\bar u_{\e,f} \,=\,\nabla\cdot f,
\end{equation}
which is viewed as a suitable dispersive correction of the standard (first-order) homogenized equation $-\nabla\cdot\bar\Aa^1\nabla\bar u_f=\nabla\cdot f$, in terms of the so-called (higher-order) homogenized tensors
\begin{equation}\label{eq:homcoeff}
\quad\bar\Aa^{n}_{j_1\ldots j_{n-1}}\ee_{j_{n}}\,=\,\expecm{\Aa\big(\nabla\varphi^{n}_{j_1\ldots j_{n}}+\varphi^{n-1}_{j_1\ldots j_{n-1}}\ee_{j_{n}}\big)},
\qquad n\ge1.
\end{equation}
Note that, due to the dispersive corrections, equation~\eqref{eq:homog-eqn} requires a suitable regularization to ensure its well-posedness; see Remark~\ref{rem:sol-homog-high} below. If $f$ has compactly supported Fourier transform, though, well-posedness holds for $\e$ small enough.
\end{enumerate}
The expansion~\eqref{eq:2sc} would give two pieces of information:
\begin{enumerate}[(A)]
\item The ensemble average $\expec{\nabla u_{\e,f}}\sim\nabla\bar u_{\e,f}$ would satisfy an homogenized equation of the form~\eqref{eq:homog-eqn}.
\smallskip\item The fluctuation $\nabla u_{\e,f}-\expec{\nabla u_{\e,f}}$ has spatial oscillations on the scale $O(\e)$, just as the coefficient field $\Aa_\e$ itself, and these would be captured by the two-scale expansion~\eqref{eq:2sc} in form of
\begin{equation}\label{eq:2sc-class}
\qquad\nabla u_{\e,f}-\expec{\nabla u_{\e,f}}\,\sim\,\nabla\sum_{n\ge1}\e^n\varphi^n_{j_1\ldots j_n}(\tfrac\cdot\e)\nabla^{n-1}_{j_1\ldots j_{n-1}}\expec{\nabla_{j_n} u_{\e,f}},
\end{equation}
or equivalently, expanding the gradient in the right-hand side,
\begin{equation*}
\qquad\nabla u_{\e,f}\,\sim\,\sum_{n\ge1}\e^{n-1}\psi^n_{j_1\ldots j_n}(\tfrac\cdot\e)\nabla^{n-1}_{j_1\ldots j_{n-1}}\expec{\nabla_{j_{n}} u_{\e,f}},
\end{equation*}
in terms of $\psi^n_{j_1\ldots j_n}:=\nabla\varphi^{n}_{j_1\ldots j_{n}}+\varphi^{n-1}_{j_1\ldots j_{n-1}}e_{j_{n}}$, which are stationary random fields satisfying $\E[\psi^1_{j}]=e_j$ and $\E[\psi^n]=0$ for $n>1$.
This is viewed as a series of stationary mean-zero oscillatory modulations around the ensemble average $\E[\nabla u_{\e,f}]$.
\end{enumerate}
In case of a periodic coefficient field $\Aa$, all correctors $\{\varphi^n\}_{n\ge1}$ can indeed be constructed as periodic solutions of the corrector equations~\eqref{eq:correctors}, and the different series above are all convergent for $\e\ll1$ small enough provided that~$f$ is smooth enough (say, provided that~$f$ has compactly supported Fourier transform); see e.g.~\cite[Proposition~3.3]{DGL}.
In contrast, in the case of a stationary random coefficient field $\Aa$, it is well-known that higher-order correctors cannot all be constructed as well-behaved stationary objects. More precisely, under suitable mixing assumptions, only the correctors $\varphi^n$'s with $n<\lceil\frac d2\rceil$ can be defined in general as stationary random fields with bounded second moments.
In this random setting, the above asymptotic expansions must then be truncated to order $\lceil\frac d2\rceil$ and in turn only yield an accurate description of the ensemble average $\expec{\nabla u_\e}$ and of the fluctuation $\nabla u_\e-\expec{\nabla u_\e}$ to order~$O(\e^{d/2-})$; see e.g.~\cite{Gu-17,DO1}.

\subsection{Conlon--Naddaf--Sigal alternative approach}
We now turn to another way to approach the homogenization question for~\eqref{eq:ellipt}, which was initiated in the early works of Conlon and Naddaf~\cite{CN,conlon2000green} and was rediscovered in slightly different terms by Sigal a few years ago in an unpublished note~\cite{Sigal}.
Taking Sigal's point of view, letting~$P=\E$ and~$P^\bot=\Id-\E$ on $\Ld^2(\R^d\times\Omega)$, we consider the block decomposition
\begin{equation}\label{eq:block-schur}
\nabla\cdot\Aa\nabla\,\equiv\,\
\begin{pmatrix}
\nabla\cdot P^\bot\Aa P^\bot\nabla&\nabla\cdot P^\bot\Aa P\nabla\\
\nabla\cdot P\Aa P^\bot\nabla&\nabla\cdot P\Aa P\nabla
\end{pmatrix},
\end{equation}
and a direct application of the Schur complement formula then yields the following {non-asymptotic} version of items~(A)--(B) above.
The first instance of this result can be found in the work of Conlon and Naddaf~\cite{CN,conlon2000green} in form of a related representation formula for the Green's function (see~\cite[(2.4)]{CN}, as well as~\cite[(8.1)]{Conlon-Spencer-14a} or~\cite[(6.5)]{Conlon-Giunti-Otto-17}). Sigal's note~\cite{Sigal} only contains a formulation of item~(A$'$) below for ensemble averages.
A short proof is included in Appendix~\ref{app:Sigal} for completeness.

\begin{lem}[Conlon, Naddaf, Sigal]$ $\label{lem:Sigal}
Let $\Psi(\cdot,\nabla)$ be the bounded pseudo-differential operator $\Ld^2(\R^d)^d\to\Ld^2(\R^d\times\Omega)^d$ given by
\begin{equation}\label{eq:defin-K}
\Psi(\cdot,\nabla)\,:=\,P^\bot\nabla(-\nabla\cdot P^\bot\Aa P^\bot\nabla)^{-1}\nabla\cdot P^\bot\Aa P,
\end{equation}
and let $\bar\Ac(\nabla)$ be the bounded convolution operator $\Ld^2(\R^d)^d\to\Ld^2(\R^d)^{d}$ given by
\begin{equation}\label{eq:defin-barA}
\bar\Ac(\nabla)\,:=\,\expec{\Aa(\Id+\Psi(\cdot,\nabla))}.
\end{equation}
By the stationarity, ellipticity, and boundedness assumptions~\eqref{eq:ellipt-a} for the coefficient field~$\Aa$, those operators are well-defined and satisfy the following properties:
\begin{enumerate}[---]
\item The pseudo-differential operator $\Psi(\cdot,\nabla)$ has a symbol $i\R^d\to\Ld^\infty(\R^d;\Ld^2(\Omega))^{d\times d}$,
\begin{equation}\label{eq:defin-k}
\qquad\Psi(\cdot,i\xi)\,=\,P^\bot(\nabla+i\xi)\Big(-(\nabla+i\xi)\cdot P^\bot\Aa P^\bot(\nabla+i\xi)\Big)^{-1}(\nabla+i\xi)\cdot P^\bot\Aa,
\end{equation}
where for all $\xi$ the (matrix-valued) random field $\Psi(\cdot,i\xi)$ is stationary and has vanishing expectation and finite second moments,
\[\qquad\E[\Psi(\cdot,i\xi)]\,=\,0,\qquad\E[|\Psi(\cdot,i\xi)e|^2]\,\le\,C_0^4|e|^2,\qquad\text{for all $e\in\R^d$}.\]
\item The convolution operator $\bar\Ac(\nabla)$ has a symbol $i\R^d\to\R^{d\times d}$,
\begin{equation}\label{eq:defin-barA-re}
\qquad\bar\Ac(i\xi)\,=\,\E[\Aa(\Id+\Psi(\cdot,i\xi))],
\end{equation}
where for all $\xi$ the matrix $\bar\Ac(i\xi)$ is uniformly elliptic and bounded in the sense of
\[\qquad e\cdot\bar\Ac(i\xi)e\,\ge\,\tfrac1{C_0}|e|^2,\qquad|\bar\Ac(i\xi)e|\le C_0^3|e|,\qquad\text{for all $e\in\R^d$}.\]
\end{enumerate}
In these terms, the following exact representation result holds:
\begin{enumerate}[\emph{(A$'$)}]
\item[\emph{(A$'$)}]\label{item:FS1} The ensemble average $\expec{\nabla u_{\e,f}}=\nabla\bar u_{\e,f}$ satisfies the following {well-posed} pseudo-differential equation
\begin{equation}\label{eq:homog-eqn+}
-\nabla\cdot\bar\Ac(\e\nabla)\,\nabla\bar u_{\e,f}\,=\,\nabla\cdot f.
\end{equation}
\item[\emph{(B$'$)}]\label{item:FS2} The fluctuation $\nabla u_{\e,f}-\expec{\nabla u_{\e,f}}$ has spatial oscillations on the scale $O(\e)$ and can be described as follows as a pseudo-differential operator with stationary symbol applied to the ensemble average,
\begin{equation}\label{eq:fluct+}
\nabla u_{\e,f}-\expec{\nabla u_{\e,f}}\,=\,\Psi(\tfrac\cdot\e,\e\nabla)\,\expec{\nabla u_{\e,f}}.
\end{equation}
\end{enumerate}
\end{lem}

{We view~\eqref{eq:homog-eqn+} and~\eqref{eq:fluct+} as {non-asymptotic} versions of the formal higher-order homogenized equation~\eqref{eq:homog-eqn} and of the two-scale expansion~\eqref{eq:2sc-class}, respectively.} Although this description is exact, it is complicated and of no {immediate} practical use since it involves general convolution and pseudo-differential operators.
It is therefore natural to investigate to what accuracy the convolution operator $\bar\Ac(\e\nabla)$ in~\eqref{eq:homog-eqn+} and the pseudo-differential operator $\Psi(\frac\cdot\e,\e\nabla)$ in~\eqref{eq:fluct+} can be approximated by partial differential operators as $\e\downarrow0$.
In fact, we showed in~\cite{DGL} that the regularity of the Fourier symbol~$i\xi\mapsto\bar\Ac(i\xi)$ at the origin $\xi=0$
is equivalent to the existence of a homogenized approximation for the ensemble average. We recall the following result from~\cite{DGL}, where this equivalence is made explicit: homogenized coefficients are equal to derivatives of the symbol $\bar\Ac(i\xi)$ at~$\xi=0$.

\begin{prop}[see Prop.~2.1 in~\cite{DGL}]\label{prop:equiv-DGL}
Given regularity exponents $\ell\in\N$ and $0<\eta<1$, the following two properties are equivalent:
\begin{enumerate}[\emph{(ii)}]
\item[\emph{(i)}] The symbol $i\R^d\to\R^{d\times d}:i\xi\mapsto \bar\Ac(i\xi)$
is of H\"older class $C^{\ell-\eta}$ at the origin.
\smallskip\item[\emph{(ii)}] {There exist constant tensors $\{\bar \Aa^n\}_{1\leq n\leq \ell}$,
where for all $n$ and $1\le j_1,\ldots,j_{n-1}\le d$ the value $\bar \Aa^n_{j_1\ldots j_{n-1}}$ is a matrix,
such that the following property holds. For all~$\e>0$ and~$f\in\Ld^2(\R^d)^d$, letting $u_{\e,f}\in\Ld^\infty(\Omega;\dot H^1(\R^d))$ be the unique Lax--Milgram solution of the heterogeneous elliptic equation~\eqref{eq:ellipt},
and defining $\bar u^\ell_{\e,f}\in\dot H^1(\R^d)$ as a suitable notion of solution (in the sense of Remark~\ref{rem:sol-homog-high} below) for the $\ell$th-order homogenized equation
\begin{equation}\label{eq:homog-high}
\qquad-\nabla\cdot\Big(\sum_{n=1}^\ell\bar\Aa^n_{j_1\ldots j_{n-1}}(\e\nabla)^{n-1}_{j_1\ldots j_{n-1}}\Big)\nabla\bar u_{\e,f}^\ell\,=\,\nabla\cdot f+O(\e^\ell),
\end{equation}
we have the following error bound for ensemble averages,
\begin{equation}\label{eq:estim-err-homog-aver}
\qquad\|\nabla(\E[u_{\e,f}]-\bar u^\ell_{\e,f})\|_{\Ld^2(\R^d)}\leq \e^{\ell-\eta} C_\ell \|\langle\nabla\rangle^{2\ell-1} f\|_{\mathrm{L}^2(\R^d)},
\end{equation}
for some constant $C_\ell$ depending only on $d,C_0,\ell$.}
\end{enumerate}
Moreover, if those properties hold, then the so-called homogenized coefficients $\{\bar\Aa^n\}_{1\le n\le\ell}$ are related to derivatives of the symbol of $\bar\Ac(\nabla)$: for all $1\le n\le \ell$ and $z\in\R^d$,
\begin{equation}\label{eq:expansionDGL}
\frac{\bar\Aa^n_{j_1\ldots j_{n-1}}+(\bar\Aa^n_{j_1\ldots j_{n-1}})^T }{2}z_{j_1}\ldots z_{j_{n-1}}~=~\sum_{|\alpha|=n-1}
\tfrac{z^\alpha}{\alpha!}(\nabla_{i\xi}^\alpha \bar\Ac|_{\xi=0}),
\end{equation}
an identity between symmetric matrices.
\end{prop}

\begin{rem}[Higher-order homogenized solutions]\label{rem:sol-homog-high}
We recall that the higher-order homogenized equation~\eqref{eq:homog-high} might not be well-posed as the symbol of the operator might not be positive due to the dispersive corrections. In the above statement, as in~\cite{DO1,DGL}, we can use for instance the following well-defined proxy for the higher-order homogenized solution, which solves the desired equation up to $O(\e^\ell)$: we define $\bar u_{\e,f}^\ell:=\sum_{n=1}^\ell\e^{n-1} \tilde u^n_f$, where $\tilde u_f^1$ is the solution in $\dot H^1(\R^d)$ of the (first-order) homogenized equation
\[-\nabla\cdot\bar\Aa^1\nabla \tilde u^1_f\,=\,\nabla\cdot f,\]
and where the corrections $\{\tilde u^n_f\}_{2\leq n\leq \ell}$ are iteratively defined as the solutions in~$\dot{H}^1(\R^d)$ of
\begin{equation*}
-\nabla\cdot\bar\Aa^1\nabla \tilde u^n_f
\,=\,\nabla\cdot \sum_{k=2}^n \bar\Aa^k_{j_1\ldots j_{k-1}}\nabla^{k-1}_{j_1\ldots j_{k-1}}\nabla\tilde u^{n+1-k}_f,\qquad 2\leq n\leq \ell.
\end{equation*}
\end{rem}

The proof of the above result in~\cite{DGL} is easily adapted to a similar equivalence result for fluctuations: the regularity of the Fourier symbol $i\xi\mapsto\Psi(\cdot,i\xi)$ at the origin $\xi=0$ is equivalent to the accuracy of higher-order two-scale expansions. Note that by stationarity it suffices to consider the symbol $i\R^d\to\Ld^2(\Omega)^{d\times d}:i\xi\mapsto\Psi(0,i\xi)$ at $x=0$. We omit the proof for conciseness.

\begin{prop}\label{prop:equiv-2}
Given regularity exponents $\ell\in\N$ and $0<\eta<1$, the following two properties are equivalent:
\begin{enumerate}[\emph{(ii)}]
\item[\emph{(i)}] The symbol $i\R^d\to\Ld^2(\Omega)^{d\times d}:i\xi\mapsto\Psi(0,i\xi)$ is of H\"older class $C^{\ell-\eta}$ at the origin. In particular, by definition~\eqref{eq:defin-barA-re}, this implies that the symbol $i\xi\mapsto\bar\Ac(i\xi)$ has also (at least) the same regularity.
\smallskip\item[\emph{(ii)}] There exist random fields $\{\psi^n\}_{1\le n\le\ell}$ and constant tensors $\{\bar\Aa^n\}_{1\le n\le\ell}$, where each~$\psi^n$ is a tensor-valued stationary random field with bounded second moments, \mbox{$\E[\psi^1_j]=e_j$}, and $\E[\psi^n]=0$ for $n>1$, such that the following property holds.
For all \mbox{$\e>0$} and $f\in\Ld^2(\R^d)^d$, letting $u_{\e,f}\in\Ld^\infty(\Omega;\dot H^1(\R^d))$ be the unique Lax--Milgram solution of the heterogeneous elliptic equation~\eqref{eq:ellipt}, and defining the $\ell$th-order homogenized solution~$\bar u_{\e,f}^\ell$ as in Proposition~\ref{prop:equiv-DGL}(ii),
we have the following error bound for two-scale expansions,
\[\qquad\bigg\|\nabla u_{\e,f}-\sum_{n=1}^{\ell}\e^{n-1}\psi^n_{j_1\ldots j_{n}}(\tfrac\cdot\e)\nabla_{j_1\ldots j_n}^n\bar u_{\e,f}^\ell\bigg\|_{\Ld^2(\R^d\times\Omega)}\,\le\,\e^{\ell-\eta}C_\ell\|\langle\nabla\rangle^{2\ell-1}f\|_{\Ld^2(\R^d)},\]
for some constant $C_\ell$ depending only on $d,C_0,\ell$.
\end{enumerate}
Moreover, if those properties hold, then the collection $\{\psi^n\}_{1\le n\le\ell}$ is related to derivatives of the symbol of $\Psi(\cdot,\nabla)$: for all $1\le n\le \ell$ and $x,z\in\R^d$,
\begin{equation}\label{eq:expansionDGL-bis}
\psi^n_{j_1\ldots j_{n}}(x)\,z_{j_1}\ldots z_{j_{n}}~=~z\mathds1_{n=1}+\sum_{|\alpha|=n-1}
\tfrac{z^\alpha}{\alpha!}(\nabla_{i\xi}^\alpha \Psi)(x,0).
\end{equation}
\end{prop}

In short, the above two propositions state that all standard homogenization questions are equivalent to regularity questions for the symbols
\[
\begin{array}{rclrcl}
i\R^d&\to&\R^{d\times d}:
\quad&i\xi&\mapsto&\bar\Ac(i\xi),\\[1mm]
i\R^d&\to& \Ld^2(\Omega)^{d\times d}:
\quad&i\xi&\mapsto&\Psi(0,i\xi),
\end{array}
\]
at the origin $\xi=0$.
In the periodic setting, as two-scale expansions are convergent, both symbols can be checked to be analytic in a neighborhood of the origin; see~\cite{DGL}. In contrast, in the random setting, as homogenization theory only gives in general the accuracy of two-scale expansions at best to order $O(\e^{d/2-})$, cf.~\cite{Gu-17,DO1}, we conclude that the symbol $i\xi\mapsto\Psi(0,i\xi)$ is at best of class~$C^{d/2-}$ at the origin.
As we will see below, in the weakly random regime, Bourgain proposed a new approach that goes \emph{four times} beyond this threshold for the regularity of the symbol $i\xi\mapsto\bar\Ac(i\xi)$, hence also for the accuracy of homogenized approximations for ensemble averages by Proposition~\ref{prop:equiv-DGL}.

\subsection{Bourgain's surprising result}\label{sec:Bourgain-res}

As the ensemble average $\expec{\nabla u_{\e,f}}$ is an averaged quantity, we may expect a more accurate intrinsic description to hold than for the pointwise solution field $\nabla u_{\e,f}$ itself. Equivalently, as the convolution operator $\bar\Ac(\nabla)$ is given by the expected value~\eqref{eq:defin-barA}, we may expect its symbol to have a better regularity than that of $\Psi(\cdot,\nabla)$.
As such improved results for ensemble averages could not be obtained from standard corrector theory, this direction of research was long abandoned in the homogenization community.
A recent result by Bourgain~\cite{Bourgain-18}, in its improved form obtained in~\cite{Lemm-18} by Kim and the second-named author, has shown that the above intuition is indeed correct: in the weakly random regime, the symbol of $\bar\Ac(\nabla)$ is actually four times more regular than that of $\Psi(\cdot,\nabla)$.
More precisely, the result in~\cite{Bourgain-18,Lemm-18} was obtained for an i.i.d.\@ discrete analog of the elliptic equation~\eqref{eq:ellipt}, and it can be stated as follows.\footnote{This result is only stated in~\cite{Lemm-18} for $d\ge3$, but we note that the case $d=2$ actually follows from~\cite[Theorem~1.3]{Lemm-18} up to letting the regularization parameter tend to $0$.
In addition, for~$d=1$, explicitly solving equation~\eqref{eq:ellipt} shows that the result is trivial with $\bar\Ac(\nabla)=\bar\Aa^1$.}

\begin{theor}[Bourgain~\cite{Bourgain-18}, Kim and Lemm~\cite{Lemm-18}]\label{th:bourgain}
Consider the discrete operator \mbox{$\nabla^*\cdot\Aa\nabla$} on $\ell^2(\Z^d)$, where $\nabla$ and $\nabla^*$ are forward and backward finite differences and where the coefficient field $\Aa$ on $\Z^d$ is a collection $\Aa=\{\Aa(x)\}_{x\in\Z^d}$ of i.i.d.\@ random variables.\footnote{Note that this discrete model differs from the standard random conductance model, where i.i.d.\@ coefficients would be defined on edges rather than on vertices. The proof in~\cite{Bourgain-18,Lemm-18} is easily adapted to the random conductance model, up to a slight coarse-graining argument in the spirit of Section~\ref{sect:stretchedexponentialmixingproof} of this work.}
There exists a constant $K<\infty$ (only depending on $d,C_0$) such that the following holds. If the coefficient field~$\Aa$ is close enough to a constant coefficient $\Aa_0\in\R^{d\times d}$ in the sense of
\begin{equation}\label{eq:delta-pert}
\delta\,:=\,\|\Aa-\Aa_0\|_{\Ld^\infty(\R^d\times\Omega)}\,\le\,\tfrac1K,
\end{equation}
then the convolution operator $\bar\Ac(\nabla)$ defined in Lemma~\ref{lem:Sigal} can be decomposed as
\[\bar\Ac(\nabla)\,=\,\Aa_0+\delta\,\bar\Bc(\nabla),\]
where the kernel of $\bar\Bc(\nabla)$ satisfies the following estimate for all $x,y\in\Z^d$,
\begin{equation}\label{eq:improved-dec-bourgain}
|\bar\Bc(\nabla)(x-y)|\,\le\,\delta\,\mathds1_{x=y}+ \delta^3K\langle x-y\rangle^{\delta K-3d}.
\end{equation}
In particular, this implies that the symbol of $\bar\Ac(\nabla)$ belongs to $C^{2d-\delta K-}_b(i\R^d)$.
\end{theor}

\begin{rem}
The above result actually states the {\it global} regularity of the symbol of~$\bar\Ac(\nabla)$ on~$i\R^d$, not only at the origin. Comparing it with Proposition~\ref{prop:equiv-DGL}, it amounts to the accuracy of homogenized approximations for ensemble averages around any given frequency. More precisely, given $\xi_0\in\R^d$, we can consider the solution $u_{\e,g,\xi_0}$ of equation~\eqref{eq:ellipt} with right-hand side $f$ replaced by $f_\e(x):=e^{ix\cdot\xi_0/\e}g(x)$, for some fixed $g\in\Ld^2(\R^d)^d$. In this setting, the regularity of the symbol of $\bar\Ac(\nabla)$ at $\xi_0$ amounts to the accuracy of asymptotic expansions for $\bar v_{\e,g,\xi_0}:=e^{-ix\cdot\xi_0/\e}\E[u_{\e,g,\xi_0}]$.
Such expansions could be obtained for any fixed $\xi_0$ from the standard corrector approach to homogenization, but we emphasize that Bourgain's approach allows to cover automatically all $\xi_0$'s at once in a uniform way.
This global regularity can be useful, as for instance in Remark~\ref{rem:no-frequ-cutoff} below.
\end{rem}

Theorem~\ref{th:bourgain} above constitutes a surprisingly strong improvement of classical predictions of stochastic homogenization theory: based on the existence of stationary correctors $\varphi^n$'s for all $n<\lceil\frac d2\rceil$, which is known to hold under suitable mixing assumptions, cf.~\cite{DO1}, we can only deduce that
\[\text{the symbol of $\bar\Ac(\nabla)$ is of class $C^{\frac{d}2-}$ at the origin,}\]
and similarly for $\Psi(0,\nabla)$, cf.~\cite{DGL}.
In contrast, Theorem~\ref{th:bourgain} provides a four times better regularity $C^{2d-\delta K-}$ for $\bar\Ac(\nabla)$ in the $\delta$-perturbative regime~\eqref{eq:delta-pert}.
A natural question, referred to in~\cite{DGL} as the \emph{Bourgain--Spencer conjecture}, is whether this improved regularity actually holds beyond the perturbative regime~\eqref{eq:delta-pert}, that is, independently of $\delta$.
In~\cite{D-21a}, the first-named author showed that in the general non-perturbative regime at least the following intermediate result holds (actually in some slightly weaker form):
\[\text{the symbol of $\bar\Ac(\nabla)$ is of class $C^{d-}$ at the origin.}\]
This last result is based on the construction of a larger number of so-called ``weak'' stationary correctors $\{\varphi^n\}_{n<d}$ in some distributional sense on the probability space.
Any further improvement in the non-perturbative setting, or any argument in favor of optimality, remains an open question.
We refer to Section~\ref{sec:weak-cor} for further discussion on this topic.

\section{Main results}\label{sec:mainres}
From now on, we focus on the case~$d>1$: indeed, for $d=1$, explicitly solving equation~\eqref{eq:ellipt} yields $\bar\Ac(\nabla)=\bar \Aa^1$ and $\Psi(x,\nabla)=\nabla\varphi^1(x)$, so that all our results are actually trivial in that case.

\subsection{Statistical assumptions}\label{sec:assumptions}
Given an underlying probability space $(\Omega,\Pm)$, we consider a uniformly elliptic stationary measurable random coefficient field $\Aa:\R^d\times\Omega\to\R^{d\times d}$, in the following sense:
\begin{enumerate}[---]
\item \emph{Measurability:} The map $\Aa(x,\cdot):\Omega\to\R^{d\times d}$ is measurable for all $x$, and $\Aa$ is jointly measurable on $\R^d\times\Omega$. This ensures in particular that realizations $\Aa(\cdot,\omega)$ are almost surely measurable functions on $\R^d$.
\smallskip\item \emph{Stationarity:} The finite-dimensional laws of the field $\Aa$ are shift-invariant. More precisely, for all $n\ge1$ and $x_1,\ldots,x_n\in\R^d$, the law of $(\Aa(x_1+x),\ldots,\Aa(x_1+x))$ does not depend on the shift $x\in\R^d$.
\smallskip\item \emph{Uniform ellipticity:} Almost surely, realizations of $\Aa$ are uniformly elliptic in the sense that
\[\quad\ee\cdot\Aa(x,\cdot)\ee\ge\tfrac1{C_0}|\ee|^2,\qquad|\Aa(x,\cdot)\ee|\le C_0|\ee|,\qquad\text{for all $x,\ee\in\R^d$},\]
for some $C_0<\infty$.
\end{enumerate}
On top of those general assumptions, as usual for quantitative stochastic homogenization theory, we shall need some strong mixing condition on the coefficient field~$\Aa$. More precisely, we shall consider the following two situations. On the one hand, we consider a general $\alpha$-mixing condition with stretched exponential mixing rate, in which case we will be able to recover the same regularity~\eqref{eq:improved-dec-bourgain} as in Bourgain's theorem. On the other hand, we also wish to study what the result becomes in a strongly correlated setting and for that purpose we consider a model Gaussian setting with algebraic correlation structure.
\begin{enumerate}[\textbf{(H$_1$)}]
\item[\textbf{(H$_1$)}]\label{H1} \emph{Stretched exponential $\alpha$-mixing setting:} The random field $\Aa$ is {$\alpha$-mixing with some stretched exponential rate function} in the sense that for all $U,V\subset\R^d$ and all events $A\in\sigma(\Aa|_U)$ and $B\in\sigma(\Aa|_V)$ we have
\[\qquad|\,\pr{A\cap B}-\pr{A}\pr{B}|\,\le\,C_0\exp(-\tfrac1{C_0}\dist(U,V)^\gamma),\]
for some exponent $\gamma>0$ and some constant $C_0<\infty$.

\smallskip\item[\textbf{(H$_2$)}]\label{H2} \emph{Correlated Gaussian setting:} The random field $\Aa$ is of the form
\begin{equation}\label{eq:rep-AA_0}
\Aa(x,\omega)\,=\,A_0(G(x,\omega)),
\end{equation}
where the function $A_0\in C_b^2(\R^\kappa;\R^{d\times d})$ is such that the uniform ellipticity assumption for $\Aa$ is satisfied pointwise,
and where $G:\R^d\times\Omega\to\R^\kappa$ is an $\R^\kappa$-valued stationary centered Gaussian random field on $\R^d$, characterized by its covariance function
\[c(x-y)\,:=\,\expec{G(x,\cdot)\otimes G(y,\cdot)},\qquad c:\R^d\to\R^{\kappa\times\kappa}.\]
Moreover, we assume that the covariance function has algebraic decay at infinity in the following sense: we assume that $c$ can be decomposed as $c=c_0\ast c_0$ for some even convolution kernel $c_0:\R^d\to\R^{\kappa\times\kappa}$ satisfying
\begin{equation}\label{eq:decay-c0}
|c_0(x)|\,\le\,C_0(1+|x|)^{-\beta},\qquad\beta:=\gamma\vee\tfrac{d+\gamma}2,
\end{equation}
for some exponent $\gamma>0$ and some constant $C_0<\infty$.
This implies in particular that the covariance function satisfies precisely
\[|c(x)|\lesssim_\gamma C_0^2(1+|x|)^{-\gamma}.\]
Under those assumptions, whenever $\gamma>0$, we note that the Gaussian field $G$ is necessarily strongly mixing, hence ergodic, but it is $\alpha$-mixing only provided~$\gamma>d$; see e.g.~\cite{Doukhan-94}.
\end{enumerate}

\subsection{Extensions of Bourgain's approach}
Our first main result generalizes Bourgain's approach, cf.~Theorem~\ref{th:bourgain}, to the continuum setting in case of a coefficient field with stretched exponential $\alpha$-mixing rate~{\bf(H$_1$)}.

\begin{theor}[Main result~1]\label{th:exp-mix}
Consider the stretched exponential $\alpha$-mixing setting~\emph{\bf(H$_1$)}, with some exponent $\gamma>0$ and some constant $C_0<\infty$.
There exists a constant \mbox{$K<\infty$} (only depending on~$d,\gamma,C_0$) such that the following holds. If the coefficient field~$\Aa$ is close enough to a constant coefficient $\Aa_0\in\R^{d\times d}$ in the sense of
\begin{equation}\label{eq:pertu-delta}
\delta\,:=\,\|\Aa-\Aa_0\|_{\Ld^\infty(\R^d\times\Omega)}\,\le\,\tfrac1K,
\end{equation}
then the convolution operator $\bar\Ac(\nabla)$ defined in Lemma~\ref{lem:Sigal} can be decomposed as
\begin{equation}
\label{eq:decompmain}\bar\Ac(\nabla)\,=\,\Aa_0+\delta\,\bar\Bc(\nabla),\end{equation}
where $\bar\Bc(\nabla)$ satisfies the following estimate for all $x,y\in\R^d$ and $1+\delta K<q<\tfrac1{\delta K}$,
\begin{equation}\label{eq:main1}
\|\mathds1_{Q(x)}\bar\Bc(\nabla)\mathds1_{Q(y)}\|_{\Ld^q(\R^d)^d\to\Ld^q(\R^d)^d}\,\le\,\delta K\tfrac{q^2}{q-1}\log(2+|x-y|)^K
\langle x-y\rangle^{\delta K-3d}.
\end{equation}
In particular, the symbol of $\bar\Ac(\nabla)$ belongs to $C_b^{2d-\delta K-}(i\R^d)$.
\end{theor}

Next, in order to illustrate how the decay rate is affected in case of strongly correlated coefficient fields, we focus on the model Gaussian setting~{\bf(H$_2$)}. In this case, the regularity exponent~$2d$ is replaced by $(2d)\wedge\gamma$ when the covariance function has algebraic decay of order $\gamma$. Note that this decay saturates whenever $\gamma\ge2d$.

\begin{theor}[Main result 2]\label{th:correl}
Consider the correlated Gaussian setting~\emph{\bf(H$_2$)}, with some exponent $\gamma>0$, say $\gamma\ne d$, and for some constant $C_0<\infty$. There exists a constant $K<\infty$ (only depending on~$d,\gamma,C_0$) such that the following holds. If the coefficient field~$\Aa$ is close enough to a constant coefficient $\Aa_0\in\R^{d\times d}$ in the sense that the function~$A_0$ in the representation~\eqref{eq:rep-AA_0} satisfies
\[\delta\,:=\,\|A_0-\Aa_0\|_{C^2_b(\R^\kappa)}\,\le\,\tfrac1K,\]
then the convolution operator $\bar\Ac(\nabla)$ defined in Lemma~\ref{lem:Sigal} can be decomposed as
\[\bar\Ac(\nabla)\,=\,\Aa_0+\delta\bar\Bc(\nabla),\]
where $\bar\Bc(\nabla)$ satisfies the following estimate for all $x,y\in\R^d$ and $1+\delta K<q<\tfrac1{\delta K}$,
\begin{equation}\label{eq:main2}
\|\mathds1_{Q(x)}\bar\Bc(\nabla)\mathds1_{Q(y)}\|_{\Ld^q(\R^d)^d\to\Ld^q(\R^d)^d}\,\le\,
\delta K\tfrac{q^2}{q-1}\langle x-y\rangle^{\delta K-d-(2d)\wedge\gamma}.
\end{equation}
In particular, the symbol of $\bar\Ac(\nabla)$ is of class {$C^{(2d)\wedge\gamma-\delta K-}_b(i\R^d)$}.
\end{theor}

\begin{rem}[$\delta$-dependence and transitions]
We briefly comment on the dependence on~$\delta$ in the above results: similarly as in the result~\eqref{eq:improved-dec-bourgain} obtained in the i.i.d.\@ discrete setting, we may expect the prefactor~$\delta$ in~\eqref{eq:main1} and~\eqref{eq:main2} to be somehow replaced by $\delta^3$ for $|x-y|\gg1$.
In fact, the kernel of $\bar\Bc(\nabla)$ can be decomposed into two contributions:
\begin{enumerate}[---]
\item the first contribution is of order $O(\delta)$ and has decay given by the na\"ive bound $\langle x-y\rangle^{-d}$ multiplied by another decay rate driven by mixing properties;
\smallskip\item the second contribution is of order~$O(\delta^3)$ and has the unusual decay $\langle x-y\rangle^{-3d}$ due to random cancellations.
\end{enumerate}
This decomposition is clear from the proof: in the perturbation series and the path analysis performed in the proof, the two contributions correspond respectively to so-called reducible and irreducible paths.
More precisely, the bound~\eqref{eq:main1} in Theorem~\ref{th:exp-mix} can be improved to the following, for all~$\e>0$,
\begin{equation*}
\|\mathds1_{Q(x)}\bar\Bc(\nabla)\mathds1_{Q(y)}\|_{\Ld^q(\R^d)^d\to\Ld^q(\R^d)^d}\,\le\,K\tfrac{q^2}{q-1}\Big(\delta\,\exp\big(\!-|x-y|^\frac{\e\gamma}{3d}\big)+\delta^3\langle x-y\rangle^{\e+\delta K-3d}\Big),
\end{equation*}
and similarly the bound~\eqref{eq:main2} in Theorem~\ref{th:correl} can be improved to
\begin{equation*}
\|\mathds1_{Q(x)}\bar\Bc(\nabla)\mathds1_{Q(y)}\|_{\Ld^q(\R^d)^d\to\Ld^q(\R^d)^d}\,\le\,K\tfrac{q^2}{q-1}
\Big(\delta\, \langle x-y\rangle^{\delta K-d-\gamma}+\delta^3\langle x-y\rangle^{\delta K-3d}\Big).
\end{equation*}
In this last estimate, we find that the decay of the first contribution dominates that of the second one when the decay of correlations is too slow in the sense of $\gamma<2d$. This amounts to a transition where the leading term in the perturbation series switches from the irreducible path $(x,y,x,y)$ to the reducible path $(x,y)$, and this comes with a corresponding transition in the $\delta$-dependence.
\end{rem}

\begin{rem}[$q$-dependence]
The bounds~\eqref{eq:main1} and~\eqref{eq:main2} contain two distinct pieces of information.
The main piece is the surprising decay at large distances $|x-y|\gg1$, which has important consequences as discussed below.
At short distances $|x-y|\lesssim1$, on the other hand, the bounds do not give precise information on the pointwise singularity of the kernel of~$\bar\Bc(\nabla)$, but they still contain nontrivial information: they show that the operator norm of~$\bar\Bc(\nabla)$ on~$\Ld^q$ is bounded by $O(\tfrac{q^2}{q-1})$,
which matches the bound on the operator norm of Calder\'on--Zygmund operators by Marcinkiewicz interpolation, at least for $q$ not too close to $1$ or $\infty$ in the sense of $1+\delta K<q<\frac1{\delta K}$.
\end{rem}

\begin{rem}[Sobolev regularity of the symbol]\label{rem:reg-Hs}
In the above statements, we focus on the H\"older regularity of the symbol of $\bar\Ac(\nabla)$, which is shown to belong to $C^{2d-
\delta K-}(i\R^d)$ or~$C^{(2d)\wedge\gamma-\delta K-}(i\R^d)$ in Theorems~\ref{th:exp-mix} and~\ref{th:correl}, respectively.
Alternatively, we may also investigate its Sobolev regularity: following the lines of the proof of H\"older regularity in Section~\ref{ssect:hoelder}, we can easily check that the symbol also belongs for instance to~$H^{\frac{5d}2-\delta K-}(i\R^d)$ or~$H^{\frac d2+(2d)\wedge\gamma-\delta K-}(i\R^d)$ in the setting of Theorems~\ref{th:exp-mix} and~\ref{th:correl}, respectively.
This weak differentiability can be useful, as in particular for our estimates on the averaged Green's function in Remark~\ref{rem:no-frequ-cutoff} below.
\end{rem}

\subsection{Application~1: homogenization of ensemble averages}
We briefly describe several consequences of Theorems~\ref{th:exp-mix} and~\ref{th:correl} to homogenization theory, and we start with the homogenization of ensemble averages. The following result is an immediate consequence of Theorems~\ref{th:exp-mix} and~\ref{th:correl} combined with the equivalence stated in Proposition~\ref{prop:equiv-DGL}. As a comparison, we recall that standard quenched homogenized expansions only reach the order~$\ell=\frac d2$ in the stretched exponential mixing setting~{\bf(H$_1$)}, and the order $\ell=\frac 12(d\wedge\gamma)$ in the correlated Gaussian setting~{\bf(H$_2$)} with exponent~$\gamma$, cf.~\cite{DO1}.
In particular, while in the stretched exponential mixing setting the ensemble-averaged solution allows for an homogenized approximation with an accuracy {\it four times better} than the quenched solution itself, we find that the accuracy only gets {\it two times better} in the very correlated Gaussian setting with $\gamma< d$, thus matching in that case the non-perturbative result that would be obtained with the weak corrector method of~\cite{D-21a}.

\begin{cor}[Homogenization for ensemble averages]\label{cor:applyDGL}\
\begin{enumerate}[(i)]
\item \emph{Stretched exponential $\alpha$-mixing setting:}
Under the assumptions of Theorem~\ref{th:exp-mix}, there exist constant tensors $\{\bar\Aa^n\}_{1\le n\le 2d}$ such that the following property holds for any integer $\ell\ge1$ and any $0\le\eta<1$ with
\[\ell-\eta\,<\,2d-\delta K.\]
For all $\e>0$ and~$f\in\Ld^2(\R^d)^d$, letting $u_{\e,f}\in\Ld^\infty(\Omega;\dot H^1(\R^d))$ be the unique Lax--Milgram solution of the heterogeneous elliptic equation~\eqref{eq:ellipt},
and defining $\bar u^\ell_{\e,f}\in\dot H^1(\R^d)$ as a suitable notion of solution (in the sense of Remark~\ref{rem:sol-homog-high}) for the $\ell$th-order homogenized equation
\[\qquad-\nabla\cdot\Big(\sum_{n=1}^{\ell}\bar\Aa^n_{j_1\ldots j_{n-1}}(\e\nabla)^{n-1}_{j_1\ldots j_{n-1}}\Big)\nabla\bar u_{\e,f}^{\ell}\,=\,\nabla\cdot f+O(\e^{\ell}),\]
we have the following error bound for ensemble averages,
\begin{equation*}
\qquad\|\nabla(\E[u_{\e,f}]-\bar u^\ell_{\e,f})\|_{\mathrm{L}^2(\R^d)}\leq \e^{\ell-\eta} C_\ell \|\langle\nabla\rangle^{2\ell-1} f\|_{\mathrm{L}^2(\R^d)}.
\end{equation*}
\item \emph{Correlated Gaussian setting with exponent $\gamma$:}
Under the assumptions of Theorem~\ref{th:correl}, the same result holds as in~(i) up to any order
\[\ell-\eta\,<\,(2d)\wedge\gamma-\delta K.\]
\end{enumerate}
\end{cor}

For the above homogenization result to be any useful to practitioners, it should be complemented with a practical way to actually compute numerically the constant tensors~$\{\bar\Aa^n\}_{1\le n\le\ell}$ that define the homogenized equation. According to Proposition~\ref{prop:equiv-DGL}, those can be obtained as derivatives of the symbol of~$\bar\Ac(\nabla)$ at the origin, cf.~\eqref{eq:expansionDGL}, but this description is of no much use for numerics.
To solve this issue, we show that those homogenized tensors can also be obtained as the limits of their massive approximations, which in turn are amenable to numerical computations: indeed, as e.g.~in~\cite{Gloria-Habibi-16}, we recall that massive approximations can be evaluated numerically by periodization and Monte Carlo methods. The proof of the following massive approximation result is postponed to Section~\ref{sect:massive}. The question of convergence rates is skipped for conciseness.

\begin{cor}[Massive approximation]\label{cor:massive}
For $\mu>0$, replacing the operator \mbox{$-\nabla\cdot\Aa\nabla$} by its massive version $\mu-\nabla\cdot\Aa\nabla$, all correctors and homogenized coefficients can be constructed as follows:
\begin{enumerate}[---]
\item The massive correctors $\{\varphi_\mu^n\}_{n\ge1}$ can be uniquely defined iteratively by letting $\varphi_{\mu,i_1\ldots i_n}^n$ be the stationary random field that has vanishing expectation, finite second moments, and satisfies almost surely the following massive version of the corrector equation~\eqref{eq:correctors}, in the weak sense on $\R^d$,
\begin{multline*}
\qquad(\mu-\nabla\cdot\Aa\nabla)\varphi^n_{\mu,j_1\ldots j_n}\,=\,\nabla\cdot(\Aa\varphi^{n-1}_{\mu,j_1\ldots j_{n-1}}\ee_{j_n})\\
+\ee_{j_n}\cdot P^\bot\Aa(\nabla\varphi^{n-1}_{\mu,j_1\ldots j_{n-1}}+\varphi^{n-2}_{\mu,j_1\ldots j_{n-2}}\ee_{j_{n-1}}),
\end{multline*}
with the conventions $\varphi_\mu^0\equiv1$ and $\varphi_\mu^{-1}\equiv0$.
\smallskip\item The homogenized tensors $\{\bar\Aa_\mu^n\}_{n\ge1}$ are defined by
\[\bar\Aa^n_{\mu,j_1\ldots j_{n-1}}e_{j_n}\,=\,\expecm{\Aa\big(\nabla\varphi^{n}_{\mu,j_1\ldots j_{n}}+\varphi^{n-1}_{\mu,j_1\ldots j_{n-1}}\ee_{j_{n}}\big)}.\]
\end{enumerate}
In these terms, letting $\ell$ denote the highest order obtained in Corollary~\ref{cor:applyDGL}, we have for all~\mbox{$1\le n\le\ell$},
\[\lim_{\mu\downarrow0}\Sym(\bar\Aa_{\mu}^n)\,=\,\Sym(\bar\Aa^n),\]
where we use the notation $\Sym(\bar\Aa^n)_{j_1\ldots j_n}:=\frac1{n!}\sum_{\sigma\in\Sym(n)}\frac12(\bar\Aa^n_{j_{\sigma(1)}\ldots j_{\sigma(n)}}+(\bar\Aa^n_{j_{\sigma(1)}\ldots j_{\sigma(n)}})^T)$ for tensor symmetrization.
\end{cor}

\subsection{Application~2: asymptotics of annealed Green's function}
As shown in~\cite{Lemm-Keller-21} in the discrete setting,
in addition to giving rise to expansions of ensemble averages of the solution operator,
the regularity of the symbol of $\bar\Ac(\nabla)$ can also be used to derive asymptotic expansions for the annealed Green's function $\mathcal G$. The latter is defined for $d>2$ as the tempered distribution\footnote{Note that for $d=2$, due to the integrability issue at $\xi=0$ in the integral, only the gradient of the Green's function can be defined as a tempered distribution. We focus on $d>2$ here for shortness.}
\begin{eqnarray}\label{eq:def-calG}
\Gc(x)&:=&\E\big[(-\nabla\cdot\Aa\nabla)^{-1}\big](x)\nonumber\\
&=&\big(-\nabla\cdot \bar\Ac(\nabla)\nabla\big)^{-1}(x)
\,=\,\int_{\R^d}e^{ix\cdot\xi}\,\big(\xi\cdot \bar\Ac(i\xi)\xi\big)^{-1}\,\frac{d\xi}{(2\pi)^{d/2}}.
\end{eqnarray}
Previous work on Green's functions in stochastic homogenization has been focused on $\Ld^2$-annealed estimates. Specifically, the main concern was to show that the quenched Green's function {$G(x,y):=(-\nabla\cdot\Aa\nabla)^{-1}(x,y)$} behaves similarly as the Green's function for the Laplacian up to $\Ld^2$-averaging over the random ensemble: for $|x-y|\ge1$,
\begin{eqnarray*}
\|G(x,y)\|_{\Ld^2(\Omega)}&\lesssim&|x-y|^{2-d},\\
\|\partial_x G(x,y)\|_{\Ld^2(\Omega)}&\lesssim&|x-y|^{1-d},\\
\|\partial_x\partial_yG(x,y)\|_{\Ld^2(\Omega)}&\lesssim&|x-y|^{-d}.
\end{eqnarray*}
While the first of those bounds follows from the Aronson estimates, the latter two show that the classical De Giorgi--Nash--Moser regularity theory can be substantially improved upon $\Ld^2$-averaging. This was first established in the discrete i.i.d.\@ setting by Delmotte and Deuschel~\cite{Delmotte-Deuschel-05} (see also~\cite{conlon2000green,naddaf1997homogenization}), and has been largely extended and refined since then (see e.g.~\cite{marahrens2015annealed,GMa,armstrong2019quantitative,BGO-20}). It can be viewed as a precursor of the large-scale regularity theory later developed in~\cite{AS,armstrong2019quantitative,GNO-reg}.
We also refer to~\cite{MO-16,BGO-17,BGO-20} for related two-scale expansions of the Green's function, which amount to quenched asymptotic expansions of~$G$.

In the present work, instead of those $\Ld^2$-averaged estimates on the quenched Green's function $G$, we focus on the averaged Green's function $\Gc$ defined in~\eqref{eq:def-calG},
\[\Gc(x-y)\,=\,\E[G(x,y)],\]
and we show that better estimates and expansions can be proved for the latter.
From the perspective of statistical mechanics, we recall that correlation functions for several equilibrium statistical ensembles can be written as the averaged Green's function for some elliptic operator in divergence form with random coefficients, see~\cite{naddaf1997homogenization}, thus providing further motivation for the present study.

We start with the leading-order asymptotics of the annealed Green's function: more precisely, due to the leading-order expansion
\begin{equation}\label{eq:expand-barA-leading}
-\nabla\cdot\bar\Ac(\nabla)\nabla\,=\,-\nabla\cdot\bar\Aa^1\nabla+O(|\nabla|^3),
\end{equation}
we can naturally compare $\Gc$ with the Green's function $\bar G$ associated with the homogenized operator $-\nabla\cdot\bar\Aa^1\nabla$. Note that $\bar G$ satisfies $|\nabla^\alpha \bar G(x)|\simeq_\alpha|x|^{2-d-|\alpha|}$ for all $\alpha\ge0$, so that the result below indeed identifies the leading asymptotics of $\nabla^\alpha\mathcal G$ as $|x|\uparrow\infty$. The proof is postponed to Section~\ref{sect:corproofs-avgcasreg}.
The possibility of removing the frequency cut-off $\chi$ is the object of Remark~\ref{rem:no-frequ-cutoff} below.
Note that a corresponding result holds, with the same proof, under the assumptions of Theorem~\ref{th:correl}, which we skip here for conciseness.

\begin{cor}[Leading asymptotics of annealed Green's function]\label{cor:avgcasreg}
Let $d>2$ and let the assumptions of Theorem~\ref{th:exp-mix} hold. Given a frequency cut-off $\chi$ with Fourier transform $\hat\chi\in C_c^\infty(\R^d)$, we have for all multi-indices $\alpha\ge0$ with $|\alpha|\le d$,
\begin{equation}\label{eq:AGFasymptotics}
|\chi*(\nabla^\alpha\mathcal G-\nabla^\alpha\bar G)(x)|\,\lesssim_\chi\,\langle x\rangle^{1-d-|\alpha|}.
\end{equation}
In addition, for $|\alpha|=d$, we have for all $\eta>\delta K$,
\begin{equation}\label{eq:AGFasymptotics-Cr}
[\chi\ast(\nabla^\alpha \Gc-\nabla^\alpha \bar G)]_{C^{1-\eta}(B(x))}\,\lesssim_{\chi,\eta}\,\langle x\rangle^{\eta-2d},
\end{equation}
where $[\cdot]_{C^{1-\eta}(B(x))}$ stands for the H\"older semi-norm on $B(x)$. If the coefficient field $\Aa$ is symmetric in law (that is, if $\Aa$ has the same law as the pointwise transpose~$\Aa^T$), then the different decay rates can be improved by one order: $\langle x\rangle^{1-d-|\alpha|}$ and~$\langle x\rangle^{\eta-2d}$ can be replaced by~$\langle x\rangle^{-d-|\alpha|}$ and~$\langle x\rangle^{\eta-2d-1}$ in~\eqref{eq:AGFasymptotics} and~\eqref{eq:AGFasymptotics-Cr}, respectively.
\end{cor}

Pursuing the expansion~\eqref{eq:expand-barA-leading} to higher order,
we are led to derive the following higher-order expansion of the averaged Green's function. The corrections to the homogenized Green's function $\bar G$ are expressed in a hierarchical form similarly as for higher-order homogenized equations in Remark~\ref{rem:sol-homog-high}.
By definition, the $n$th correction $\tilde G^n$ in~\eqref{eq:corr-Green-0} is an homogeneous function of degree $3-d-n$ that can be computed explicitly in terms of the constants $\{\bar\Aa^k\}_{1\le k\le n}$, so that the present result amounts to an expansion of the annealed Green's function $\Gc$ in powers of $1/|x|$.
The proof is postponed to Section~\ref{sect:corproofs-avgcasreg-2}.

\begin{cor}[Higher asymptotics of annealed Green's function]\label{cor:higherorder}
Let $d>2$ and let the assumptions of Theorem~\ref{th:exp-mix} hold.
For $1\le\ell\le2d$, we define
\begin{equation}\label{eq:corr-Green-0}
\bar G^\ell(x)\,:=\,\sum_{n=1}^\ell\tilde G^n(x),
\end{equation}
where $\tilde G^1:=\bar G$ is the Green's function for the homogenized operator $-\nabla\cdot\bar\Aa^1\nabla$,
and where the corrections $\{\tilde G^n\}_{2\le n\le 2d}$ are tempered distributions defined iteratively as follows,
\begin{equation}\label{eq:corr-Green}
-\nabla\cdot\bar\Aa^1\nabla\tilde G^n=\nabla\cdot\sum_{k=2}^n\bar\Aa^k_{j_1\ldots j_{k-1}}\nabla_{j_1\ldots j_{k-1}}^{k-1}\nabla\tilde G^{n+1-k},\qquad 2\le n\le 2d.
\end{equation}
Given a frequency cut-off $\chi$ with Fourier transform $\hat\chi\in C^\infty_c(\R^d)$,
we have for all $\ell\ge1$ and all multi-indices $\alpha\ge0$ with $|\alpha|+\ell\le d+1$,
\begin{equation}\label{eq:higherorder}
|\chi\ast(\nabla^\alpha\Gc-\nabla^\alpha\bar G^\ell)(x)|\,\lesssim\,\langle x\rangle^{2-d-|\alpha|-\ell}.
\end{equation}
In addition, for $|\alpha|+\ell=d+1$, we have for all $\eta>\delta K$,
\begin{equation}\label{eq:higherorder-re}
[\chi\ast(\nabla^\alpha\mathcal G-\nabla^\alpha\bar G^\ell)]_{C^{1-\eta}(B(x))}\,\lesssim_{\chi,\eta}\,\langle x\rangle^{\eta-2d}.
\end{equation}
\end{cor}

\begin{rem}[Multipole screening interpretation]
The annealed Green's function is the kernel for the averaged solution: for $f\in C^\infty_c(\R^d)$, considering the Lax--Milgram solution $v_f\in \Ld^\infty(\Omega;\dot H^1(\R^d))$ of the equation $-\nabla\cdot\Aa\nabla v_f=f$, the averaged solution is $\E[v_f]=\Gc\ast f$.
We claim that Corollary~\ref{cor:avgcasreg} yields the following screening effect for multipoles: for all $n\le d+1$, if $f\in C^\infty_c(\R^d)$ corresponds to a $2^{n}$-pole, meaning that $\int_{\R^d}x^mf(x)dx=0$ for all $0\le m\le n-1$, then we have the decay $|\E[v_f](x)|\lesssim_f\langle x\rangle^{2-d-n}$.
This is indeed a consequence of the proof of Corollary~\ref{cor:avgcasreg} since the condition that $f\in C^\infty_c(\R^d)$ is a $2^{n}$-pole allows to represent $f=\nabla^\alpha h$ for some $h\in C^\infty_c(\R^d)$ and $|\alpha|=n$, hence $\E[v_f]=\nabla^\alpha \Gc\ast h$.
\end{rem}

\begin{rem}[Estimates without frequency cut-off]\label{rem:no-frequ-cutoff}
We may wonder in which cases the frequency cut-off $\chi$ can be removed in Corollaries~\ref{cor:avgcasreg} and~\ref{cor:higherorder}.
Note that in the corresponding discrete setting of~\cite{Lemm-Keller-21,uchiyama1998green} this cut-off does naturally not occur.
In the continuum setting, removing the cut-off requires information on the high-frequency behavior of $\bar\Ac(\nabla)$, and this question is entirely foreign to homogenization.
We can show for instance the following results, the proofs of which are postponed to Section~\ref{sect:corproofs-cutoff}:
\begin{enumerate}[(a)]
\item For the homogenized Green's function $\bar G$ or its corrections $\bar G^\ell$, the cut-off $\chi$ can always be removed: for all $1\le\ell\le2d$ and all multi-indices $\alpha\ge0$, we indeed have for all $|x|\ge1$ and $p\ge0$,
\[|\chi\ast\nabla^\alpha\bar G(x)-\nabla^\alpha\bar G|
+|\chi\ast\nabla^\alpha\bar G^\ell(x)-\nabla^\alpha\bar G^\ell|
\,\lesssim_{\chi,\alpha,p}\,|x|^{-p}.\]
\item For the annealed Green's function $\mathcal G$, the cut-off $\chi$ can be removed when less than two derivatives are taken: for $|\alpha|\le1$, we have for all $|x|\ge1$,
\begin{equation*}
|\chi\ast\nabla^\alpha\Gc(x)-\nabla^\alpha\Gc(x)|\,\lesssim_{\chi}\,|x|^{-d-1},
\end{equation*}
and in addition, for $|\alpha|=1$, for all $|x|\ge2$,
\begin{equation*}
[\chi\ast\nabla^\alpha\Gc-\nabla^\alpha\Gc]_{C^{1-\eta}(B(x))}\,\lesssim_{\chi,\eta}\,\left\{\begin{array}{lll}
|x|^{-d-1}&:&\eta>\frac{5d}{6d+6}+C\delta K,\\
|x|^{-d}&:&\eta>0.
\end{array}\right.
\end{equation*}
\item If $\Aa$ is rotationally symmetric in law (that is, if $\Aa$ has the same law as $O^T\Aa(O\cdot)O$ for all~$O\in O(d)$), then for the annealed Green's function $\mathcal G$ the cut-off $\chi$ can be removed when less than $\frac{d+3}{2}$ derivatives are taken: for $|\alpha|<\frac{d+3}2$, we have for all~$|x|\ge1$,
\[|\chi\ast\nabla^\alpha\Gc(x)-\nabla^\alpha\Gc(x)|\,\lesssim_{\chi}\,|x|^{2-d-\frac{d+3}2}.\]
\end{enumerate}
\end{rem}

\subsection{Weak correctors}\label{sec:weak-cor}
We turn to our last main result, that is, the construction of weak correctors, or alternatively, the regularity of the symbol of $\Psi(\cdot,\nabla)$ in a suitable weak sense.
As recalled in Section~\ref{sec:standard-hom}, correctors $\{\varphi^n\}_n$ are defined to describe the fine spatial oscillations of the solution field $\nabla u_{\e,f}$, and, if well-defined, they are necessarily given by the recurrence equations~\eqref{eq:correctors}.
Heuristically, focusing on the first right-hand side term in those corrector equations, and replacing the heterogeneous operator $-\nabla\cdot\Aa\nabla$ by the Laplacian, one may expect in a perturbative perspective that the $n$th corrector $\varphi^n$ behaves similarly as
\begin{equation}\label{eq:tildephin}
\tilde\varphi^n\,:=\,[(-\triangle)^{-1}\nabla]^nP^\bot\Aa.
\end{equation}
Due to the iteration of the Riesz potential $(-\triangle)^{-1}\nabla$, such quantities can only make sense for $n$ not too large. More precisely, under strong enough mixing assumptions, such as~{\bf(H$_1$)}, we have the following easy observations:
\begin{enumerate}[---]
\item $\tilde\varphi^n$ can be uniquely constructed as a well-defined stationary random field in $\Ld^2(\Omega)$ provided that~$n<\frac d2$;
\smallskip\item $\tilde\varphi^n$ can be uniquely constructed as a well-defined stationary random field in a Schwartz-like distributional sense on $\Omega$ (in the sense of~\cite{D-21a}) provided that~$n<d$. More precisely, this means that $\tilde\varphi^n$ can be defined through
its conditional expectations $\expec{\tilde\varphi^n\|\Aa|_{B_R}}\in\Ld^2(\Omega)$ for all $R>0$.
\end{enumerate}
The construction of genuine correctors~$\{\varphi^n\}_n$ satisfying the corrector equations~\eqref{eq:correctors} instead of~\eqref{eq:tildephin} is much more intricate, but it is now well-understood that the same result holds: under~{\bf(H$_1$)}, the $n$th corrector $\varphi^n$ can be constructed as a stationary random field in~$\Ld^2(\Omega)$ for all $n<\frac d2$ (see e.g.~\cite{DO1}), and as a stationary random field in a distributional sense on $\Omega$ for all $n<d$ (see~\cite{D-21a}).
The latter notion of so-called {\it weak} correctors is unconventional and was first introduced in~\cite{D-21a} by the first-named author in link with a non-perturbative approach to Bourgain's result.
More precisely, we recall that the existence of weak stationary correctors of orders $n<d$ essentially implies the following intermediate result (actually in some slightly weaker form):
\begin{equation}\label{eq:intermed-dec}
\text{the symbol of $\bar\Ac(\nabla)$ is of class $C^{d-}$ at the origin},
\end{equation}
which is to be compare to the $C^{2d-\delta K-}$ regularity
obtained in the $\delta$-perturbative setting in~Theorem~\ref{th:exp-mix}.
A natural question is therefore whether the above number of weak stationary correctors is optimal or not: if weak stationary correctors could be constructed to all orders $n<2d$, then we might expect the weak corrector theory developed in~\cite{D-21a} to improve on~\eqref{eq:intermed-dec} and to solve the Bourgain--Spencer conjecture.

Intuitively, constructing a larger number of weak stationary correctors seems unlikely as it would mean that the analogy~\eqref{eq:tildephin} actually fails --- but this is indeed so: it turns out that this heuristic comparison~\eqref{eq:tildephin} does not accurately take into 
account the subtle cancellations and algebraic properties of the hierarchy of corrector equations~\eqref{eq:correctors}.
To see this, we propose a different way to construct correctors:
if the symbol of~$\Psi(\cdot,\nabla)$ could be differentiated $k$ times in some sense, then it could be applied to a polynomial of order $k$, say $p_{\ee,x}^k(z):=(e\cdot(z-x))^k$, which would yield, in view of~\eqref{eq:expansionDGL-bis},
\[\big(\Psi(\cdot,\nabla)p_{\ee,x}^n\big)(z)\,=\,\sum_{k=0}^{n}\frac{n!}{(n-k)!}\,\big(\nabla\varphi^{k+1}_{\ee\ldots\ee}(z)+\mathds1_{k>0}\varphi^k_{\ee\ldots\ee}(z)\ee\big)\,(\ee\cdot(z-x))^{n-k},\]
hence, when evaluated at $z=x$,
\[\nabla\varphi^{n+1}_{\ee\ldots\ee}(x)+\mathds1_{n>0}\varphi^n_{\ee\ldots\ee}(x)\ee\,=\,\tfrac1{n!}\big(\Psi(\cdot,\nabla)p_{\ee,x}^n\big)(x).\]
Similarly, recalling that $\Psi(\cdot,\nabla)$ is given by~\eqref{eq:defin-K}, the following formula can be checked to define the $n$th-order stationary corrector, provided that it makes sense,
\begin{equation}\label{eq:def-varphin}
\varphi^n_{e\ldots e}(x)\,=\,\tfrac1{n!}\Big((-\nabla\cdot P^\bot\Aa P^\bot\nabla)^{-1}\nabla\cdot P^\bot\Aa P\nabla p_{\ee,x}^n\Big)(x).
\end{equation}
Surprisingly, in the weakly random regime, the following result shows that weak stationary correctors can be constructed in this way for all $n<2d$. The proof is postponed to Section~\ref{sect:weakcorrectors}.
We focus here for shortness on the construction of weak correctors, but this result could be alternatively formulated in terms of the weak regularity of the symbol of $\Psi(\cdot,\nabla)$. Recalling~\eqref{eq:defin-barA}, this constitutes a generalization of the regularity of the symbol of $\bar\Ac(\nabla)$ obtained in Theorem~\ref{th:exp-mix}.
Note that a corresponding result holds, with the same proof, under the assumptions of Theorem~\ref{th:correl}, which we skip here for conciseness.

\begin{theor}[Main result 3]\label{th:weak-cor}
Under the assumptions of Theorem~\ref{th:exp-mix},
we can define weak stationary correctors $\varphi^n_{e\ldots e}$ for all $n<2d$ via the above formula~\eqref{eq:def-varphin} in the distributional sense on $\Omega$ (in the sense of~\cite{D-21a}).
More precisely, the following estimates hold
for all $n<2d$, $x\in\R^d$, and $R>0$,
\[\big\|\expecm{\varphi^n_{e\ldots e}(x)\|\Aa|_{B_R}}\big\|_{\Ld^\infty(\Omega)}\,\le\, \delta KR^{n}\big(1+\langle\tfrac{x}R\rangle^{\delta K+n-d}\big).\]
\end{theor}

The above estimates show that for $n\ge d$ the weak stationary corrector $\varphi^n_{e\ldots e}$ starts to display a polynomial growth with respect to the distance $\dist(x,B_R)$ between the evaluation point~$x$ and the conditioning set $B_R$.
This growth, which we believe to be optimal in this perturbative setting, makes correctors hard to exploit. It explains why the weak corrector theory developed in~\cite{D-21a} could only make use of weak stationary correctors up to order~$n<d$ and therefore failed at implying the full decay of Bourgain's perturbative theorem.
This tends to show that correctors could not be used to prove Bourgain's theorem in the non-perturbative setting beyond the intermediate decay~\eqref{eq:intermed-dec} obtained in~\cite{D-21a}.

\section{Deterministic estimates}
\label{sect:detest}
We consider the perturbative regime~\eqref{eq:pertu-delta}.
For notational simplicity, let us assume that~$\E[\Aa]$ is symmetric.
Without loss of generality, up to changing variables, we may then assume that $\Aa$ is close to the constant coefficient $\Aa_0=\Id$. 
More precisely, we may assume that the coefficient field~$\Aa$ takes the form
\begin{equation}\label{eq:decomp-AB}
\Aa\,=\,\Id+\delta \Bb,\qquad\delta\ll1,
\end{equation}
for some random field $\Bb$ satisfying
\[|\Bb|\le1,\qquad\expec{\Bb}=0.\]
This convenient reduction is indeed allowed without loss of generality up to replacing the coefficient field~$\Aa$ by $\expec{\Aa}^{-1/2}\Aa\expec{\Aa}^{-1/2}$, which amounts to changing variables in equation~\eqref{eq:ellipt}.
In this setting~\eqref{eq:decomp-AB}, we shall analyze the decay of the kernel of the convolution operator~$\bar\Ac(\nabla)$ defined in Lemma~\ref{lem:Sigal}, and we start in this section with some preliminary deterministic estimates and useful notation.

\subsection{Representation formula}
In terms of~\eqref{eq:decomp-AB}, the representation formulae~\eqref{eq:defin-K}--\eqref{eq:defin-barA} for the convolution operator~$\bar\Ac(\nabla)$ takes on the following guise,
\begin{equation*}
\bar\Ac(\nabla)\,=\,\Id+\delta^2P\Bb P^\bot\nabla\big(-\triangle-\delta\nabla\cdot P^\bot\Bb P^\bot\nabla\big)^{-1}\nabla\cdot P^\bot\Bb P.
\end{equation*}
Using Neumann series, this yields
\begin{equation}\label{eq:barA}
\bar\Ac(\nabla)\,=\,\Id+\delta\,\bar\Bc(\nabla),\qquad\bar\Bc(\nabla)\,:=\,\sum_{n=1}^\infty (-\delta)^n P\Bb (\Kk P^\bot\Bb)^{n} P,
\end{equation}
in terms of the singular convolution operator
\begin{equation}\label{eq:defKk}
\Kk\,:=\,\nabla\triangle^{-1}\nabla\cdot~:~\Ld^2(\R^d)^d\to\Ld^2(\R^d)^d.
\end{equation}
As we have
\[\|\Kk\|_{\Ld^2(\R^d)^d\to\Ld^2(\R^d)^d}=1,\]
the Neumann series~\eqref{eq:barA} is obviously convergent on $\Bc(\Ld^2(\R^d)^d)$ provided that $\delta<1$. Note that this calculation for $\bar\Ac(\nabla)$ holds in any space dimension.

\subsection{Singular convolution operator $\Kk$}
The following provides an explicit description of the singular convolution operator $\Kk$ defined in~\eqref{eq:defKk} above.

\begin{lem}\label{lem:CZ}
The kernel of the operator $\Kk$ on $\Ld^2(\R^d)^d$ can be decomposed as
\[\Kk(x)\,=\,\tfrac1d\delta(x)\Id\,+\,\Ll(x),\]
where the second term is given by the Calder\'on--Zygmund kernel
\[\textstyle\Ll(x)\,:=\,-\pv\tfrac1{|B|}|x|^{-d}\big(\frac{x\otimes x}{|x|^2}-\tfrac1d\Id\big).\]
In particular, for all $1<q<\infty$, by the Calder\'on--Zygmund theory, the operator $\Kk$ is bounded on~$\Ld^q(\R^d)^d$ with operator norm
\[\|\Kk\|_{\Ld^q(\R^d)^d\to\Ld^q(\R^d)^d}\,\le\,\tfrac{Cq^2}{q-1}.\]
\end{lem}

\begin{proof}
The kernel of $\triangle$ is given by the Green's function $G(x)=|\partial B|^{-1}(2-d)^{-1}|x|^{2-d}$ (and $G(x)=|\partial B|^{-1}\log|x|$ in case $d=2$), so that the kernel of $\Kk$ is obtained by integrating by parts in the representation $\Kk g(x)=\int_{\R^d}G(y)\nabla^2g(x-y)dy$ for $g\in C^\infty_c(\R^d)$. Alternatively, we may also argue by decomposing the symbol of the operator $\Kk$ as
\[\tfrac{\xi\otimes\xi}{|\xi|^2}=\tfrac1d\Id+\tfrac{1}{|\xi|^2}\big(\xi\otimes\xi-\tfrac1d|\xi|^2\Id\big),\]
and by appealing to the general result of~\cite[Theorem~3.5]{Stein-70} to deduce the kernel.
Next, the Calder\'on--Zygmund theory ensures that $\Ll$ satisfies a weak-$\Ld^1(\R^d)$ bound, on top of being bounded on $\Ld^2(\R^d)$, hence the claimed bound on $\Ld^q(\R^d)$ follows by Marcinkiewicz interpolation for~$1<q\leq2$,
and by duality for~$2\leq q<\infty$.
\end{proof}

The Calder\'on--Zygmund theory further yields the following approximation result. We shall use it occasionally to express operators as limits of absolutely converging integrals, which can then be manipulated more freely. The proof is standard and is skipped for shortness.

\begin{lem}\label{lem:CZ-eta}
In terms of the truncated kernels
\[\Kk^{(\eta)}(x)\,:=\,\tfrac1d\delta(x)\Id+\Ll(x)\mathds1_{|x|>\eta},\]
we have for all $1<q<\infty$ and $g\in\Ld^q(\R^d)^d$,
\[\lim_{\eta\downarrow0}\|\Kk^{(\eta)} g-\Kk g\|_{\Ld^q(\R^d)}=0,\qquad\limsup_{\eta\downarrow0}\|\Kk^{(\eta)}\|_{\Ld^q(\R^d)^d\to\Ld^q(\R^d)^d}\le\tfrac{Cq^2}{q-1}.\]
\end{lem}

\subsection{Mixed Lebesgue spaces}
It is often convenient to distinguish between the behavior of singular kernels on small and large scales. For that purpose, for all $1\le p,q<\infty$, we shall consider the mixed Lebesgue space $\Ld^p_q(\R^d)$ as the closure of $C^\infty_c(\R^d)$ for the norm
\[\|g\|_{\Ld^p_q(\R^d)}\,:=\,\bigg(\sum_{z\in\Z^d}\Big(\int_{Q(z)}|g|^q\Big)^\frac pq\bigg)^\frac1p.\]
Note that $\Ld^p_p(\R^d)=\Ld^p(\R^d)$ and that the norm is equivalent to
\[\|g\|_{\Ld^p_q(\R^d)}\,\simeq\,\bigg(\int_{\R^d}\Big(\int_{Q(x)}|g|^q\Big)^\frac pqdx\bigg)^\frac1p.\]
For all $1\le p,q,r<\infty$, we further consider the mixed Lebesque space $\Ld^p_q(\R^d;\Ld^r(\Omega))$ as the closure of $C^\infty_c(\R^d;\Ld^\infty(\Omega))$ for the norm
\[\|g\|_{\Ld^p_{q}(\R^d;\Ld^r(\Omega))}\,:=\,\bigg(\sum_{z\in\Z^d}\Big(\int_{Q(z)}\expec{|g|^r}^\frac qr\Big)^\frac pq\bigg)^\frac1p.\]
Note that those definitions can be naturally adapted also to possibly infinite exponents {$1\le p,q,r\le\infty$}.
For $1<p,q,r<\infty$, the singular convolution operator $\Kk$ on $\Ld^p(\R^d)$ can be extended as an operator on the mixed space $\Ld^p_{q}(\R^d;\Ld^r(\Omega))$,
and we show that the same operator bounds hold as in Lemma~\ref{lem:CZ} in this mixed setting.

\begin{lem}\label{lem:mixed-K}
For all $1<p\le q\le r\le2$, the singular convolution operator $\Kk$ on $\Ld^p(\R^d)^d$ extends as a bounded operator on $\Ld^p_{q}(\R^d;\Ld^r(\Omega))^d$ with
\begin{equation*}
\|\Kk\|_{\Ld^p_{q}(\R^d;\Ld^r(\Omega))^d\to\Ld^p_{q}(\R^d;\Ld^r(\Omega))^d}\,\le\,\tfrac{C}{p-1}.
\end{equation*}
Note that the same result holds in particular on $\Ld^p_{q}(\R^d)^d$.
\end{lem}

\begin{proof}
We split the proof into two steps.

\medskip
\step1 Proof that for all $1<p\le r\le 2$ we have
\[\|\Kk\|_{\Ld^p(\R^d;\Ld^r(\Omega))^d\to\Ld^p(\R^d;\Ld^r(\Omega))^d}\,\le\,\tfrac{C}{p-1}.\]
Starting point is Lemma~\ref{lem:CZ}, which yields by Fubini's theorem, for all $1<r\le2$ and $g\in\Ld^r(\R^d;\Ld^r(\Omega))^d$,
\[\|\Kk g\|_{\Ld^r(\R^d;\Ld^r(\Omega))}\,\le\,\tfrac{C}{r-1}\|g\|_{\Ld^r(\R^d;\Ld^r(\Omega))}.\]
As $\Kk$ is given by a Calder\'on--Zygmund kernel on $\R^d$, cf.~Lemma~\ref{lem:CZ}, we may then appeal to~\cite{Benedek-Calderon-Panzone-62} and deduce the following Banach-valued weak-$\Ld^1$ estimate,
\[|\{x:\|\Kk g\|_{\Ld^r(\Omega)}>\lambda\}|\,\le\,\tfrac{C}{r-1}\lambda^{-1}\|g\|_{\Ld^1(\R^d;\Ld^r(\Omega))}.\]
By Marcinkiewicz interpolation, this yields the claim.

\medskip
\step2 Conclusion.\\
Consider the trivial extension of the convolution operator $\Kk$ to $\Ld^p(\R^d;\Ld^q(Q;\Ld^r(\Omega)))^d$,
and further consider the operator $T$ on that space given by
\[Th(x,\alpha,\omega)\,:=\,\int_Qh(x+\alpha-\alpha',\alpha',\omega)\,d\alpha'.\]
For all $1<q\le r\le2$, integrating over $Q$, the result of Step~1 yields
\begin{eqnarray*}
\|\Kk Th\|_{\Ld^q(\R^d;\Ld^q(Q;\Ld^r(\Omega)))}
~=~\|\Kk Th\|_{\Ld^q(Q;\Ld^q(\R^d;\Ld^r(\Omega)))}
&\le&\tfrac{C}{q-1}\|Th\|_{\Ld^q(Q;\Ld^q(\R^d;\Ld^r(\Omega)))}\\
&\le&\tfrac{C}{q-1}\|h\|_{\Ld^q(\R^d;\Ld^q(Q;\Ld^r(\Omega)))}.
\end{eqnarray*}
As in Step~1, we may then appeal to~\cite{Benedek-Calderon-Panzone-62} and deduce the following weak-$\Ld^1$ estimate,
\[|\{x:\|\Kk Th\|_{\Ld^q(Q;\Ld^r(\Omega))}>\lambda\}|\,\le\,\tfrac{C}{q-1}\lambda^{-1}\|h\|_{\Ld^1(\R^d;\Ld^q(Q;\Ld^r(\Omega)))},\]
which implies, by Marcinkiewicz interpolation, for all $1<p\le q\le r\le2$,
\begin{equation*}
\|\Kk Th\|_{\Ld^p(\R^d;\Ld^q(Q;\Ld^r(\Omega)))}\,\le\,\tfrac{C}{p-1}\|h\|_{\Ld^p(\R^d;\Ld^q(Q;\Ld^r(\Omega)))}.
\end{equation*}
Considering the linear map $U:\Ld^p_q(\R^d;\Ld^r(\Omega))^d\to\Ld^p(\R^d;\Ld^q(Q;\Ld^r(\Omega)))^d$ given by
\[Ug(x,\alpha,\omega)\,:=\,g(x+\alpha,\omega),\]
noting that $\Kk TUg(x,\alpha,\omega)=(\Kk g)(x+\alpha,\omega)$, and applying the above with $h=Ug$, the conclusion follows.
\end{proof}

\subsection{Discrete truncations of $\Kk$}
In order to tackle the perturbative expansion~\eqref{eq:barA}, it will be convenient to discretize the space $\R^d$ at a given scale $R>0$. More precisely, for all~$x\in\R^d$, we define~$z_R(x)\in R\Z^d$ as the lattice point satisfying
\begin{equation}\label{eq:zRdef}
|x-z_R(x)|_\infty\,=\,\min_{z\in R\Z^d}|x-z|_\infty,
\end{equation}
where we take $z_R(x)$ to be the smallest with respect to lexicographic order in case the minimum is attained by several lattice points. In these terms, for all $\ell \geq R$, we define the following truncated kernel,
\begin{equation}\label{eq:truncate-K}
\Kk_{\ell;R}(x,y)\,:=\,\Kk(x-y)\mathds1_{|z_R(x)-z_R(y)|_\infty\le\ell}.
\end{equation}
Of course, we note that $\Kk_{\ell;R}$ is no longer a convolution kernel, but we show that it still satisfies the same operator estimates.

\begin{lem}\label{lem:trunc-CZ}
For all $1<p\le q\le r\le2$, uniformly with respect to $\ell\ge R>0$, we have
\[\|\Kk_{\ell;R}\|_{\Ld^p_q(\R^d;\Ld^r(\Omega))^d\to\Ld^p_q(\R^d;\Ld^r(\Omega))^d}\,\le\,\tfrac{C}{p-1}.\]
\end{lem}

\begin{proof}
In view of the proof of Lemma~\ref{lem:mixed-K} above, it suffices to prove the result on $\Ld^p(\R^d)^d$.
First define the radially-truncated kernel
\[\tilde\Kk_\ell(x)\,:=\,\Kk(x)\mathds1_{|x|\le\ell}\,=\,\tfrac1d\delta(x)\Id+\tilde\Ll_\ell(x),\qquad \tilde\Ll_\ell(x)\,:=\,\Ll(x)\mathds1_{|x|\le\ell}.\]
We check that $\tilde\Ll_\ell$ satisfies the following properties, uniformly with respect to $\ell>0$,
\begin{enumerate}[(a)]
\item $|\tilde\Ll_\ell(x)|\lesssim|x|^{-d}$ for all $x$;
\item $\int_{a<|x|<b}\tilde\Ll_\ell(x)dx=0$ for all $b>a>0$;
\item $\int_{x:|x|>2|y|}|\tilde\Ll_\ell(x)-\tilde\Ll_\ell(x-y)|\,dx\,\lesssim\,1$ for all $y$.
\end{enumerate}
While properties~(a) and~(b) are obvious from Lemma~\ref{lem:CZ}, property~(c) is obtained as follows: noting that the inequalities $|x|>2|y|$ and $|x-y|>\ell$ imply $|x|>\frac23\ell$, and that the inequalities $|x|>2|y|$ and $|x|>\ell$ imply $|x-y|>\frac12\ell$, we can estimate
\begin{eqnarray*}
\lefteqn{\int_{x:|x|>2|y|}|\tilde\Ll_\ell(x)-\tilde\Ll_\ell(x-y)|\,dx}\\
&\le&\int_{x:|x|>2|y|}|\Ll(x)-\Ll(x-y)|\,dx
+2\int_{x:\frac12\ell<|x|\le\ell}|\Ll(x)|\,dx,
\end{eqnarray*}
from which property~(c) follows as $\Ll$ is itself a Calder\'on--Zygmund kernel.
From~(a)--(c), the Calder\'on-Zygmund theory can be applied as in the proof of Lemma~\ref{lem:CZ}, which entails that $\tilde\Kk_\ell$ satisfies the following $\Ld^p$ estimates: uniformly with respect to $\ell >0$, we have for all $1<p<\infty$,
\begin{equation}\label{eq:CZest-Kr}
\|\tilde\Kk_\ell\|_{\Ld^p(\R^d)^d\to\Ld^p(\R^d)^d}\,\le\,\tfrac{Cp^2}{p-1}.
\end{equation}
It remains to compare $\Kk_{\ell;R}$ to $\tilde\Kk_\ell$. Noting that for $\ell\ge R$ we have
\[\overline{B_{s_1(\ell)}(x)}\,\subset\,\big\{y:|z_R(x)-z_R(y)|_\infty\le\ell\big\}\, \subset\, \overline{B_{s_2(\ell)}(x)},\]
with $s_1(\ell), s_2(\ell)\simeq \ell$,
we find for all $g\in\Ld^p(\R^d)^d$,
\begin{eqnarray*}
\|(\Kk_{\ell;R}-\tilde\Kk_\ell)g\|_{\Ld^p(\R^d)}^p&\lesssim&\int_{\R^d}\bigg(\int_{\R^d}|x-y|^{-d}|g(y)|\,\mathds1_{s_1(\ell)\le |x-y|\le s_2(\ell)}\,dy\bigg)^p\,dx\\
&\lesssim&\ell^{-dp}\int_{\R^d}\bigg(\int_{\R^d}|g(y)|\,\mathds1_{s_1(\ell)\le |x-y|\le s_2(\ell)}\,dy\bigg)^p\,dx\\
&\lesssim&\|g\|_{\Ld^p(\R^d)}^p.
\end{eqnarray*}
Combined with~\eqref{eq:CZest-Kr}, this yields the conclusion. 
\end{proof}

\subsection{A useful operator notation}
The following notation is particularly convenient in the continuum setting as it
allows to stick very closely to Bourgain's discrete notation in~\cite{Bourgain-18} while singling out local regularity issues: given $1\le q\le\infty$, we define the locally averaged kernel of a bounded operator $T$ on $\Ld^q(\R^d)^d$ as
\begin{gather*}
\cro{T}_{q}:\R^d\times\R^d\to\R^+,\\
\cro{T}_{q}(x,y)\,:=\,\|\mathds1_{Q(x)}T\mathds1_{Q(y)}\|_{\Ld^q(\R^d)^d\to\Ld^q(\R^d)^d}.
\end{gather*}
It will occasionally be useful to further include spatial averaging with a different power over some intermediate scale~$R\ge1$:
given $1\le p,q\le\infty$, we define the $R$-locally averaged kernel of a bounded operator $T$ on $\Ld^p_q(\R^d)^d$ as
\begin{gather*}
\cro{T}_{p,q;R}:R\Z^d\times R\Z^d\to\R^+,\\
\cro{T}_{p,q;R}(x,y)\,:=\,\|\mathds1_{Q_R(x)}T\mathds1_{Q_R(y)}\|_{\Ld^p_q(\R^d)^d\to\Ld^p_q(\R^d)^d}.
\end{gather*}
Note the useful lower bound,
\begin{equation}\label{eq:cro-lower}
\cro{T}_{p,q;R}(x,y)
\,\ge\,\cro{T}_{p,q;1}(x,y)
\,\simeq\,\cro{T}_{q}(x,y).
\end{equation}
(To see the latter equivalence, note that $\cro{T}_{p,q;1}(x,y)$ depends only on local integrability properties and that for these all $\Ld^p_q(\R^d)$ norms are equivalent for different values of $p$.)
As we shall often consider operators acting on random fields, that is, on functions defined on~$\R^d\times\Omega$, we further define corresponding notions of averaged kernels: given $1\le q,r\le\infty$,
for a bounded operator $T$ on $\Ld^q(\R^d;\Ld^r(\Omega))^d$, we define
\begin{gather*}
\cro{T}_{q,r}:\R^d\times\R^d\to\R^+,\\
\cro{T}_{q,r}(x,y)\,:=\,\|\mathds1_{Q(x)}T\mathds1_{Q(y)}\|_{\Ld^q(\R^d;\Ld^r(\Omega))^d\to\Ld^q(\R^d;\Ld^r(\Omega))^d},
\end{gather*}
and given $1\le p,q,r\le\infty$
and $R\ge1$, for a bounded operator $T$ on $\Ld^p_q(\R^d;\Ld^r(\Omega))^d$, we define
\begin{gather*}
\cro{T}_{p,q,r;R}:R\Z^d\times R\Z^d\to\R^+,\\
\cro{T}_{p,q,r;R}(x,y)\,:=\,\|\mathds1_{Q_R(x)}T\mathds1_{Q_R(y)}\|_{\Ld^p_q(\R^d;\Ld^r(\Omega))^d\to\Ld^p_q(\R^d;\Ld^r(\Omega))^d}.
\end{gather*}
For $p,q,r<\infty$, as the dual space of $\Ld^{p}_{q}(\R^d;\Ld^{r}(\Omega))^d$ is clearly isomorphic to $\Ld^{p'}_{q'}(\R^d;\Ld^{r'}(\Omega))^d$, the above definition can be equivalently written as
\begin{multline}\label{eq:def-brack}
\cro{T}_{p,q,r;R}(x,y)\,=\,\sup\Big\{|\langle\phi_x,T\phi_y\rangle|~:~\phi_x,\phi_y\in C^\infty_c(\R^d;\Ld^\infty(\Omega))^d,\\
\supp(\phi_x)\subset Q_R(x)\times\Omega,~\supp(\phi_y)\subset Q_R(y)\times\Omega,\\[-1mm]
\|\phi_x\|_{\Ld^{p'}_{q'}(\R^d;\Ld^{r'}(\Omega))}=\|\phi_y\|_{\Ld^p_q(\R^d;\Ld^r(\Omega))}=1\Big\},
\end{multline}
where henceforth we use the short-hand notation $\langle\cdot,\cdot\rangle$ for the pairing of vector fields on $\R^d\times\Omega$, that is,
\[\langle\phi_x,T\phi_y\rangle:=\expec{\int_{\R^d}\phi_x\cdot T\phi_y}.\]
Given $1\le p,q,r\le\infty$,
for a bounded operator $T$ on $\Ld^r(\Omega)^d$, extended trivially as an operator on $\Ld^p_q(\R^d;\Ld^r(\Omega))^d$, note that we have $\cro{T}_{p,q,r;R}(x,y)\le\mathds1_{x=y}\|T\|_{\Ld^r(\Omega)^d\to\Ld^r(\Omega)^d}$ for all~$x,y\in R\Z^d$.
Moreover, given two operators $T,T'$ on $\Ld^p_q(\R^d;\Ld^r(\Omega))$, we have
\begin{eqnarray}
\cro{T}_{p,q,r;R}(x,y)&\le&\|T\|_{\Ld^p_q(\R^d;\Ld^r(\Omega))^d\to\Ld^p_q(\R^d;\Ld^r(\Omega))^d},\nonumber\\[1mm]
\cro{TT'}_{p,q,r;R}(x,y)&\le&\sum_{z\in R\Z^d}\cro{T}_{p,q,r;R}(x,z)\,\cro{T'}_{p,q,r;R}(z,y).\label{eq:compo-Rp}
\end{eqnarray}
Finally, we show that the following estimate holds for the truncated kernel~$\Kk_{\ell;R}$.

\begin{lem}\label{lem:decay-K}
For all $1<p\le q\le r\le2$, uniformly with respect to $\ell\ge R\ge1$, we have
\[\cro{\Kk_{\ell;R}}_{p,q,r;R}(x,y)\,\lesssim\,\tfrac{1}{p-1}\wedge (|x-y|-2R)_+^{-d}\,\lesssim\,\tfrac{1}{p-1}\langle\tfrac1R(x-y)\rangle^{-d}.\]
\end{lem}

\begin{proof}
For $|x-y|_\infty>2R$, the pointwise decay of the kernel yields
\[\cro{\Kk_{\ell;R}}_{p,q,r;R}(x,y)\,\lesssim\,|x-y|^{-d}.\]
As Lemma~\ref{lem:trunc-CZ} further yields
\[\cro{\Kk_{\ell;R}}_{p,q,r;R}(x,y)\,\le\,\|\Kk_{\ell;R}\|_{\Ld^p_q(\R^d;\Ld^r(\Omega))^d\to\Ld^p_q(\R^d;\Ld^r(\Omega))^d}\,\lesssim\,\tfrac{1}{p-1},\]
the conclusion follows.
\end{proof}

\subsection{Bourgain's deterministic lemma}\label{sec:detlem}
Using Lemma~\ref{lem:decay-K} and~\eqref{eq:compo-Rp}, say for $R=1$,
a direct estimate of the terms in the expansion~\eqref{eq:barA} for $\bar\Bc(\nabla)$ yields for all $1<p\le r\le2$,
\[\cro{P\Bb(\Kk P^\bot\Bb)^nP}_{p,r}(x,y)\,\le\,(\tfrac{C}{p-1})^n\sum_{z_1,\ldots,z_{n-1}\in\Z^d}\langle x-z_1\rangle^{-d}\ldots\langle z_{n-1}-y\rangle^{-d}.\]
Therefore, evaluating the sum,
\[\cro{P\Bb(\Kk P^\bot\Bb)^nP}_{p,r}(x,y)\,\le\,(\tfrac{C}{p-1})^n\langle x-y\rangle^{-d}\log(2+|x-y|)^{n-1}.\]
For all $\e>0$, using $\log t\le \e^{-1}t^\e$ for $t\ge1$, this bound translates into
\begin{equation}\label{eq:brutal-estim-det}
\cro{P\Bb(\Kk P^\bot\Bb)^nP}_{p,r}(x,y)\,\le\,n^n\e^{1-n}(\tfrac{C}{p-1})^n\langle x-y\rangle^{\e-d}.
\end{equation}
The combinatorial factor $n^n$ destroys any possible use of this direct estimate in the perturbative expansion~\eqref{eq:barA}.
Instead, in~\cite[Lemma~1]{Bourgain-18}, Bourgain made a more clever use of the global structure of the terms to show that this combinatorial factor can, in fact, be removed. In the present continuum setting, Bourgain's argument can be adapted as follows. We include a general statement for later purposes.

\begin{lem}[Bourgain's deterministic lemma]\label{lem:deterministic}
Given $n\ge1$, let $\ell_1,\ldots,\ell_n\ge R\ge1$ and let~$\Bb_1,\ldots,\Bb_n$ be bounded self-adjoint operators on $\Ld^2(\R^d\times\Omega)^d$ such that $\mathds1_{Q_R(z)}\Bb_j=\Bb_j\mathds1_{Q_R(z)}$ for all $j$ and $z\in R\Z^d$. Define
\[B_{p,q}\,:=\,\sup_j\|\Bb_j\|_{\Ld^{p}_q(\R^d;\Ld^2(\Omega))^d\to\Ld^{p}_q(\R^d;\Ld^2(\Omega))^d}.\]
Then, for all $1<p\le q\le2$, we have
\begin{equation*}
\cro{\Kk_{\ell_1;R}\Bb_1\ldots\Kk_{\ell_n;R}\Bb_n}_{p,q,2;R}(x,y)
\,\le\,
(CB_{p,q}\theta^{-1})^n\langle\tfrac1R(x-y)\rangle^{\theta-d},
\end{equation*}
where we have set $\theta:=\frac{2d}{p'}>0$, provided that $2\theta<d$.
\end{lem}

\begin{proof}
For $|x-y|_\infty\le2nR$, the claimed bound is an immediate consequence of~\eqref{eq:compo-Rp} and Lemma~\ref{lem:trunc-CZ} in form of
\begin{eqnarray*}
\cro{\Kk_{\ell_1;R}\Bb_1\ldots\Kk_{\ell_n;R}\Bb_n}_{p,q,2;R}(x,y)
&\le&\|\Kk_{\ell_1;R}\Bb_1\ldots\Kk_{\ell_n;R}\Bb_n\|_{\Ld^p_q(\R^d;\Ld^2(\Omega))^d\to\Ld^p_q(\R^d;\Ld^2(\Omega))^d}\\
&\le&(\tfrac{C}{p-1}B_{p,q})^n.
\end{eqnarray*}
Now let $x,y\in R\Z^d$ be fixed with $|x-y|_\infty>2nR$,
and consider $\phi_x,\phi_y\in C^\infty_c(\R^d;\Ld^\infty(\Omega))^d$ with $\supp(\phi_x)\subset Q_R(x)\times\Omega$ and $\supp(\phi_y)\subset Q_R(y)\times\Omega$.
Decomposing each operator $\Bb_j$ as
\[\Bb_j\,=\,\sum_{z\in R\Z^d}\Bb_j\mathds1_{Q_R(z)},\]
we may then write by Fubini's theorem,
\begin{multline}\label{eq:decomp-100}
\langle \phi_x,\Kk_{\ell_1;R}\Bb_1\ldots\Kk_{\ell_n;R}\Bb_n\phi_y\rangle\\
\,=\,\sum_{z_1,\ldots,z_{n-1}\in R\Z^d}\Big\langle \phi_x,(\Kk_{\ell_1;R}\Bb_1\mathds1_{Q_R(z_1)})\ldots(\Kk_{\ell_{n-1};R}\Bb_{n-1}\mathds1_{Q_R(z_{n-1})})(\Kk_{\ell_n;R}\Bb_n)\phi_y\Big\rangle,
\end{multline}
since this sum is absolutely convergent: indeed, we can bound
\begin{multline*}
\Big|\Big\langle \phi_x,(\Kk_{\ell_1;R}\Bb_1\mathds1_{Q_R(z_1)})\ldots(\Kk_{\ell_{n-1};R}\Bb_{n-1}\mathds1_{Q_R(z_{n-1})})(\Kk_{\ell_n;R}\Bb_n)\phi_y\Big\rangle\Big|\\[1mm]
\,\le\,B_{p,q}^n\|\phi_x\|_{\Ld^{p'}_{q'}(\R^d;\Ld^{2}(\Omega))}\|\phi_y\|_{\Ld^{p}_q(\R^d;\Ld^2(\Omega))}\cro{\Kk_{\ell_1;R}}_{p,q,2;R}(x,z_1)\ldots \cro{\Kk_{\ell_n;R}}_{p,q,2;R}(z_{n-1},y),
\end{multline*}
and the summability of the right-hand side over $z_1,\ldots,z_{n-1}\in R\Z^d$ immediately follows from Lemma~\ref{lem:decay-K}.
To estimate the sum~\eqref{eq:decomp-100}, we shall distinguish with respect to the first interval $|z_j-z_{j+1}|_\infty$ that reaches the largest dyadic value, and we denote by $j=j_1$ the corresponding index. More precisely, setting for notational simplicity $z_0:=x$ and $z_n:=y$, we define for all $0\le j_1< n$ and $m\ge0$,
\begin{multline*}
S_{j_1;R}^m(x,y):=\Big\{(z_1,\ldots,z_{n-1})\in(R\Z^d)^{n-1}~:~\max_{0\le j< n}|z_j-z_{j+1}|_\infty\le2^{m+1}R,\\
~|z_{j_1}-z_{j_1+1}|_\infty>2^mR,
~\text{and}~\max_{0\le j<j_1}|z_j-z_{j+1}|_\infty\le2^mR\Big\}.
\end{multline*}
Note that by the triangle inequality the condition $|x-y|_\infty>2nR$ implies
\[\max_{0\le j<n}|z_j-z_{j+1}|_\infty\,>\,2R.\]
In these terms, the sum~\eqref{eq:decomp-100} becomes
\begin{multline*}
\langle \phi_x,\Kk_{\ell_1;R}\Bb_1\ldots\Kk_{\ell_n;R}\Bb_n\phi_y\rangle
\,=\,\sum_{j_1=0}^{n-1}\sum_{m=0}^\infty\sum_{(z_1,\ldots,z_{n-1})\in S^m_{j_1;R}(x,y)}\\
\times\Big\langle \phi_x\,,\,(\Kk_{\ell_1;R}\Bb_1\mathds1_{Q_R(z_1)})\ldots(\Kk_{\ell_{n-1};R}\Bb_{n-1}\mathds1_{Q_R(z_{n-1})})(\Kk_{\ell_n;R}\Bb_n)\phi_y\Big\rangle.
\end{multline*}
For $(z_1,\ldots,z_{n-1})\in S^m_{j_1;R}(x,y)$, the triangle inequality
\[|x-y|_\infty\,\le\,\sum_{j=0}^{n-1}|z_j-z_{j+1}|_\infty\,\le\,n2^{m+1}R\]
entails that the sum over $m$ is automatically restricted to $2^mR\ge\frac1{2n}|x-y|_\infty$. Similarly, the sums over $z_{j_1},z_{j_1+1}$ are restricted to $|x-z_{j_1}|_\infty\le n2^{m}R$ and $|y-z_{j_1+1}|_\infty\le n2^{m+1}R$, respectively.
Further noting that the truncated kernels are defined in~\eqref{eq:truncate-K} to satisfy for all~$z,z'\in R\Z^d$ and $\ell,\ell'\ge R$,
\[\mathds1_{|z-z'|_\infty\le \ell'}\mathds1_{Q_R(z)}\Kk_{\ell;R}\mathds1_{Q_R(z')}\,=\,\mathds1_{Q_R(z)}\Kk_{\ell\wedge\ell';R}\mathds1_{Q_R(z')},\]
the above can be rewritten as
\begin{multline}\label{eq:use-Sj0R}
\langle \phi_x,\Kk_{\ell_1;R}\Bb_1\ldots\Kk_{\ell_n;R}\Bb_n\phi_y\rangle
\,=\,\sum_{j_1=0}^{n-1}\sum_{m=0}^\infty\mathds1_{2^mR\ge\frac1{2n}|x-y|_\infty}\\
\times\sum_{z_{j_1},z_{j_1+1}\in R\Z^d}\mathds1_{2^mR<|z_{j_1}-z_{j_1+1}|_\infty\le2^{m+1}R}\,\mathds1_{|x-z_{j_1}|_\infty,|y-z_{j_1+1}|_\infty\le n2^{m+1}R}\\
\times\Big\langle \phi_x\,,\,(\Kk_{\ell_1\wedge(2^mR);R}\Bb_1)\ldots(\Kk_{\ell_{j_1}\wedge(2^mR);R}\Bb_{j_1})\\
\times(\mathds1_{Q_R(z_{j_1})}\Kk_{\ell_{j_1+1};R}\mathds1_{Q_R(z_{j_1+1})})\\
\times\Bb_{j_1+1}(\Kk_{\ell_{j_1+2}\wedge(2^{m+1}R);R}\Bb_{j_1+2})\ldots(\Kk_{\ell_n\wedge(2^{m+1}R)}\Bb_n)\phi_y\Big\rangle.
\end{multline}
In other words, the restricted summation over $S_{j_1;R}^m$ has been replaced by kernel truncations.
Using the pointwise decay of the kernel $\Kk$ along the long segment $[z_{j_1},z_{j_1+1}]$, cf.~Lemma~\ref{lem:CZ},
we get
\begin{multline*}
\langle \phi_x,\Kk_{\ell_1;R}\Bb_1\ldots\Kk_{\ell_n;R}\Bb_n\phi_y\rangle
\,\lesssim\,\sum_{j_1=0}^{n-1}\sum_{m=0}^\infty(2^{m}R)^{-d}\mathds1_{2^{m}R\ge\frac1{2n}|x-y|_\infty}\\
\times\sum_{z\in R\Z^d:\atop |x-z|_\infty\le n2^{m+1}R}\int_{Q_R(z)}\Big\|(\Bb_{j_1}\Kk_{\ell_{j_1}\wedge(2^{m}R);R}) \ldots(\Bb_1\Kk_{\ell_1\wedge(2^{m}R);R})\phi_x\Big\|_{\Ld^2(\Omega)}\\
\times\sum_{z'\in R\Z^d:\atop |y-z'|_\infty\le n2^{m+1}R}\int_{Q_R(z')}\Big\|\Bb_{j_1+1}(\Kk_{\ell_{j_1+2}\wedge(2^{m+1}R);R}\Bb_{j_1+2})\ldots(\Kk_{\ell_n\wedge(2^{m+1}R)}\Bb_n)\phi_y\Big\|_{\Ld^2(\Omega)}.
\end{multline*}
Then using H\" older's inequality to bound the sums over $z,z'$, we find
\begin{multline*}
\langle \phi_x,\Kk_{\ell_1;R}\Bb_1\ldots\Kk_{\ell_n;R}\Bb_n\phi_y\rangle
\,\lesssim\,\sum_{j_1=0}^{n-1}\sum_{m=0}^\infty(n2^{m+1}R)^{\frac{2d}{p'}}(2^{m}R)^{-d}\mathds1_{2^{m}R\ge\frac1{2n}|x-y|_\infty}\\
\times\|(\Bb_{j_1}\Kk_{\ell_{j_1}\wedge(2^{m}R);R} )\ldots(\Bb_1\Kk_{\ell_1\wedge(2^{m}R);R})\phi_x\|_{\Ld^p_q(\R^d;\Ld^2(\Omega))}\\
\times\|\Bb_{j_1+1}(\Kk_{\ell_{j_1+2}\wedge(2^{m+1}R);R}\Bb_{j_1+2})\ldots(\Kk_{\ell_n\wedge(2^{m+1}R)}\Bb_n)\phi_y\|_{\Ld^p_q(\R^d;\Ld^2(\Omega))}.
\end{multline*}
By Lemma~\ref{lem:trunc-CZ} together with the boundedness assumption for the $\Bb_j$'s, this yields
\begin{multline*}
\langle \phi_x,\Kk_{\ell_1;R}\Bb_1\ldots\Kk_{\ell_n;R}\Bb_n\phi_y\rangle\\
\,\lesssim\,B_{p,q}^n\big(\tfrac{C}{p-1}\big)^{n-1}\|\phi_x\|_{\Ld^p_q(\R^d;\Ld^2(\Omega))}\|\phi_y\|_{\Ld^p_q(\R^d;\Ld^2(\Omega))}
\sum_{m=0}^\infty(2^m R)^{\frac{2d}{p'}-d}\,\mathds1_{2^mR\ge\frac1{2n}|x-y|_\infty},
\end{multline*}
and thus, evaluating the dyadic series, provided that $\frac{4d}{p'}<d$,
\begin{equation*}
\langle \phi_x,\Kk_{\ell_1;R}\Bb_1\ldots\Kk_{\ell_n;R}\Bb_n\phi_y\rangle
\,\lesssim\,B_{p,q}^n\big(\tfrac{C}{p-1}\big)^{n-1}|x-y|^{\frac{2d}{p'}-d}\|\phi_x\|_{\Ld^p_q(\R^d;\Ld^2(\Omega))}\|\phi_y\|_{\Ld^p_q(\R^d;\Ld^2(\Omega))}.
\end{equation*}
As $\phi_x$ is supported in $Q_R(x)\times\Omega$, we find by Jensen's inequality,
\[\|\phi_x\|_{\Ld^p_q(\R^d;\Ld^2(\Omega))}\,\lesssim\, R^{d\frac{2-p}p}\|\phi_x\|_{\Ld^{p'}_{q'}(\R^d;\Ld^2(\Omega))}\,=\,R^{d-\frac{2d}{p'}}\|\phi_x\|_{\Ld^{p'}_{q'}(\R^d;\Ld^2(\Omega))}.\]
Using this and taking the supremum over $\phi_x,\phi_y$, the above yields
\begin{equation*}
\cro{\Kk_{\ell_1;R}\Bb_1\ldots\Kk_{\ell_n;R}\Bb_n}_{p,q,2;R}
\,\lesssim\,B_{p,q}^n\big(\tfrac{C}{p-1}\big)^{n-1}
\langle\tfrac1R(x-y)\rangle^{\frac{2d}{p'}-d},
\end{equation*}
and the conclusion follows.
\end{proof}

\section{Stretched exponential mixing setting}
\label{sect:stretchedexponentialmixingproof}
This section is devoted to the proof of Theorem~\ref{th:exp-mix} in the exponential mixing setting~{\bf(H$_1$)}.
More precisely, the $\alpha$-mixing condition will be used in form of the following covariance inequality (see e.g.~\cite[Theorem~1.2.3]{Doukhan-94}): for all $U,V\subset\R^d$ and all bounded random variables $X,Y$, if~$X$ is $\sigma(\Aa|_U)$-measurable and if~$Y$ is $\sigma(\Aa|_V)$-measurable, we have for all $p,q,r\ge1$ with $\frac1p+\frac1q+\frac1r=1$,
\begin{equation}\label{eq:cov-inequ}
|\!\cov{X}{Y}\!|\,\le\,8\big[C_0\exp(-\tfrac1{C_0}\dist(U,V)^\gamma)\big]^\frac1r\|X\|_{\Ld^p(\Omega)}\|Y\|_{\Ld^q(\Omega)}.
\end{equation}
We start with several preliminary ingredients and conclude the proof of Theorem~\ref{th:exp-mix} in Section~\ref{sec:proof-exp-mix}. We follow the general strategy of Bourgain~\cite{Bourgain-18}, improved by Kim and the second-named author in~\cite{Lemm-18}, but the present continuum setting requires several important modifications, starting with the coarse-grained notion of irreducibility in Definition~\ref{def:irred} below. We emphasize that most of the present modifications are also needed in the discrete setting beyond the particular i.i.d.\@ case treated in~\cite{Bourgain-18,Lemm-18}.

\subsection{Path decomposition}
Starting point is the perturbative expansion~\eqref{eq:barA}, and we shall proceed to a separate analysis of the kernel of the different terms $\{P\Bb(\Kk P^\bot \Bb)^nP\}_{n\ge1}$.
Given $R\ge1$ and $x,y\in R\Z^d$, arguing as in~\eqref{eq:decomp-100}, we can decompose for all $\phi_x,\phi_y\in C^\infty_c(\R^d;\Ld^\infty(\Omega))^d$ with $\supp(\phi_x)\subset Q_R(x)\times\Omega$ and $\supp(\phi_y)\subset Q_R(y)\times\Omega$,
\begin{multline}\label{eq:decomp-10}
\langle \phi_x,P\Bb(\Kk P^\bot\Bb)^nP\phi_y\rangle\\
\,=\,\sum_{z_1,\ldots,z_{n-1}\in R\Z^d}\Big\langle \phi_x,P\Bb(\Kk P^\bot\Bb\mathds1_{Q_R(z_1)})\ldots(\Kk P^\bot\Bb\mathds1_{Q_R(z_{n-1})})(\Kk P^\bot\Bb)P\phi_y\Big\rangle,
\end{multline}
where the series is absolutely convergent.
At first sight, from the deterministic estimates of Section~\ref{sec:detlem}, this expression might seem to have no reason to enjoy a better decay than the kernel~$\Kk$ itself, that is, $O(|x-y|^{-d})$. Yet, this decay happens to be drastically improved due to couplings that appear when computing the expectation of products of the stationary random field~$\Bb$ in this expression.
The resulting cancellations can be characterized in terms of the following notion of reducibility, which is a coarse-grained version of the one introduced by Bourgain in~\cite{Bourgain-18}.

\begin{defin}\label{def:irred}
Let $n\ge2$, $R\ge1$, and $x_0,\ldots,x_n\in R\Z^d$. The sequence $(x_0,\ldots,x_n)$ is said to be an \emph{$R$-reducible path} (of length $n$ from $x_0$ to $x_n$) if there exists $0\le j< n$ such that
\[|x_a-x_b|_\infty>R\qquad\text{for all $0\le a\le j$ and $j< b\le n$}.\]
Otherwise, the sequence $(x_0,\ldots,x_n)$ is said to be an \emph{$R$-irreducible path}.
We denote by $D_R^n(x,y)$ the subset of elements $(z_1,\ldots,z_{n-1})\in(R\Z^d)^{n-1}$ such that $(x,z_1,\ldots,z_{n-1},y)$ is an $R$-reducible path.
\end{defin}

We emphasize that the condition $|x_a-x_b|_\infty>R$ in this definition means that the cubes~$Q_R(x_a)$ and $Q_R(x_b)$ are not neighbors, hence are at distance $\gtrsim R$ of one another: this distance will allow us to use approximate independence in form of $\alpha$-mixing to neglect the contribution of those reducible paths, cf.~Lemma~\ref{lem:bound-reduct}.
To start gently, let us illustrate this definition in the case when the coefficient field $\Aa$ has a finite range of dependence bounded by $R$,
showing that the contribution of $R$-reducible paths vanishes in that case.
Although a priori counterintuitive, it will be crucial to use strict subsets $A(x,y)$ of the set of irreducible paths: indeed, restricting the sum too much may destroy the special oscillatory structure of the composition of Calder\'on--Zygmund kernels and lead to worse estimates; see Remark~\ref{rem:sum-irred} below.

\begin{lem}\label{lem:sum/irred}
Assume momentarily that $\Aa$ has finite range of dependence in the sense that for all $U,V\subset\R^d$ with $\dist(U,V)>R$ the $\sigma$-algebras $\sigma(\Aa|_U)$ and $\sigma(\Aa|_V)$ are independent.
Then, for all $n\ge1$, the $n$th term in the perturbative expansion~\eqref{eq:barA} can be written as follows: for all $x,y\in R\Z^d$, for all $\phi_x,\phi_y\in C^\infty_c(\R^d)^d$ with $\supp(\phi_x)\subset Q_R(x)$ and $\supp(\phi_y)\subset Q_R(y)$,
\begin{multline}\label{eq:sum/irred-lem}
\big\langle \phi_x,P\Bb(\Kk P^\bot\Bb)^nP\phi_y\big\rangle\\
\,=\,\sum_{(z_1,\ldots,z_{n-1})\in A(x,y)}
\Big\langle \phi_x,P\Bb(\Kk P^\bot\Bb\mathds1_{Q_R(z_1)})\ldots(\Kk P^\bot\Bb\mathds1_{Q_R(z_{n-1})})(\Kk P^\bot\Bb)P\phi_y\Big\rangle,
\end{multline}
where $A(x,y)$ stands for any subset of $(R\Z^d)^{n-1}$ containing the set $(R\Z^d)^{n-1}\setminus D_R^n(x,y)$ of $R$-irreducible paths from $x$ to $y$.
\end{lem}

\begin{proof}
Starting from~\eqref{eq:decomp-10}, it suffices to show that the contribution of any $R$-reducible path vanishes in the case when $\Aa$ has finite range of dependence bounded by $R$. For that purpose, we start by appealing to Lemma~\ref{lem:CZ-eta} to express operators as limits of absolutely converging integrals: for all $z_1,\ldots, z_{n-1}\in R\Z^d$, we have
\begin{multline*}
\Big\langle \phi_x,P\Bb(\Kk P^\bot\Bb\mathds1_{Q_R(z_1)})\ldots(\Kk P^\bot\Bb\mathds1_{Q_R(z_{n-1})})(\Kk P^\bot\Bb)P\phi_y\Big\rangle\\
\,=\,\lim_{\eta\downarrow0}\Big\langle \phi_x,P\Bb(\Kk^{(\eta)} P^\bot\Bb\mathds1_{Q_R(z_1)})\ldots(\Kk^{(\eta)} P^\bot\Bb\mathds1_{Q_R(z_{n-1})})(\Kk^{(\eta)} P^\bot\Bb)P\phi_y\Big\rangle,
\end{multline*}
where the limit exists. For fixed $\eta>0$, by Fubini's theorem, the argument of the limit can now be written as follows in terms of truncated integral kernels,
\begin{multline*}
\Big\langle \phi_x,P\Bb(\Kk^{(\eta)} P^\bot\Bb\mathds1_{Q_R(z_1)})\ldots(\Kk^{(\eta)} P^\bot\Bb\mathds1_{Q_R(z_{n-1})})(\Kk^{(\eta)} P^\bot\Bb)\phi_y\Big\rangle\\
\,=\,\int\!\!\ldots\!\!\int_{Q_R(x)\times Q_R(z_1)\times\ldots\times Q_R(z_{n-1})\times Q_R(y)} \expecm{\Bb(y_0)P^\bot\Bb(y_1)\ldots  P^\bot\Bb(y_{n})}\\
\times\phi_x(y_0)\Kk^{(\eta)}(y_0-y_1)\ldots \Kk^{(\eta)}(y_{n-1}-y_{n})\phi_y(y_n)\,dy_0\ldots dy_n,
\end{multline*}
where the integral is indeed trivially absolutely convergent due to the small-scale cut-off,
and where we omit to track matrix contractions for notational simplicity.
Now here is the key point: by definition, if $\Bb$ has finite dependence length bounded by~$R$ as presently assumed, then the factor $\expecm{\Bb(y_0)P^\bot\Bb(y_1)\ldots  P^\bot\Bb(y_{n})}$ in the above expression vanishes for all $y_0\in Q_R(x),y_1\in Q_R(z_1),\ldots,y_{n-1}\in Q_R(z_{n-1}),y_n\in Q_R(y)$ whenever $(x,z_1,\ldots,z_{n-1},y)$ is an $R$-reducible path. The sum can therefore be restricted by removing any subset of $R$-reducible paths, and the conclusion follows.
\end{proof}

\begin{rem}[Summing irreducible paths]\label{rem:sum-irred}
When restricting the summation~\eqref{eq:sum/irred-lem} to irreducible paths, we can easily understand the best decay rate $O(\langle x-y\rangle^{-3d})$ from straightforward graphical considerations.
Indeed, given an irreducible path $(x,z_1,\ldots,z_{n-1},y)$, consider the set $Z:=\{x,z_1,\ldots,z_{n-1},y\}$, and define the following equivalence relation on $Z$: two elements $a,b\in Z$ are said to be equivalent if there is $k\ge0$ and a sequence $\{y_i\}_{0\le i\le k}\subset Z$ such that $y_0=a$, $y_k=b$, and such that $|y_i-y_{i+1}|_\infty=R$ for all $0\le i<k$.
Now consider the set $V$ that is the quotient of~$Z$ with respect to this equivalence relation,
and consider the non-oriented graph~$G$ induced by the path $(x,z_1,\ldots,z_{n-1},y)$ on the quotient set~$V$.
Note that the points $x$ and $y$ are not equivalent provided that $|x-y|_\infty>2Rn$. Choosing representatives, we may then write $V\equiv\{x,y,v_1,\ldots,v_l\}\subset Z$ for some $0\le l<n$.
As the path $(x,z_1,\ldots,z_{n-1},y)$ is irreducible, we can deduce that the vertices $x$ and $y$ in the induced graph $G$ have odd degrees $\ge3$ and that each other vertex has even degree~$\ge2$.
Due to this property, we can find three disjoint trails~$L_1,L_2,L_3$ from~$x$ to~$y$ in $G$. Separately evaluating the sums in~\eqref{eq:sum/irred-lem} along each of these trails, the decay rate $O(\langle x-y\rangle^{-3d})$ follows up to logarithmic corrections.
Yet, as this argument is based on a direct summation of iterated kernels, which is key to take full advantage of irreducibility, it would only provide an estimate with a prefactor $O(n^n)$ similarly as in~\eqref{eq:brutal-estim-det}; see also~\cite[Section~1.3]{Lemm-18}.
We skip the detail as such brutal estimates are anyway not summable over $n$ and thus of no use.
In the sequel, as in~\cite{Bourgain-18,Lemm-18}, the irreducibility will be used instead in a much weaker, minimal way, in order not to destroy the special oscillatory structure of compositions of Calder\'on--Zygmund kernels. More precisely, not all reducible paths will be removed from the summation, and the key technical ingredient is given by Lemma~\ref{lem:disjointif} below.
\end{rem}

\subsection{Neglecting reducible paths}
In the mixing setting, we cannot appeal to Lemma~\ref{lem:sum/irred} above to remove the contribution of reducible paths, but we show that such paths still have only a negligible contribution for $R\gg1$.
We proceed by a direct summation of reducible paths, which is why one loses a prefactor~$O(n^n)$ similarly as in~\eqref{eq:brutal-estim-det}.
In the stretched exponential mixing setting, we shall see that this factor can be compensated by the excellent mixing rate. We do not know how to improve on the present estimate and leave it as an open question, in link with the treatment of {\it algebraic} $\alpha$-mixing rates.

\begin{lem}\label{lem:bound-reduct}
For all $n,R\ge1$, and $x,y\in R\Z^d$,
the contribution of $R$-reducible paths in~\eqref{eq:decomp-10} is estimated as follows, for all $1<p\le q\le r\le2$ and~$\e>0$,
\begin{multline*}
\sum_{(z_1,\ldots, z_{n-1})\in D_R^n(x,y)}\cro{P\Bb(\Kk P^\bot\Bb\mathds1_{Q_R(z_1)})\ldots(\Kk P^\bot\Bb\mathds1_{Q_R(z_{n-1})})(\Kk P^\bot \Bb)P}_{p,q,r;R}\\
\,\le\,(Cn)^n(\e(p-1))^{1-n}\big(C_0\exp(-\tfrac1{C_0}R^\gamma)\big)^\frac1{2}\langle \tfrac1R(x-y)\rangle^{\e-d},
\end{multline*}
where we recall that $D_R^n(x,y)$ stands for the set of elements $(z_1,\ldots,z_{n-1})\in R\Z^d$ such that the path $(x,z_1,\ldots,z_{n-1},y)$ is $R$-reducible, cf.~Definition~\ref{def:irred}.
\end{lem}

\begin{proof}
Let $x,y\in R\Z^d$ be fixed, and consider $\phi_x,\phi_y\in C^\infty_c(\R^d)^d$ with $\supp(\phi_x)\subset Q_R(x)$ and $\supp(\phi_y)\subset Q_R(y)$.
For $(z_1,\ldots,z_{n-1})\in D_R^n(x,y)$, as the path $(x,z_1,\ldots,z_{n-1},y)$ is $R$-reducible, there exists some $0\le j_0< n$ such that
\[\dist\big(\{x,z_1,\ldots,z_{j_0}\},\{z_{j_0+1},\ldots,z_{n-1},y\}\big)\,>\,R.\]
Then writing
\begin{multline*}
\Big\langle\phi_x,P\Bb(\Kk P^\bot\Bb\mathds1_{Q_R(z_1)})\ldots(\Kk P^\bot\Bb\mathds1_{Q_R(z_{n-1})})(\Kk P^\bot \Bb)P\phi_y\Big\rangle\\
\,=\,\int_{Q_R(z_{j_0})}\operatorname{Cov}\Big[\Big(( P^\bot\Bb\mathds1_{Q_R(z_{j_0})}\Kk)\ldots( P^\bot\Bb\mathds1_{Q_R(z_1)}\Kk)P^\bot\Bb P\phi_x\Big)(v),\\
\Big((\Kk P^\bot\Bb\mathds1_{Q_R(z_{j_0+1})})\ldots(\Kk P^\bot\Bb\mathds1_{Q_R(z_{n-1})})(\Kk P^\bot\Bb)P\phi_y\Big)(v)\Big]\,dv,
\end{multline*}
we can appeal to the exponential $\alpha$-mixing condition in form of the corresponding covariance inequality~\eqref{eq:cov-inequ}, followed by H\"older's inequality on $Q_R(z_{j_0})$, to the effect of
\begin{multline*}
\Big|\Big\langle\phi_x,P\Bb(\Kk P^\bot\Bb\mathds1_{Q_R(z_1)})\ldots(\Kk P^\bot\Bb\mathds1_{Q_R(z_{n-1})})(\Kk P^\bot \Bb)\phi_y\Big\rangle\Big|\\
\,\le\,8\big(C_0\exp(-\tfrac1{C_0}R^\gamma)\big)^\frac1{2}\big\|( P^\bot\Bb\mathds1_{Q_R(z_{j_0})}\Kk)\ldots( P^\bot\Bb\mathds1_{Q_R(z_1)}\Kk)P^\bot\Bb P\phi_x\big\|_{\Ld^1(Q_R(z_{j_0});\Ld^{2}(\Omega))}\\
\times\big\|( \Kk P^\bot\Bb\mathds1_{Q_R(z_{j_0+1})})\ldots(\Kk P^\bot\Bb\mathds1_{Q_R(z_{n-1})})(\Kk P^\bot \Bb)P\phi_y\big\|_{\Ld^\infty(Q_R(z_{j_0})\times\Omega)}.
\end{multline*}
Using a pointwise bound on the kernel of $\Kk$ along the long segment $|z_{j_0}-z_{j_0+1}|_\infty>R$ in the last factor, and further using H\"older's
inequality, we are led to
\begin{multline*}
\Big|\Big\langle\phi_x,P\Bb(\Kk P^\bot\Bb\mathds1_{Q_R(z_1)})\ldots(\Kk P^\bot\Bb\mathds1_{Q_R(z_{n-1})})(\Kk P^\bot \Bb)\phi_y\Big\rangle\Big|\\
\,\lesssim\,\big(C_0\exp(-\tfrac1{C_0}R^\gamma)\big)^\frac1{2}R^{\frac{2d}{p'}-d}\langle \tfrac1R(z_{j_0}-z_{j_0+1})\rangle^{-d}\\
\times\big\|\mathds1_{Q_R(z_{j_0})}(\Kk P^\bot\Bb\mathds1_{Q_R(z_{j_0-1})})\ldots( \Kk P^\bot\Bb\mathds1_{Q_R(z_1)})(\Kk P^\bot\Bb)P\phi_x\big\|_{\Ld^{p}_q(\R^d;\Ld^2(\Omega))}\\
\times\big\|\mathds1_{Q_R(z_{j_0+1})}( \Kk P^\bot\Bb\mathds1_{Q_R(z_{j_0+2})})\ldots(\Kk P^\bot\Bb\mathds1_{Q_R(z_{n-1})})(\Kk P^\bot \Bb)P\phi_y\big\|_{\Ld^\infty(\Omega;\Ld^p_q(\R^d))}.
\end{multline*}
Hence, by Lemma~\ref{lem:decay-K},
\begin{multline*}
\Big|\Big\langle\phi_x,P\Bb(\Kk P^\bot\Bb\mathds1_{Q_R(z_1)})\ldots(\Kk P^\bot\Bb\mathds1_{Q_R(z_{n-1})})(\Kk P^\bot \Bb)\phi_y\Big\rangle\Big|\\
\,\lesssim\,\big(C_0\exp(-\tfrac1{C_0}R^\gamma)\big)^\frac1{2}R^{\frac{2d}{p'}-d}
(\tfrac{C}{p-1})^{n-1}\|\phi_x\|_{\Ld^p_q(\R^d)}\|\phi_y\|_{\Ld^p_q(\R^d)}\\
\times\langle\tfrac1R(x-z_1)\rangle^{-d}\langle\tfrac1R(z_1-z_2)\rangle^{-d}\ldots\langle\tfrac1R(z_{n-1}-y)\rangle^{-d}.
\end{multline*}
As $\phi_x$ is supported in $Q_R(x)$, we find by H\"older's inequality,
\[\|\phi_x\|_{\Ld^p_q(\R^d)}\,\lesssim\,R^{d\frac{2-p}p}\|\phi_x\|_{\Ld^{p'}_{q'}(\R^d)}\,=\,R^{d-\frac{2d}{p'}}\|\phi_x\|_{\Ld^{p'}_{q'}(\R^d)}.\]
Using this and taking the supremum over $\phi_x,\phi_y$, the above yields
\begin{multline*}
\cro{P\Bb(\Kk P^\bot\Bb\mathds1_{Q_R(z_1)})\ldots(\Kk P^\bot\Bb\mathds1_{Q_R(z_{n-1})})(\Kk P^\bot \Bb)P}_{p,q,2;R}\\
\,\lesssim\,\big(C_0\exp(-\tfrac1{C_0}R^\gamma)\big)^\frac1{2}
(\tfrac{C}{p-1})^{n-1}
\langle\tfrac1R(x-z_1)\rangle^{-d}\langle\tfrac1R(z_1-z_2)\rangle^{-d}\ldots\langle\tfrac1R(z_{n-1}-y)\rangle^{-d}.
\end{multline*}
Summing over $z_1,\ldots,z_{n-1}$, we deduce
\begin{multline*}
\sum_{(z_1,\ldots, z_{n-1})\in D_R^n(x,y)}\cro{P\Bb(\Kk P^\bot\Bb\mathds1_{Q_R(z_1)})\ldots(\Kk P^\bot\Bb\mathds1_{Q_R(z_{n-1})})(\Kk P^\bot \Bb)P}_{p,q,2;R}\\
\,\lesssim\,\big(C_0\exp(-\tfrac1{C_0}R^\gamma)\big)^\frac1{2}
(\tfrac{C}{p-1})^{n-1}\langle \tfrac1R(x-y)\rangle^{-d}\log(2+\tfrac1R|x-y|)^{n-1}.
\end{multline*}
For all $\e>0$, using $\log t\le\e^{-1}t^\e$ for $t\ge1$, this yields the conclusion.
\end{proof}

\subsection{Bourgain's disjointification lemma}
We shall use the following generalized version of Bourgain's disjointification lemma~\cite[pp.319--320]{Bourgain-18} (see also~\cite[Lemma~2.7]{Lemm-18}). The present version is adapted to the new coarse-grained notion of reducibility, cf.~Definition~\ref{def:irred}.
In a nutshell, this result shows that sufficiently simple path restrictions do not destroy the behavior of compositions of kernels.

\begin{lem}\label{lem:disjointif}
Given $n\ge1$, let $\Ll_1,\ldots,\Ll_n$ be bounded operators and let $\Bb_0,\ldots,\Bb_n$ be bounded multiplication operators on $\Ld^2(\R^d\times\Omega)^d$.
Given $R\ge1$, $x,y\in R\Z^d$, and $\phi_x,\phi_y\in C^\infty_c(\R^d;\Ld^\infty(\Omega))^d$ with $\supp(\phi_x)\subset Q_R(x)\times\Omega$ and $\supp(\phi_y)\subset Q_R(y)\times\Omega$,
consider the function $T_{x,y}^n:(R\Z^d)^{n-1}\to\R$ given by
\begin{equation*}
T_{x,y}^n(z_1,\ldots,z_{n-1})\,:=\,\Big\langle\phi_x\,,\,\Bb_0(\Ll_1\Bb_1\mathds1_{Q_R(z_1)})\ldots(\Ll_{n-1}\Bb_{n-1}\mathds1_{Q_R(z_{n-1})})(\Ll_n\Bb_n)\phi_y\Big\rangle.
\end{equation*}
Given a finite subset $S\subset (R\Z^d)^{n-1}$, assume that $T_{x,y}^n$ satisfies the following bound,
\begin{equation}\label{eq:as-bnd-Txy}
\bigg|\sum_{(z_1,\ldots,z_{n-1})\in S}T_{x,y}^n(z_1,\ldots,z_{n-1})\bigg|\,\le\,M(x,y),
\end{equation}
for some function $M:(R\Z^d)^2\to\R^+$, and assume that this bound is stable under any change of the $\Bb_j$'s of the form
\begin{equation}\label{eq:change-kernel}
\Bb_j\,\to\, r_j\Bb_j,
\end{equation}
where $r_j$ is any multiplication by a function in $\Ld^\infty(\R^d)$ with $\|r_j\|_{\Ld^\infty(\R^d)}\le1$.
Given $m\ge1$ and a sequence of index subsets $E_l,F_l\subset\{0,\ldots,n\}$, $1\le l\le m$, consider now the following restriction of the set $S$, where we let $z_0:=x$ and $z_n:=y$ for notational simplicity,
\[S'\,:=\,S\cap\bigcap_{l=1}^m\Big\{(z_1,\ldots,z_{n-1})\in(R\Z^d)^{n-1}:|z_j-z_k|_\infty> R~~\text{for all $j\in E_l$, $k\in F_l$}\Big\}.\]
Then the bound~\eqref{eq:as-bnd-Txy} on $T_{x,y}^n$ is essentially preserved when summing over $S'$ instead of $S$, in form of
\begin{equation}\label{eq:sum-S'-state}
\bigg|\sum_{(z_1,\ldots,z_{n-1})\in S'}T_{x,y}^n(z_1,\ldots,z_{n-1})\bigg|\,\le\,2^{3^d\sum_{l=1}^m(\sharp E_l+\sharp F_l-1)}M(x,y).
\end{equation}
\end{lem}

\begin{rem}
Given $L\in R\N$, replacing the condition $|z_i-z_j|_\infty>R$ by $|z_i-z_j|_\infty>L$ in the definition of $S'$ would lead to a corresponding estimate with $3^d$ replaced by $(2\frac LR+1)^d$. This excessive growth in the exponential would quickly make it useless for $L\gg R$, which is why we stick here to the case $L=R$.
\end{rem}

\begin{proof}[Proof of Lemma~\ref{lem:disjointif}]
First note that the restricted set $S'$ in the statement can be rewritten as follows, in terms of $N_R:=\{z\in R\Z^d:|z|_\infty\le R\}$,
\begin{equation*}
S'\,=\,S\cap\bigcap_{l=1}^m\bigcap_{\alpha\in N_R}\Big\{(z_1,\ldots,z_{n-1})\in(R\Z^d)^{n-1}:
\{z_j+\alpha:j\in E_l\}\cap\{z_j:j\in F_l\}=\varnothing\Big\},
\end{equation*}
where we let $z_0:=x$ and $z_n:=y$ for notational simplicity.
Without loss of generality, we can assume $E_l\cap F_l=\varnothing$ for all $l$, as otherwise $S'=\varnothing$.
By a direct induction argument, successively adding restrictions to the set $S$, we observe that it suffices to prove the following reduced result:
given $E,F\subset\{0,\ldots,n\}$ with $E\cap F=\varnothing$, and given $\alpha\in N_R$, considering the following restriction of the set $S$,
\begin{equation}\label{eq:def-S'-simpl}
S''\,:=\,S\cap\Big\{(z_1,\ldots,z_{n-1})\in(R\Z^d)^{n-1}:\{z_j+\alpha:j\in E\}\cap\{z_j:j\in F\}=\varnothing\Big\},
\end{equation}
we have
\begin{equation}\label{eq:sum-S'-red}
\bigg|\sum_{(z_1,\ldots,z_{n-1})\in S''}T_{x,y}^n(z_1,\ldots,z_{n-1})\bigg|\,\le\,2^{\sharp E+\sharp F-1}M(x,y).
\end{equation}
We turn to the proof of this reduced result~\eqref{eq:sum-S'-red}, for which we essentially follow Bourgain's original argument in~\cite[pp.319--320]{Bourgain-18}.
We start by introducing an additional set of variables $\bar\theta=(\theta_{z})_{z\in R\Z^d}$ with $\theta_{z}\in\R/2\pi\Z$, and 
we define
\[r_j(x,\bar\theta)\,:=\,\exp\Big(i\big(\mathds1_{j\in E}\theta_{z_R(x)+\alpha}-\mathds1_{j\in F}\theta_{z_R(x)}\big)\Big),\]
where we recall the definition of the map $z_R:\R^d\to R\Z^d$ in~\eqref{eq:zRdef}.
By definition, this means for all $z_j\in R\Z^d$,
\begin{eqnarray*}
r_j(\cdot,\bar\theta)\mathds1_{Q_R(z_j)}
&=&\mathds1_{Q_R(z_j)}
\exp\Big(i\big(\mathds1_{j\in E}\theta_{z_j+\alpha}-\mathds1_{j\in F}\theta_{z_j}\big)\Big)\\
&=&\mathds1_{Q_R(z_j)}\times\left\{\begin{array}{lll}
\exp(i\theta_{z_{j}+\alpha})&:&\text{if $j\in E$},\\
\exp(-i\theta_{z_j})&:&\text{if $j\in F$},\\
1&:&\text{otherwise.}
\end{array}\right.
\end{eqnarray*}
By assumption~\eqref{eq:as-bnd-Txy}, using stability under the change~\eqref{eq:change-kernel} with the above choice of functions~$r_j$'s,
and factoring out the phases,
we obtain
\begin{equation}\label{eq:estim-Steinhaus-pol}
\bigg|\sum_{(z_1,\ldots,z_{n-1})\in S}T_{x,y}^n(z_1,\ldots,z_{n-1})\exp\Big(i\sum_{j\in E}\theta_{z_{j}+\alpha}-i\sum_{j\in F}\theta_{z_j}\Big)\bigg|\,\le\,M(x,y),
\end{equation}
where the left-hand side is now viewed as a Steinhaus polynomial in the variables $\bar\theta$.
Next, we consider the Poisson kernel
\[P_t(\theta)=\sum_{n\in\Z}t^{|n|}e^{in\theta},\qquad t\in(-1,1),\qquad \theta\in\R/2\pi\Z,\]
and we recall that $P_t(\theta)\frac{d\theta}{2\pi}$ is a probability measure on $\R/2\pi\Z$ with
\[\int_{\R/2\pi\Z}e^{im\theta}P_t(\theta)\tfrac{d\theta}{2\pi}\,=\,t^{|m|},\qquad\text{for all $m\in\Z$.}\]
Integrating the estimate~\eqref{eq:estim-Steinhaus-pol} over $\bar\theta$ with respect to the product measure
\[\prod_{z\in R\Z^d}P_t(\theta_z)\tfrac{d\theta_z}{2\pi},\]
we deduce for all $t\in(-1,1)$,
\begin{equation}\label{eq:pol-t-markov}
\bigg|\sum_{(z_1,\ldots,z_{n-1})\in S}T_{x,y}^n(z_1,\ldots,z_{n-1})~t^{w(z_0,\ldots,z_n)}\bigg|\,\le\,M(x,y),
\end{equation}
where the left-hand side is now viewed as a polynomial in the variable $t$ and where the powers are given by
\[w(z_0,\ldots,z_n)\,:=\,\sum_{z\in R\Z^d}\big|\sharp\{j\in E:z_j+\alpha=z\}-\sharp\{j\in F:z_j=z\}\big|.\]
Note that we have
\[w(z_0,\ldots,z_n)\,\le\, \sharp E+\sharp F,\]
and that the equality $w(z_0,\ldots,z_n)= \sharp E+\sharp F$ holds if and only if
\[\{z_{j}+\alpha:j\in E\}\cap\{z_j:j\in F\}\,=\,\varnothing.\]
By the Markov brothers' inequality, we recall that the leading-order coefficient of a polynomial of degree~$D$ is bounded by $2^{D-1}$ times the maximum of the polynomial on $(-1,1)$.
Applied to the above situation~\eqref{eq:pol-t-markov}, this yields
\[\bigg|\sum_{(z_1,\ldots,z_{n-1})\in S}T_{x,y}^n(z_1,\ldots,z_{n-1})\,\mathds1_{\{z_{j}+\alpha:j\in E\}\cap\{z_j:j\in F\}=\varnothing}\bigg|\,\le\,2^{\sharp E+\sharp F-1}M(x,y).\]
By definition of $S''$, cf.~\eqref{eq:def-S'-simpl}, this precisely proves the claim~\eqref{eq:sum-S'-red}.
\end{proof}

\subsection{Proof of Theorem~\ref{th:exp-mix}}\label{sec:proof-exp-mix}
We prove the following refined version of~\eqref{eq:main1}: provided that $\delta\le\frac1K$ is small enough, we have for all $x,y\in\R^d$, $1+\delta K<q<\frac1{\delta K}$, and~$\e>0$,
\begin{equation}\label{eq:bound-todo-mix-bn-90}
\|\mathds1_{Q(x)}\bar\Bc(\nabla)\mathds1_{Q(y)}\|_{\Ld^q(\R^d)^d\to\Ld^q(\R^d)^d}\,\le\,\delta K\tfrac{q^2}{q-1}\big(\tfrac{|\!\log\e|}\e\big)^\frac{3d}\gamma\langle x-y\rangle^{\delta K+\e-3d}.
\end{equation}
By duality, it suffices to prove this result for $1+\delta K<q\le2$. For that purpose, starting from the perturbative expansion~\eqref{eq:barA}, it is then sufficient to show for all $x,y\in\R^d$, $n\ge1$, $1<q\le2$, $\e>0$, and $0<\theta\le\frac{2d}{q'}$, provided that $\theta\ll1$ is small enough,
\begin{equation}\label{eq:bound-todo-mix-bn}
\cro{P\Bb(\Kk P^\bot \Bb)^nP}_q(x,y)\,\le\,\tfrac{1}{q-1}\big(\tfrac{|\!\log\e|}\e\big)^{\frac{3d}\gamma}C^n\theta^{1-n}\langle x-y\rangle^{7\theta+\e-3d}.
\end{equation}
We shall actually prove instead the following result: for all $x,y\in \R^d$ and $n,R\ge1$, we have for all $1<q\le2$, $\e>0$, and $0<\theta\le\frac{2d}{q'}$, provided that $\theta\ll1$ is small enough,
\begin{multline}\label{eq:red-B-1}
\cro{P\Bb(\Kk P^\bot\Bb)^nP}_{q}(x,y)\,\le\,\tfrac1{q-1}C^n\theta^{1-n}\\
\times\Big(\langle\tfrac1R(x-y)\rangle^{7\theta-3d}
+n^n\e^{1-n}\big(C_0\exp(-\tfrac1{C_0}R^{\gamma})\big)^\frac12\langle\tfrac1R(x-y)\rangle^{\e-d}\Big).
\end{multline}
Choosing $R\ge1$ such that $\exp(\frac1{2C_0}R^\gamma)=\langle x-y\rangle^{2d}n^n\e^{-n}$, we indeed find that this estimate~\eqref{eq:red-B-1} implies~\eqref{eq:bound-todo-mix-bn}.
Now note that, in the case when $|x-y|_\infty\le4nR$, the estimate~\eqref{eq:red-B-1} simply follows from~\eqref{eq:compo-Rp} and from Lemma~\ref{lem:CZ} in the following form,
\[\cro{P\Bb(\Kk P^\bot \Bb)^nP}_{q}(x,y)\,\le\,\|P\Bb(\Kk P^\bot \Bb)^nP\|_{\Ld^q(\R^d)^d\to\Ld^q(\R^d)^d}\,\le\,(\tfrac{C}{q-1})^n\,\le\,\tfrac1{q-1}C^n\theta^{1-n},\]
where the last inequality follows by noting that the restriction $\theta\le\frac{2d}{q'}$ implies $\frac1{q-1}\le\frac{2d}\theta$.
Therefore, it only remains to prove~\eqref{eq:red-B-1} in the case when $|x-y|_\infty>4nR$, and we can restrict for that purpose to $x,y\in R\Z^d$.
More precisely, we shall prove the following result: for all $n,R\ge1$ and $x,y\in R\Z^d$ with $|x-y|_\infty>4nR$, we have for all $1<p\le q\le 2$ and $\e>0$, setting $\theta:=\tfrac{2d}{p'}$, provided that~$\theta\ll1$ is small enough,
\begin{multline}\label{eq:red-B-1re}
\cro{P\Bb(\Kk P^\bot\Bb)^nP}_{p,q;R}(x,y)\,\le\,C^n\theta^{2-n}\langle\tfrac1R(x-y)\rangle^{7\theta-3d}\\
+(Cn)^n\theta^{1-n}\e^{1-n}\big(C_0\exp(-\tfrac1{C_0}R^{\gamma})\big)^\frac12\langle\tfrac1R(x-y)\rangle^{\e-d}.
\end{multline}
As the left-hand side is bounded below by $\cro{P\Bb(\Kk P^\bot \Bb)^nP}_{q}(x,y)$, cf.~\eqref{eq:cro-lower}, this indeed implies the desired result~\eqref{eq:red-B-1}.

The rest of the proof is devoted to establishing~\eqref{eq:red-B-1re}.
From now on, let $n,R\ge1$ and~\mbox{$x,y\in R\Z^d$} be fixed with $|x-y|_\infty>4nR$, let $1<p\le q\le2$, $\e>0$, set $\theta:=\tfrac{2d}{p'}$,
and consider test functions $\phi_x,\phi_y\in C^\infty_c(\R^d)^d$ such that
\begin{gather}\label{eq:cond-phixy0}
\|\phi_x\|_{\Ld^{p'}_{q'}(\R^d)}=\|\phi_y\|_{\Ld^p_q(\R^d)}=1,\qquad\supp(\phi_x)\subset Q_R(x),\qquad\supp(\phi_y)\subset Q_R(y).
\end{gather}
For any subset $S\subset(R\Z^d)^{n-1}$, we shall use for convenience the following short-hand notation,
\begin{multline}\label{eq:def-TnxyS}
T_{x,y}^n(S)
\,:=\,\sum_{(z_1,\ldots,z_{n-1})\in S}~\Big\langle \phi_x,P\Bb(\Kk P^\bot\Bb\mathds1_{Q_R(z_1)})\\[-5mm]
\ldots\, (\Kk P^\bot\Bb\mathds1_{Q_R(z_{n-1})})(\Kk P^\bot\Bb)P\phi_y\Big\rangle,
\end{multline}
where we recall that the sum is always absolutely convergent, cf.~\eqref{eq:decomp-100}, and that
\[T_{x,y}^n((R\Z^d)^{n-1})\,=\,\langle\phi_x,P\Bb(\Kk P^\bot\Bb)^nP\phi_y\rangle.\]
We split the proof into three main steps.

\medskip
\step1 First reduction: we show that it suffices to prove that for all $0\le i\le j_1<i'\le n$ and~$m\ge0$ we have
\begin{multline}\label{eq:red-B-2}
|T_{x,y}^n(S_{j_1}^{m}\cap U_{i,i'})|\,\le\,
C^n\theta^{2-n}\langle\tfrac1R(x-y)\rangle^{3\theta-2d}\,(2^m)^{4\theta-d}\,\mathds1_{2^m\ge\frac1{2nR}|x-y|_\infty}\\
+(Cn)^n\theta^{1-n}\e^{1-n}\big(C_0\exp(-\tfrac1{C_0}R^{\gamma})\big)^\frac12\langle\tfrac1R(x-y)\rangle^{\e-d},
\end{multline}
and that this bound is stable under any modification of iterates of the random field $\Bb$ of the form~\eqref{eq:change-kernel},
where we have defined, setting $z_0:=x$ and $z_n:=y$ for notational simplicity,
\begin{multline*}
S_{j_1}^{m}\,:=\,\Big\{(z_1,\ldots,z_{n-1})\in(R\Z^d)^{n-1}:\max_{0\le j<n}|z_j-z_{j+1}|_\infty\le2^{m+1}R,\\
|z_{j_1}-z_{j_1+1}|_\infty>2^mR,~~\text{and}~\max_{0\le j<j_1}|z_j-z_{j+1}|_\infty\le2^mR\Big\},
\end{multline*}
and
\[U_{i,i'}\,:=\,\Big\{(z_1,\ldots,z_{n-1})\in(R\Z^d)^{n-1}:|z_{i}-z_{i'}|_\infty\le R\Big\}.\]
We start from~\eqref{eq:decomp-10} and distinguish with respect to the first interval $|z_j-z_{j+1}|_\infty$ that reaches the largest dyadic value, and we denote by $j=j_1$ the corresponding index.
Note that by the triangle inequality the condition $|x-y|_\infty>4nR$ implies
\[\max_{0\le j<n}|z_j-z_{j+1}|\,>\,4R.\] 
In terms of the above-defined disjoint index subsets $\{S_{j_1}^{m}\}_{j_1,m}$, recalling the short-hand notation~\eqref{eq:def-TnxyS}, the sum~\eqref{eq:decomp-10} becomes
\begin{equation*}
\langle \phi_x,P\Bb(\Kk P^\bot\Bb)^nP\phi_y\rangle
\,=\,\sum_{j_1=0}^{n-1}\sum_{m=0}^\infty T_{x,y}^n(S_{j_1}^{m}).
\end{equation*}
By definition of the sets $\{U_{i,i'}\}_{i,i'}$, we note that the set $(R\Z^d)^{n-1}\setminus  D_R^n(x,y)$ of irreducible paths from $x$ to $y$ is contained in
\[(R\Z^d)^{n-1}\setminus  D_R^n(x,y)\,\subset\,\bigcup_{0\le i\le j_1<i'\le n}U_{i,i'}.\]
Therefore, by the disjointness of the subsets $\{S_{j_1}^{m}\}_{j_1,m}$, the above can be bounded as follows,
\begin{multline*}
\big|\langle \phi_x,P\Bb(\Kk P^\bot\Bb)^nP\phi_y\rangle\big|
\,\le\,\sum_{j_1=0}^{n-1}\sum_{m=0}^\infty \bigg|T_{x,y}^n\bigg(S_{j_1}^{m}\cap\bigcup_{0\le i\le j_1<i'\le n}U_{i,i'}\bigg)\bigg|\\
+\sum_{(z_1,\ldots,z_n)\in D_R^n(x,y)}\big|\big\langle \phi_x,P\Bb(\Kk P^\bot\Bb\mathds1_{Q_R(z_1)})\ldots (\Kk P^\bot\Bb\mathds1_{Q_R(z_{n-1})})(\Kk P^\bot\Bb)P\phi_y\big\rangle\big|.
\end{multline*}
The last sum over reducible paths is estimated by Lemma~\ref{lem:bound-reduct},
\begin{multline}\label{eq:decomp-Sj1m-red}
\big|\langle \phi_x,P\Bb(\Kk P^\bot\Bb)^nP\phi_y\rangle\big|
\,\le\,\sum_{j_1=0}^{n-1}\sum_{m=0}^\infty \bigg|T_{x,y}^n\bigg(S_{j_1}^{m}\cap\bigcup_{0\le i\le j_1<i'\le n}U_{i,i'}\bigg)\bigg|\\
+(Cn)^n\theta^{1-n}\e^{1-n}\big(C_0\exp(-\tfrac1{C_0}R^\gamma)\big)^\frac12\langle\tfrac1R(x-y)\rangle^{\e-d}.
\end{multline}
Since the sets $\{U_{i,i'}\}_{i,i':0\le i\le j_1<i'\le n}$ are not disjoint, we start by disjointifying them, defining for all $0\le i\le j_1<i'\le n$,
\[U'_{i,i'}\,:=\,U_{i,i'}\setminus\bigg(\bigcup_{j:0\le j<i\atop j':j_1<j'\le n}U_{j,j'}~\cup\bigcup_{j':j_1<j'<i'}U_{i,j'}\bigg).\]
The modified sets $\{U'_{i,i'}\}_{i,i':0\le i\le j_1<i'\le n}$ are now pairwise disjoint, while still satisfying
\[\bigcup_{0\le i\le j_1<i'\le n}U_{i,i'}\,=\,\bigcup_{0\le i\le j_1<i'\le n}U'_{i,i'},\]
which allows to bound~\eqref{eq:decomp-Sj1m-red} as
\begin{multline}\label{eq:decomp-nmix-plop}
\big|\langle \phi_x,P\Bb(\Kk P^\bot\Bb)^nP\phi_y\rangle\big|
\,\le\,\sum_{0\le i\le j_1<i'\le n}~\sum_{m=0}^\infty |T_{x,y}^n(S_{j_1}^{m}\cap U'_{i,i'})|\\
+(Cn)^n\theta^{1-n}\e^{1-n}\big(C_0\exp(-\tfrac1{C_0}R^\gamma)\big)^\frac12\langle\tfrac1R(x-y)\rangle^{\e-d}.
\end{multline}
For all $0\le i\le j_1<i'\le n$, if the bound~\eqref{eq:red-B-2} holds and is stable under any modification of iterates of the random field $\Bb$ of the form~\eqref{eq:change-kernel}, then we can appeal to Lemma~\ref{lem:disjointif} with $S:=S_{j_1}^m \cap U_{i,i'}$
and with
\[E_1:=\{0,\ldots,i-1\},\quad
F_1:=\{j_1+1,\ldots,n\},\quad
E_2:=\{i\},\quad
F_2:=\{j_1+1,\ldots,i'-1\},\]
which implies the following corresponding bound on $U_{i,i'}'$,
\begin{multline*}
|T_{x,y}^n(S_{j_1}^{m}\cap U_{i,i'}')|\,\le\,C^n\theta^{2-n}\langle\tfrac1R(x-y)\rangle^{3\theta-2d}\,(2^m)^{4\theta-d}\,\mathds1_{2^m\ge\frac1{2nR}|x-y|_\infty}\\
+(Cn)^n\theta^{1-n}\e^{1-n}\big(C_0\exp(-\tfrac1{C_0}R^{\gamma})\big)^\frac12\langle\tfrac1R(x-y)\rangle^{\e-d}.
\end{multline*}
Combined with~\eqref{eq:decomp-nmix-plop}, this would prove~\eqref{eq:red-B-1} after evaluating the dyadic sum over $m$, provided that $5\theta<d$.
This shows that it remains to establish the bound~\eqref{eq:red-B-2} and its stability.

\medskip
\step2 Second reduction: we show that it suffices to prove that for all $0\le i\le j_1<i'\le n$, for all $0\le j_2<i$ or $i'\le j_2<n$, for all $0\le l\le j_2<l'\le n$, and all~$m\ge0$ we have
\begin{equation}\label{eq:red-B-3}
|T_{x,y}^n(S_{j_1}^m\cap R_{j_2}\cap U_{i,i'}\cap U_{l,l'})|
\,\le\,C^n\theta^{2-n}\langle\tfrac1R(x-y)\rangle^{3\theta-2d}\,(2^m)^{4\theta-d}\,\mathds1_{2^m\ge\frac1{2nR}|x-y|_\infty}.
\end{equation}
and that this bound is stable under any modification of iterates of the random field $\Bb$ of the form~\eqref{eq:change-kernel},
where we have now further defined
\begin{multline*}
R_{j_2}\,:=\,\Big\{(z_1,\ldots,z_{n-1})\in(R\Z^d)^{n-1}:\max_{0\le j<j_2}|z_j-z_{j+1}|_\infty\le \tfrac1n|x-y|_{\infty},\\
~\text{and}~~|z_{j_2}-z_{j_2+1}|_\infty>\tfrac1n|x-y|_{\infty}\Big\}.
\end{multline*}
Recalling the condition $|x-y|_\infty>4nR$,
we note that for all $(z_1,\ldots,z_{n-1})\in S_{j_1}^m\cap U_{i,i'}$ the triangle inequality yields
\begin{eqnarray*}
(n-1)\max\Big\{\max_{0\le j<i}|z_j-z_{j+1}|_\infty\,,\,\max_{i'\le j<n}|z_j-z_{j+1}|_\infty\Big\}
&\ge&|x-y|_\infty-|z_i-z_{i'}|_\infty\\
&\ge&|x-y|_\infty-R\\
&>&(1-\tfrac1{4n})|x-y|_\infty,
\end{eqnarray*}
hence
\[\max\Big\{\max_{0\le j<i}|z_j-z_{j+1}|_\infty\,,\,\max_{i'\le j<n}|z_j-z_{j+1}|_\infty\Big\}\,>\,\tfrac{1}{n}|x-y|_\infty.\]
We denote by $j=j_2$ the first index realizing this inequality. In terms of the above-defined disjoint index subsets $\{R_{j_2}\}_{j_2}$, we can decompose
\[T_{x,y}^n(S_{j_1}^m\cap U_{i,i'})\,=\,\sum_{j_2:0\le j_2<i}T_{x,y}^n(S_{j_1}^m\cap R_{j_2}\cap U_{i,i'})+\sum_{j_2:i'\le j_2<n}T_{x,y}^n(S_{j_1}^m\cap R_{j_2}\cap U_{i,i'}).\]
Next, arguing as in Step~1, we appeal to Lemma~\ref{lem:bound-reduct} to restrict index sets to
\[\bigcup_{l,l':0\le l\le j_2<l'\le n}U_{l,l'},\]
which amounts to neglecting further reducible paths, and we appeal to Lemma~\ref{lem:disjointif} to disjointify these sets $\{U_{l,l'}\}_{l,l':0\le l\le j_2<l'\le n}$.
As in Step~1, we deduce in this way that it indeed remains to establish the bound~\eqref{eq:red-B-3} and its stability.

\medskip
\step3 Conclusion.\\
It remains to prove~\eqref{eq:red-B-3} and its stability.
Let $0\le i\le j_1<i'\le n$ be fixed. Without loss of generality, we can focus on the case $0\le j_2<i$ (the other case $i'\le j_2<n$ is symmetric). Now let $l,l'$ with $0\le l\le j_2<l'\le n$ and we split the analysis of $T_{x,y}^n(S_{j_1}^m\cap R_{j_2}\cap U_{i,i'}\cap U_{l,l'})$ by considering six different cases:
\begin{enumerate}[---]
\item case~1: $0\le l\le j_2<l'< i\le j_1<i'\le n$;
\item case~2: $0\le l\le j_2< i=l'\le j_1<i'\le n$;
\item case~3: $0\le l\le j_2< i<l'\le j_1<i'\le n$;
\item case~4: $0\le l\le j_2< i\le j_1<l'<i'\le n$;
\item case~5: $0\le l\le j_2< i\le j_1<i'=l'< n$;
\item case~6: $0\le l\le j_2< i\le j_1<i'<l'\le n$.
\end{enumerate}
To ease readability, it might be useful to introduce diagrammatic representations for the different cases, namely
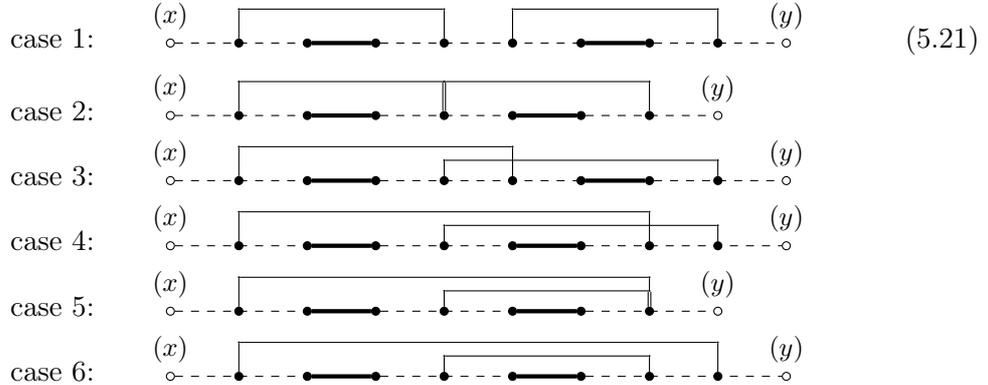
\begingroup\allowdisplaybreaks
\begin{eqnarray}
\label{eq:diagrams}
\text{case~1:}
&&{\small\begin{tikzpicture}[scale=0.9]
\begin{scope}[every node/.style={circle,draw,fill=white,inner sep=0pt,minimum size=3pt}]
     \node[label=above:$(x)$] (A) at (0,0) {};
     \node[label=above:$(y)$] (J) at (9,0) {};
\end{scope}
\begin{scope}[every node/.style={circle,fill,draw,inner sep=0pt,minimum size=3pt}]
     \node (B) at (1,0) {};
     \node (C) at (2,0) {};
     \node (D) at (3,0) {};
     \node (E) at (4,0) {};
     \node (F) at (5,0) {};
     \node (G) at (6,0) {};
     \node (H) at (7,0) {};
     \node (I) at (8,0) {};
\end{scope}
\begin{scope}[>={Stealth[black]},
every edge/.style={draw=black,dashed}]
    \path [-] (A) edge node[above,pos=0.5] {} (C);
    \path [-] (D) edge node[above,pos=0.5] {} (G);
    \path [-] (H) edge node[above,pos=0.5] {} (J);
\end{scope}
\begin{scope}[>={Stealth[black]},
every edge/.style={draw=black,ultra thick}]
    \path [-] (C) edge node[above,pos=0.5] {} (D);
    \path [-] (G) edge node[above,pos=0.5] {} (H);
\end{scope}
\begin{scope}[every node/.style={circle,fill,draw,inner sep=0pt,minimum size=0.3pt}]
    \node (B') at (1,0.5) {};
    \node (E') at (4,0.5) {};
    \node (F') at (5,0.5) {};
    \node (I') at (8,0.5) {};
\end{scope}
\begin{scope}[>={Stealth[black]},
every node/.style={fill=white,circle},
every edge/.style={draw=black}]
    \path [-] (B) edge (B');
    \path [-] (B') edge (E');
    \path [-] (E') edge (E);
    \path [-] (F) edge (F');
    \path [-] (F') edge (I');
    \path [-] (I') edge (I);
\end{scope}
\end{tikzpicture}}\\
\text{case~2:}
&&{\small\begin{tikzpicture}[scale=0.9]
\begin{scope}[every node/.style={circle,draw,fill=white,inner sep=0pt,minimum size=3pt}]
     \node[label=above:$(x)$] (A) at (0,0) {};
     \node[label=above:$(y)$] (J) at (8,0) {};
\end{scope}
\begin{scope}[every node/.style={circle,fill,draw,inner sep=0pt,minimum size=3pt}]
     \node (B) at (1,0) {};
     \node (C) at (2,0) {};
     \node (D) at (3,0) {};
     \node (E) at (4,0) {};
     \node (G) at (5,0) {};
     \node (H) at (6,0) {};
     \node (I) at (7,0) {};
\end{scope}
\begin{scope}[>={Stealth[black]},
every edge/.style={draw=black,dashed}]
    \path [-] (A) edge node[above,pos=0.5] {} (C);
    \path [-] (D) edge node[above,pos=0.5] {} (G);
    \path [-] (H) edge node[above,pos=0.5] {} (J);
\end{scope}
\begin{scope}[>={Stealth[black]},
every edge/.style={draw=black,ultra thick}]
    \path [-] (C) edge node[above,pos=0.5] {} (D);
    \path [-] (G) edge node[above,pos=0.5] {} (H);
\end{scope}
\begin{scope}[every node/.style={circle,fill,draw,inner sep=0pt,minimum size=0.3pt}]
    \node (B') at (1,0.5) {};
    \node (E') at (4,0.5) {};
    \node (F') at (4,0.5) {};
    \node (I') at (7,0.5) {};
\end{scope}
\begin{scope}[>={Stealth[black]},
every node/.style={fill=white,circle},
every edge/.style={draw=black}]
    \path [-] (B) edge (B');
    \path [-] (B') edge (E');
    \path [-] (E') edge (I');
    \path [-] (I') edge (I);
\end{scope}
\begin{scope}[>={Stealth[black]},
every node/.style={fill=white,circle},
every edge/.style={draw=black,double}]
    \path [-] (E') edge (E);
\end{scope}
\end{tikzpicture}}
\nonumber\\
\text{case~3:}
&&{\small\begin{tikzpicture}[scale=0.9]
\begin{scope}[every node/.style={circle,draw,fill=white,inner sep=0pt,minimum size=3pt}]
     \node[label=above:$(x)$] (A) at (0,0) {};
     \node[label=above:$(y)$] (J) at (9,0) {};
\end{scope}
\begin{scope}[every node/.style={circle,fill,draw,inner sep=0pt,minimum size=3pt}]
     \node (B) at (1,0) {};
     \node (C) at (2,0) {};
     \node (D) at (3,0) {};
     \node (E) at (4,0) {};
     \node (F) at (5,0) {};
     \node (G) at (6,0) {};
     \node (H) at (7,0) {};
     \node (I) at (8,0) {};
\end{scope}
\begin{scope}[>={Stealth[black]},
every edge/.style={draw=black,dashed}]
    \path [-] (A) edge node[above,pos=0.5] {} (C);
    \path [-] (D) edge node[above,pos=0.5] {} (G);
    \path [-] (H) edge node[above,pos=0.5] {} (J);
\end{scope}
\begin{scope}[>={Stealth[black]},
every edge/.style={draw=black,ultra thick}]
    \path [-] (C) edge node[above,pos=0.5] {} (D);
    \path [-] (G) edge node[above,pos=0.5] {} (H);
\end{scope}
\begin{scope}[every node/.style={circle,fill,draw,inner sep=0pt,minimum size=0.3pt}]
    \node (B') at (1,0.5) {};
    \node (E') at (4,0.3) {};
    \node (F') at (5,0.5) {};
    \node (I') at (8,0.3) {};
\end{scope}
\begin{scope}[>={Stealth[black]},
every node/.style={fill=white,circle},
every edge/.style={draw=black}]
    \path [-] (B) edge (B');
    \path [-] (B') edge (F');
    \path [-] (F') edge (F);
    \path [-] (E) edge (E');
    \path [-] (E') edge (I');
    \path [-] (I') edge (I);
\end{scope}
\end{tikzpicture}}
\nonumber\\
\text{case~4:}
&&{\small\begin{tikzpicture}[scale=0.9]
\begin{scope}[every node/.style={circle,draw,fill=white,inner sep=0pt,minimum size=3pt}]
     \node[label=above:$(x)$] (A) at (0,0) {};
     \node[label=above:$(y)$] (J) at (9,0) {};
\end{scope}
\begin{scope}[every node/.style={circle,fill,draw,inner sep=0pt,minimum size=3pt}]
     \node (B) at (1,0) {};
     \node (C) at (2,0) {};
     \node (D) at (3,0) {};
     \node (E) at (4,0) {};
     \node (F) at (5,0) {};
     \node (G) at (6,0) {};
     \node (H) at (7,0) {};
     \node (I) at (8,0) {};
\end{scope}
\begin{scope}[>={Stealth[black]},
every edge/.style={draw=black,dashed}]
    \path [-] (A) edge node[above,pos=0.5] {} (C);
    \path [-] (D) edge node[above,pos=0.5] {} (F);
    \path [-] (G) edge node[above,pos=0.5] {} (J);
\end{scope}
\begin{scope}[>={Stealth[black]},
every edge/.style={draw=black,ultra thick}]
    \path [-] (C) edge node[above,pos=0.5] {} (D);
    \path [-] (F) edge node[above,pos=0.5] {} (G);
\end{scope}
\begin{scope}[every node/.style={circle,fill,draw,inner sep=0pt,minimum size=0.3pt}]
    \node (B') at (1,0.5) {};
    \node (H') at (7,0.5) {};
    \node (E') at (4,0.3) {};
    \node (I') at (8,0.3) {};
\end{scope}
\begin{scope}[>={Stealth[black]},
every node/.style={fill=white,circle},
every edge/.style={draw=black}]
    \path [-] (B) edge (B');
    \path [-] (B') edge (H');
    \path [-] (H') edge (H);
    \path [-] (E) edge (E');
    \path [-] (E') edge (I');
    \path [-] (I') edge (I);
\end{scope}
\end{tikzpicture}}
\nonumber\\
\text{case~5:}
&&{\small\begin{tikzpicture}[scale=0.9]
\begin{scope}[every node/.style={circle,draw,fill=white,inner sep=0pt,minimum size=3pt}]
     \node[label=above:$(x)$] (A) at (0,0) {};
     \node[label=above:$(y)$] (J) at (8,0) {};
\end{scope}
\begin{scope}[every node/.style={circle,fill,draw,inner sep=0pt,minimum size=3pt}]
     \node (B) at (1,0) {};
     \node (C) at (2,0) {};
     \node (D) at (3,0) {};
     \node (E) at (4,0) {};
     \node (F) at (5,0) {};
     \node (G) at (6,0) {};
     \node (H) at (7,0) {};
\end{scope}
\begin{scope}[>={Stealth[black]},
every edge/.style={draw=black,dashed}]
    \path [-] (A) edge node[above,pos=0.5] {} (C);
    \path [-] (D) edge node[above,pos=0.5] {} (F);
    \path [-] (G) edge node[above,pos=0.5] {} (J);
\end{scope}
\begin{scope}[>={Stealth[black]},
every edge/.style={draw=black,ultra thick}]
    \path [-] (C) edge node[above,pos=0.5] {} (D);
    \path [-] (F) edge node[above,pos=0.5] {} (G);
\end{scope}
\begin{scope}[every node/.style={circle,fill,draw,inner sep=0pt,minimum size=0.3pt}]
    \node (B') at (1,0.5) {};
    \node (H') at (7,0.5) {};
    \node (E') at (4,0.3) {};
    \node (I') at (7,0.3) {};
\end{scope}
\begin{scope}[>={Stealth[black]},
every node/.style={fill=white,circle},
every edge/.style={draw=black,double}]
    \path [-] (I') edge (H);
\end{scope}
\begin{scope}[>={Stealth[black]},
every node/.style={fill=white,circle},
every edge/.style={draw=black}]
    \path [-] (B) edge (B');
    \path [-] (B') edge (H');
    \path [-] (E) edge (E');
    \path [-] (H') edge (I');
    \path [-] (I') edge (E');
\end{scope}
\end{tikzpicture}}
\nonumber\\
\text{case~6:}
&&{\small\begin{tikzpicture}[scale=0.9]
\begin{scope}[every node/.style={circle,draw,fill=white,inner sep=0pt,minimum size=3pt}]
     \node[label=above:$(x)$] (A) at (0,0) {};
     \node[label=above:$(y)$] (J) at (9,0) {};
\end{scope}
\begin{scope}[every node/.style={circle,fill,draw,inner sep=0pt,minimum size=3pt}]
     \node (B) at (1,0) {};
     \node (C) at (2,0) {};
     \node (D) at (3,0) {};
     \node (E) at (4,0) {};
     \node (F) at (5,0) {};
     \node (G) at (6,0) {};
     \node (I) at (7,0) {};
     \node (H) at (8,0) {};
\end{scope}
\begin{scope}[>={Stealth[black]},
every edge/.style={draw=black,dashed}]
    \path [-] (A) edge node[above,pos=0.5] {} (C);
    \path [-] (D) edge node[above,pos=0.5] {} (F);
    \path [-] (G) edge node[above,pos=0.5] {} (J);
\end{scope}
\begin{scope}[>={Stealth[black]},
every edge/.style={draw=black,ultra thick}]
    \path [-] (C) edge node[above,pos=0.5] {} (D);
    \path [-] (F) edge node[above,pos=0.5] {} (G);
\end{scope}
\begin{scope}[every node/.style={circle,fill,draw,inner sep=0pt,minimum size=0.3pt}]
    \node (B') at (1,0.5) {};
    \node (I') at (7,0.3) {};
    \node (E') at (4,0.3) {};
    \node (H') at (8,0.5) {};
\end{scope}
\begin{scope}[>={Stealth[black]},
every node/.style={fill=white,circle},
every edge/.style={draw=black}]
    \path [-] (B) edge (B');
    \path [-] (B') edge (H');
    \path [-] (H') edge (H);
    \path [-] (E) edge (E');
    \path [-] (E') edge (I');
    \path [-] (I') edge (I);
\end{scope}
\end{tikzpicture}}
\nonumber
\end{eqnarray}
\endgroup
where each dotted main line represents a path $(z_0,\ldots,z_n)$ from $z_0=x$ to $z_n=y$,
where bold segments on the line represent long intervals $|z_j-z_{j+1}|_\infty>R$, and where edges above the main line indicate `coincidence' points $z_i,z_{i'}$ with $|z_i-z_{i'}|_\infty\le R$.
For shortness, we focus on cases~1 and~4: indeed, cases~2 and~3 (resp.\@ cases~5 and~6) are actually similar to case~1 (resp.\@ case~4). We split the proof into two further substeps.

\medskip
\substep{3.1} Case~1.\\
Let $0\le l\le j_2<l'< i\le j_1<i'\le n$ be fixed.
By definition of $S_{j_1}^m,R_{j_2},U_{i,i'},U_{l,l'}$, setting $\ell:=\tfrac1n|x-y|_\infty$ for shortness, and arguing as in~\eqref{eq:use-Sj0R} to replace restricted summations over $S_{j_1}^m$ and $R_{j_2}$ by kernel truncations,
we can write
\begin{multline*}
T_{x,y}^n(S_{j_1}^m\cap R_{j_2}\cap U_{i,i'}\cap U_{l,l'})\\
\,=\,\sum_{z_{j_1},z_{j_1+1},z_{j_2},z_{j_2+1},z_i,z_{i'},z_l,z_{l'}\in R\Z^d}\mathds1_{2^mR<|z_{j_1}-z_{j_1+1}|_\infty\le2^{m+1}R}\\
\times\,\mathds1_{|z_{j_2}-z_{j_2+1}|_\infty>\frac1n|x-y|_\infty}\,\mathds1_{|z_i-z_{i'}|_\infty\le R}\,\mathds1_{|z_l-z_{l'}|_\infty\le R}\\
\times\Big\langle \phi_x,P\Bb(\Kk_{(2^mR)\wedge\ell;R} P^\bot\Bb)^{l}\mathds1_{Q_R(z_l)}(\Kk_{(2^mR)\wedge\ell;R} P^\bot\Bb)^{j_2-l}\mathds1_{Q_R(z_{j_2})}\Kk_{2^mR;R} P^\bot\Bb\mathds1_{Q_R(z_{j_2+1})}\\
\times(\Kk_{2^mR;R} P^\bot\Bb)^{l'-j_2-1}\mathds1_{Q_R(z_{l'})}
(\Kk_{2^mR;R} P^\bot\Bb)^{i-l'}\mathds1_{Q_R(z_{i})}
(\Kk_{2^mR;R} P^\bot\Bb)^{j_1-i}\mathds1_{Q_R(z_{j_1})}\\
\times\Kk P^\bot\Bb\mathds1_{Q_R(z_{j_1+1})}
(\Kk_{2^{m+1}R;R} P^\bot\Bb)^{i'-j_1-1}\mathds1_{Q_R(z_{i'})}
(\Kk_{2^{m+1}R;R} P^\bot\Bb)^{n-i'}P\phi_y\Big\rangle.
\end{multline*}
Taking the supremum over $\phi_x,\phi_y$, and using~\eqref{eq:compo-Rp} and the pointwise decay of $\Kk$, we find
\begin{multline*}
\sup_{\phi_x,\phi_y}|T_{x,y}^n(S_{j_1}^m\cap R_{j_2}\cap U_{i,i'}\cap U_{l,l'})|
\,\lesssim\,n^d\langle\tfrac1R(x-y)\rangle^{-d}(2^m )^{-d}\\
\times\sum_{z_{j_1},z_{j_1+1},z_{j_2},z_{j_2+1},z_i,z_{i'},z_l,z_{l'}\in R\Z^d}\mathds1_{|z_i-z_{i'}|_\infty\le R}\,\mathds1_{|z_l-z_{l'}|_\infty\le R}
\,\cro{(\Kk_{(2^mR)\wedge\ell;R} P^\bot\Bb)^{l}}_{p,q,2;R}(x,z_l)\\
\times\cro{(\Kk_{(2^mR)\wedge\ell;R} P^\bot\Bb)^{j_2-l}}_{p,q,2;R}(z_l,z_{j_2})
\cro{(\Kk_{2^mR;R} P^\bot\Bb)^{l'-j_2-1}}_{p,q,2;R}(z_{j_2+1},z_{l'})\\
\times\cro{(\Kk_{2^mR;R} P^\bot\Bb)^{i-l'}}_{p,q,2;R}(z_{l'},z_i)
\cro{(\Kk_{2^mR;R} P^\bot\Bb)^{j_1-i}}_{p,q,2;R}(z_i,z_{j_1})\\
\times \cro{(\Kk_{2^{m+1}R;R} P^\bot\Bb)^{i'-j_1-1}}_{p,q,2;R}(z_{j_1+1},z_{i'})
\cro{(\Kk_{2^{m+1}R;R} P^\bot\Bb)^{n-i'}}_{p,q,2;R}(z_{i'},y),
\end{multline*}
where the supremum is implicitly understood to run over all functions $\phi_x,\phi_y$ satisfying~\eqref{eq:cond-phixy0}.
By the triangle inequality, we note that kernel truncations entail that the summation is restricted to
\begin{gather*}
|z_{j_1}-z_i|_\infty,|z_{j_1+1}-z_{i'}|_\infty,|z_{j_2}-z_l|_\infty,|z_{j_2+1}-z_{l'}|_\infty\,\le\, n2^{m+1}R,\\
2^m\ge\tfrac1{2nR}|x-y|_\infty.
\end{gather*}
Further appealing to the deterministic estimates of Lemma~\ref{lem:deterministic}, we are then led to  
\begin{multline}\label{eq:estim-kernels-case1}
\sup_{\phi_x,\phi_y}|T_{x,y}^n(S_{j_1}^m\cap R_{j_2}\cap U_{i,i'}\cap U_{l,l'})|
\,\le\,C^n\theta^{2-n}\langle\tfrac1R(x-y)\rangle^{-d}\,(2^m)^{-d}\mathds1_{2^m\ge\frac1{2nR}|x-y|_\infty}\\
\times\sum_{z_{j_1},z_{j_1+1},z_{j_2},z_{j_2+1},z_i,z_{i'},z_l,z_{l'}\in R\Z^d}\mathds1_{|z_i-z_{i'}|_\infty\le R}\,\mathds1_{|z_l-z_{l'}|_\infty\le R}\\
\times\mathds1_{|z_{j_1}-z_i|_\infty,|z_{j_1+1}-z_{i'}|_\infty,|z_{j_2}-z_l|_\infty,|z_{j_2+1}-z_{l'}|_\infty\,\le\, n2^{m+1}R}\\
\times\langle\tfrac1R(x-z_l)\rangle^{\theta-d}
\langle\tfrac1R(z_l-z_{j_2})\rangle^{\theta-d}
\langle\tfrac1R(z_{j_2+1}-z_{l'})\rangle^{\theta-d}
\langle\tfrac1R(z_{l'}-z_i)\rangle^{\theta-d}\\
\times\langle\tfrac1R(z_i-z_{j_1})\rangle^{\theta-d}
\langle\tfrac1R(z_{j_1+1}-z_{i'})\rangle^{\theta-d}
\langle\tfrac1R(z_{i'}-y)\rangle^{\theta-d}.
\end{multline}
Evaluating the sums over $z_{j_1},z_{j_1+1},z_{j_2},z_{j_2+1}$, and using that the summations over $(z_{i},z_{i'})$ and over $(z_l,z_{l'})$ are restricted to pairs of neighboring vertices in $R\Z^d$, we deduce
\begin{multline*}
\sup_{\phi_x,\phi_y}|T_{x,y}^n(S_{j_1}^m\cap R_{j_2}\cap U_{i,i'}\cap U_{l,l'})|
\,\le\,C^n\theta^{2-n}\langle\tfrac1R(x-y)\rangle^{-d}\,(2^m)^{4\theta-d}\,\mathds1_{2^m\ge\frac1{2nR}|x-y|_\infty}\\
\times\sum_{z_i,z_l\in R\Z^d}
\langle\tfrac1R(x-z_l)\rangle^{\theta-d}
\langle\tfrac1R(z_{l}-z_i)\rangle^{\theta-d}
\langle\tfrac1R(z_{i}-y)\rangle^{\theta-d},
\end{multline*}
and thus, estimating the remaining sums, provided that $4\theta<d$,
\begin{multline*}
\sup_{\phi_x,\phi_y}|T_{x,y}^n(S_{j_1}^m\cap R_{j_2}\cap U_{i,i'}\cap U_{l,l'})|
\,\le\,C^n\theta^{2-n}\langle\tfrac1R(x-y)\rangle^{3\theta-2d}\,(2^m)^{4\theta-d}\,\mathds1_{2^m\ge\frac1{2nR}|x-y|_\infty}.
\end{multline*}
The same estimate holds in cases~2 and~3.

\medskip
\substep{3.2} Case~4.\\
Let $0\le l\le j_2<i\le j_1<l'<i'\le n$ be fixed. Arguing as above, we then find, instead of~\eqref{eq:estim-kernels-case1},
\begin{multline*}
T_{x,y}^n(S_{j_1}^m\cap R_{j_2}\cap U_{i,i'}\cap U_{l,l'})
\,\lesssim\,C^n\theta^{2-n}\langle\tfrac1R(x-y)\rangle^{-d}(2^m)^{-d}\\
\times\sum_{z_{j_1},z_{j_1+1},z_{j_2},z_{j_2+1},z_i,z_{i'},z_l,z_{l'}\in R\Z^d}\mathds1_{|z_i-z_{i'}|_\infty\le R}\,\mathds1_{|z_l-z_{l'}|_\infty\le R}\\
\times\mathds1_{|z_{j_1}-z_i|_\infty,|z_{j_1+1}-z_{l'}|_\infty,|z_{j_2}-z_l|_\infty,|z_{j_2+1}-z_{i}|_\infty\,\le\, n2^{m+1}R}\\
\times\langle\tfrac1R(x-z_l)\rangle^{\theta-d}\langle\tfrac1R(z_l-z_{j_2})\rangle^{\theta-d}
\langle\tfrac1R(z_{j_2+1}-z_i)\rangle^{\theta-d}\langle\tfrac1R(z_i-z_{j_1})\rangle^{\theta-d}\\
\times\langle\tfrac1R(z_{j_1+1}-z_{l'})\rangle^{\theta-d}
\langle\tfrac1R(z_{l'}-z_{i'})\rangle^{\theta-d}\langle\tfrac1R(z_{i'}-y)\rangle^{\theta-d}.
\end{multline*}
Evaluating the sums over $z_{j_1},z_{j_1+1},z_{j_2},z_{j_2+1}$, and using that the summations over $(z_{i},z_{i'})$ and over $(z_l,z_{l'})$ are restricted to pairs of neighboring vertices in $R\Z^d$, we deduce again
\begin{multline*}
T_{x,y}^n(S_{j_1}^m\cap R_{j_2}\cap U_{i,i'}\cap U_{l,l'})
\,\lesssim\,C^n\theta^{2-n}\langle\tfrac1R(x-y)\rangle^{-d}(2^m)^{4\theta-d}\\
\times\sum_{z_i,z_l\in R\Z^d}
\langle\tfrac1R(x-z_l)\rangle^{\theta-d}
\langle\tfrac1R(z_{l}-z_{i})\rangle^{\theta-d}\langle\tfrac1R(z_{i}-y)\rangle^{\theta-d},
\end{multline*}
and thus, estimating the remaining sums, provided that $4\theta<d$,
\begin{multline*}
\sup_{\phi_x,\phi_y}|T_{x,y}^n(S_{j_1}^m\cap R_{j_2}\cap U_{i,i'}\cap U_{l,l'})|
\,\le\,C^n\theta^{2-n}\langle\tfrac1R(x-y)\rangle^{3\theta-2d}\,(2^m)^{4\theta-d}\,\mathds1_{2^m\ge\frac1{2nR}|x-y|_\infty}.
\end{multline*}
The same estimate holds in cases~5 and~6.
Combined with the result of Substep~3.1, this concludes the proof of~\eqref{eq:red-B-3}. As required, we note that the proof is clearly stable under any modification of iterates of the random field $\Bb$ of the form~\eqref{eq:change-kernel}. This concludes the proof of~\eqref{eq:bound-todo-mix-bn-90}, and the main decay estimate~\eqref{eq:main1} in Theorem~\ref{th:exp-mix} follows.\qed

\subsection{Proof of H\"older continuity}\label{ssect:hoelder}
We claim in Theorem \ref{th:exp-mix} that the decay estimate~\eqref{eq:main1} for the kernel of $\bar\Bc(\nabla)$ implies that its symbol is globally of H\"older class $C_b^{2d-\delta K-}$. Since the decay estimate~\eqref{eq:main1} does not hold pointwise on the kernel, this is not completely standard and we include a short proof. Derivatives of the symbol can be expressed through the convolution kernel as
\begin{eqnarray}
\nabla_\xi^\alpha\bar\Bc(i\xi )
&=&\int_{\R^d} (-ix)^\alpha e^{-ix\cdot\xi} \,\bar{\Bc}(\nabla)(x)\,\frac{dx}{(2\pi)^{d/2}}\nonumber\\
&=&\frac{(-i)^{|\alpha|}}{(2\pi)^{d/2}}\sum_{z\in\Z^d}\int_{Q(z)} \int_{Q(0)} e^{-i(x-y)\cdot\xi}(x-y)^{\alpha}\bar\Bc(\nabla)(x-y)\, dxdy.
\label{eq:decomp-der-B-kernel}
\end{eqnarray}
Using the binomial theorem to expand $(x-y)^\alpha$, this can be decomposed as follows,
\begin{equation*}
\nabla_\xi^\alpha\bar\Bc(i\xi)
\,=\,\frac{(-i)^{|\alpha|}}{(2\pi)^{d/2}}\sum_{\beta\le\alpha}(-1)^{|\beta|}\binom\alpha\beta\sum_{z\in\Z^d}\big\langle {\phi_{\xi}^{\alpha-\beta}},\mathds1_{Q(z)}\bar\Bc(\nabla)\mathds1_{Q(0)}\,\phi_{\xi}^{\beta}\big\rangle,
\end{equation*}
in terms of $\phi_{\xi}^\beta(x):=e^{ix\cdot\xi}x^\beta$,
and we are thus led to
\begin{equation*}
|\nabla_\xi^\alpha\bar\Bc(i\xi )|
\,\lesssim_\alpha\,\sum_{z\in\Z^d}\langle z\rangle^{|\alpha|}\|\mathds1_{Q(z)}\bar\Bc(\nabla)\mathds1_{Q(0)}\|_{\Ld^2(\R^d)^d\to\Ld^2(\R^d)^d}.
\end{equation*}
Now appealing to the decay estimate~\eqref{eq:main1}, we obtain
\begin{equation*}
|\nabla_\xi^\alpha\bar\Bc(i\xi )|
\,\lesssim\,\delta K\sum_{z\in\Z^d}\log(2+|z|)^K\langle z\rangle^{\delta K+|\alpha|-3d}.
\end{equation*}
This sum is convergent provided $|\alpha|\le2d-1<2d-\delta K$, which proves $\bar\Bc\in C_b^{2d-1}(i\R^d)$.
It remains to show that for $|\alpha|=2d-1$ the derivative $\nabla_\xi^\alpha\bar\Bc$ is H\"older continuous with any exponent~$<1-\delta K$.
Starting from~\eqref{eq:decomp-der-B-kernel}, we can write for all $\xi,w\in\R^d$,
\begin{multline*}
\nabla_\xi^\alpha\bar\Bc(i\xi )-\nabla_\xi^\alpha\bar\Bc(i(\xi+w))\\
\,=\,\frac{(-i)^{|\alpha|}}{(2\pi)^{d/2}}\sum_{z\in\Z^d}\int_{Q(z)}\int_{Q(0)} e^{-i(x-y)\cdot\xi}(1-e^{-i(x-y)\cdot w})\, (x-y)^\alpha \,\bar\Bc(\nabla)(x-y)\, dxdy.
\end{multline*}
Decomposing $1-e^{-i(x-y)\cdot w}=(1-e^{iy\cdot w})+(1-e^{-ix\cdot w})e^{iy\cdot w}$, and using again the binomial theorem to expand $(x-y)^\alpha$, we get
\begin{multline*}
\nabla_\xi^\alpha\bar\Bc(i\xi )-\nabla_\xi^\alpha\bar\Bc(i(\xi+w))
\,=\,\frac{(-i)^{|\alpha|}}{(2\pi)^{d/2}}\sum_{\beta\le\alpha}(-1)^{|\beta|}\binom\alpha\beta\sum_{z\in\Z^d}\langle\phi_{\xi}^{\alpha-\beta},\mathds1_{Q(z)}\bar\Bc(\nabla)\mathds1_{Q(0)}\phi^\beta_{\xi}\rho_w\rangle\\
+\frac{(-i)^{|\alpha|}}{(2\pi)^{d/2}}\sum_{\beta\le\alpha}(-1)^{|\beta|}\binom\alpha\beta\sum_{z\in\Z^d}\langle\phi^{\alpha-\beta}_{\xi}\rho_w,\mathds1_{Q(z)}\bar\Bc(\nabla)\mathds1_{Q(0)}\phi^\beta_{\xi+w}\rangle,
\end{multline*}
where we have further set $\rho_w(x):=1-e^{ix\cdot w}$.
For $0<s<1$, noting that $|\rho_w(x)|\le|x|^s|w|^s$, we deduce
\begin{equation*}
\frac{|\nabla_\xi^\alpha\bar\Bc(i\xi )-\nabla_\xi^\alpha\bar\Bc(i(\xi+w))|}{|w|^s}
\,\lesssim_\alpha\,\sum_{z\in\Z^d}\langle z\rangle^{|\alpha|+s}\|\mathds1_{Q(z)}\bar\Bc(\nabla)\mathds1_{Q(0)}\|_{\Ld^2(\R^d)^d\to\Ld^2(\R^d)^d}.
\end{equation*}
Now appealing again to the decay estimate~\eqref{eq:main1}, we find that this sum is convergent provided $|\alpha|+s<2d-\delta K$, which concludes the proof of H\"older continuity.
\qed

\section{Correlated Gaussian setting}
\label{sect:correlatedGaussianproof}
This section is devoted to the proof of Theorem~\ref{th:correl} in the Gaussian setting of Assumption~{\bf(H$_2$)}.
Recall that we consider the perturbative regime~\eqref{eq:decomp-AB}, and more precisely, setting $\delta:=\|A_0-\Id\|_{C^2_b(\R^\kappa)}\ll1$, the representation~\eqref{eq:rep-AA_0} becomes
\begin{equation}\label{eq:rep-AA_0reB}
\Aa\,=\,\Id+\delta \Bb,\qquad\Bb(x,\omega)\,=\,B_0(G(x,\omega)),\qquad\expec{\Bb}=0,\qquad\|B_0\|_{C^2_b(\R^\kappa)}=1.
\end{equation}
We start by recalling useful notions from Malliavin calculus with respect to the underlying Gaussian field $G$.
Using Malliavin calculus, we shall see that the proof in the previous section gets drastically reduced and does no longer require any use of Bourgain's disjointification lemma.
Note that a corresponding stochastic calculus could be used to similarly reduce the proof in the discrete i.i.d.\@ setting of Bourgain's original result~\cite{Bourgain-18,Lemm-18}.

\subsection{Malliavin calculus}
Since the covariance function $c$ is uniformly bounded, cf.~\eqref{eq:decay-c0}, the Gaussian random field $G$ can be viewed as a random Schwartz distribution: for all $\zeta_1,\zeta_2\in C^\infty_c(\R^d)^\kappa$, we define $G(\zeta_1)=\int_{\R^d}G\cdot\zeta_1$ and $G(\zeta_2)=\int_{\R^d}G\cdot\zeta_2$ as centered Gaussian random variables with covariance
\[\cov{G(\zeta_1)}{G(\zeta_2)}\,=\,\iint_{\R^d\times\R^d}\zeta_1(x)\cdot c(x-y)\zeta_2(y)\,dxdy.\]
We define $\Hf$ as the closure of $C_c^\infty(\R^d)^\kappa$ for the seminorm
\[\|\zeta_1\|_\Hf^2\,:=\,\langle\zeta_1,\zeta_1\rangle_\Hf,\qquad\langle\zeta_1,\zeta_2\rangle_\Hf\,:=\,\iint_{\R^d\times\R^d} \zeta_1(x)\cdot c(x-y)\zeta_2(y)\,dxdy.\]
Taking the quotient with respect to the kernel of $\|\cdot\|_\Hf$, the space $\Hf$ is a separable Hilbert space.
We recall some basic definitions of the Malliavin calculus with respect to the Gaussian field $G$ (see e.g.~\cite{NP-book} for details). Without loss of generality, we work under the assumption that the probability space $(\Omega,\Pm)$ is endowed with the $\sigma$-algebra generated by $G$, which ensures that the linear subspace
\[\Rc\,:=\,\Big\{g\big(G(\zeta_1),\ldots,G(\zeta_n)\big)\,:\,n\in\N,\,g\in C_c^\infty(\R^n),\,\zeta_1,\ldots,\zeta_n\in C_c^\infty(\R^d)^\kappa\Big\}\,\subset\,\Ld^2(\Omega)\]
is dense in $\Ld^2(\Omega)$.
For a random variable $X$ in this model subspace, say of the form $X=g(G(\zeta_1),\ldots,G(\zeta_n))$, we define its {Malliavin derivative} $DX\in \Ld^2(\Omega;\Hf)$ as
\begin{align}\label{eq:D-expl}
DX\,=\,\sum_{i=1}^n\zeta_i \,\partial_i g (G(\zeta_1),\ldots,G(\zeta_n)).
\end{align}
It can be checked that this densely defined operator $D:\Rc\subset\Ld^2(\Omega)\to\Ld^2(\Omega;\Hf)$ is closable. Setting
\begin{gather*}
\langle X,Y\rangle_{\Dm^{1,2}(\Hf^{\otimes r})}\,:=\,\expec{XY}+\expec{\langle DX,DY\rangle_{\Hf}},
\end{gather*}
and defining the {Malliavin-Sobolev space} $\Dm^{1,2}\subset\Ld^2(\Omega)$ as the closure of $\Rc$ for the corresponding norm, we may then extend the Malliavin derivative $D$ by density to this space.
Next, we define a {divergence operator} $D^*$ as the adjoint of the Malliavin derivative $D$,
and we construct the so-called {Ornstein-Uhlenbeck operator}
\[\Lc:=D^* D,\]
which is an essentially self-adjoint nonnegative operator.
We refer e.g.\@ to~\cite[p.34]{NP-book} for a description of the explicit action of $D^*$ and $\Lc$ on $\Rc$:
in particular, a direct computation leads to the commutator relation
\begin{align}\label{eq:commut-DL}
D\Lc=(1+\Lc)D.
\end{align}
Based on the above definitions, we can state the following useful result.
It is best known in the discrete Gaussian setting~\cite{HS-94},
and we refer e.g.\@ to~\cite[Appendix~A]{DO1} for a short proof.

\begin{lem}[Helffer--Sj\"ostrand identity]\label{lem:Mall}
For all $X,Y\in \Dm^{1,2}(\Omega)$, the covariance can be represented as
\begin{equation}\label{eq:HS}
\cov{X}{Y}\,=\,\expec{\langle DX,(1+\Lc)^{-1}DY\rangle_{\Hf}},
\end{equation}
where the inverse operator~$(1+\Lc)^{-1}$ is well-defined on $\Ld^2(\Omega)$ and has operator norm bounded by $1$ since $\Lc$ is nonnegative.
In particular, this implies the following Malliavin--Poincar\'e inequality,
\[\var{X}\,\le\,\expec{\|DX\|_{\Hf}^2}.\]
\end{lem}

\subsection{Proof of Theorem~\ref{th:correl}}
By duality, it suffices to prove the result for $1<q\le2$.
Starting from the perturbative expansion~\eqref{eq:barA}, it is then sufficient to prove for all $x,y\in\R^d$, $n\ge1$, $1<q\le2$, and $0<\theta\le\frac{2d}{q'}$, provided that~$\theta\ll1$ is small enough,
\begin{equation}\label{eq:bound-todo-g-bn}
\cro{P\Bb(\Kk P^\bot \Bb)^nP}_q(x,y)\,\le\,\tfrac{1}{q-1}C^n\theta^{1-n}\langle x-y\rangle^{C\theta-d-\gamma}
+C^n\theta^{2-n}\langle x-y\rangle^{C\theta-2d-d\wedge\gamma}.
\end{equation}
Now note that, in the case when $|x-y|_\infty\le4n$, say, this result simply follows from~\eqref{eq:compo-Rp} and from Lemma~\ref{lem:CZ} in the following form, for all $1<q\le2$ and $0<\theta\le\frac{2d}{q'}$,
\[\cro{P\Bb(\Kk P^\bot \Bb)^nP}_q(x,y)\,\le\,\|P\Bb(\Kk P^\bot \Bb)^nP\|_{\Ld^q(\R^d)^d\to\Ld^q(\R^d)^d}\,\le\,(\tfrac{Cq^2}{q-1})^n\,\le\,\tfrac{1}{q-1}\,C^n\theta^{1-n}.\]
It remains to prove~\eqref{eq:bound-todo-g-bn} in the case when $|x-y|_\infty>4n$, and for that purpose we can restrict to $x,y\in\Z^d$. More precisely, we shall show for all~$n\ge1$, for all~$x,y\in\Z^d$ with $|x-y|_\infty>4n$, for all $1<q\le2$ and $0<\theta\le\frac{2d}{q'}$, provided that $\theta\ll1$ is small enough,
\begin{equation}\label{eq:estim-bn-gauss}
\cro{P\Bb(\Kk P^\bot \Bb)^nP}_q(x,y)\\
\,\le\,C^n\theta^{1-n}\langle x-y\rangle^{C\theta-d-\gamma}
+C^n\theta^{2-n}\langle x-y\rangle^{C\theta-2d-d\wedge\gamma}.
\end{equation}
We split the proof into three steps.
From now on, let $x,y\in \Z^d$ be fixed with $|x-y|_\infty>4n$, let $1<q\le2$, and consider $\phi_x,\phi_y\in C^\infty_c(\R^d)^d$ such that
\begin{gather}\label{eq:cond-phixy}
\|\phi_x\|_{\Ld^{q'}(\R^d)}=\|\phi_y\|_{\Ld^q(\R^d)}=1,\qquad\supp(\phi_x)\subset Q(x),\qquad\supp(\phi_y)\subset Q(y).
\end{gather}

\medskip
\step1 Path decomposition of $P\Bb(\Kk P^\bot\Bb)^nP$.\\
Arguing as in~\eqref{eq:decomp-100}, we can decompose
\begin{multline*}
\big\langle \phi_x,P\Bb(\Kk P^\bot \Bb)^nP\phi_y\big\rangle\\
\,=\,\sum_{z_1,\ldots,z_{n-1}\in\Z^d}\Big\langle \phi_x,P\Bb(\Kk P^\bot\Bb\mathds1_{Q(z_1)})\ldots(\Kk P^\bot\Bb\mathds1_{Q(z_{n-1})})(\Kk P^\bot\Bb)P\phi_y\Big\rangle,
\end{multline*}
where the series is absolutely convergent.
As in the proof of Lemma~\ref{lem:deterministic}, we shall classify the contributions in this sum by conditioning on the first interval $|z_j-z_{j+1}|_\infty$ that reaches the largest dyadic value. More precisely, arguing as for~\eqref{eq:use-Sj0R}, we can rewrite the sum as
\begin{multline}\label{eq:decomp-Sj-termn}
\langle \phi_x,P\Bb(\Kk P^\bot \Bb)^nP\phi_y\rangle
\,=\,\sum_{j_1=0}^{n-1}\,\sum_{m=0}^\infty\,\sum_{z_{j_1},z_{j_1+1}\in\Z^d}\mathds1_{2^m<|z_{j_1}-z_{j_1+1}|_\infty\le 2^{m+1}}\\
\times\Big\langle \phi_x,P\Bb(\Kk_{2^m} P^\bot\Bb)^{j_1}\mathds1_{Q(z_{j_1})}
\Kk \mathds1_{Q(z_{j_1+1})}P^\bot\Bb(\Kk_{2^{m+1}} P^\bot\Bb)^{n-j_1-1}P\phi_y\Big\rangle,
\end{multline}
where henceforth we use the short-hand notation $\Kk_{\ell}:=\Kk_{\ell;1}$ for the discretely truncated kernel~\eqref{eq:truncate-K} with $R=1$.
Now we appeal to Lemma~\ref{lem:Mall} in the following form: for two random variables $X,Y\in\Dm^{1,2}$, we have
\begin{equation*}
PXP^\bot YP\,=\,\cov{X}{Y}{P}=\,\expec{\langle DX,(1+\Lc)^{-1}DY\rangle_\Hf}{P},
\end{equation*}
or alternatively, recalling the definition of the scalar product in $\Hf$, and introducing the short-hand notation $\hat DX\in\Ld^2(\R^d\times\Omega)$ given by $\hat D_zX\,:=\,\int_{\R^d}c_0(z-z')D_{z'}X\,dz'$,
\begin{equation*}
PXP^\bot YP\,=\,\int_{\R^d} \Big(P(\hat D_zX)(1+\Lc)^{-1}(\hat D_zY)P\Big)\,dz.
\end{equation*}
In view of~\eqref{eq:decomp-Sj-termn}, we shall apply this identity to products of $\Bb$: for all $x_0,\ldots, x_n$, considering the random variables
\[X:=\Bb(x_{j_1})P^\bot\Bb(x_{j_1-1})\ldots P^\bot\Bb(x_0),\qquad Y:=\Bb(x_{j_1+1})P^\bot\Bb(x_{j_1+2})\ldots P^\bot\Bb(x_n),\]
we find
\begin{multline*}
{P\Bb(x_0)\big(P^\bot\Bb(x_1)\ldots P^\bot\Bb(x_{j_1})\big)P^\bot \Bb(x_{j_1+1})\big(P^\bot\Bb(x_{j_1+2})\ldots P^\bot\Bb(x_n)\big)P}\\
\,=\,\cov{X}{Y}{P}
\,=\,\int_{\R^d}\Big(P(\hat D_zX)(1+\Lc)^{-1}(\hat D_zY)P\Big)\,dz.
\end{multline*}
Using the chain rule for the Malliavin derivative in form of
\begin{equation}\label{eq:Mall-chain}
\hat D_zX\,=\,\sum_{i=0}^{j_1}\big(\Bb(x_{j_1})\ldots \Bb(x_{i+1})\big)\hat D_z\Bb(x_i)\big(P^\bot\Bb(x_{i-1})\ldots P^\bot\Bb(x_0)\big),
\end{equation}
and similarly for $\hat D_zY$, we deduce
\begin{multline}\label{eq:decomp-PBPBP}
P\Bb(x_0)\big(P^\bot\Bb(x_1)\ldots P^\bot\Bb(x_{j_1})\big)P^\bot \Bb(x_{j_1+1})\big(P^\bot\Bb(x_{j_1+2})\ldots P^\bot\Bb(x_n)\big)P\\
\,=\,\sum_{i=0}^{j_1}\sum_{i'=j_1+1}^n\int_{\R^d}\Big(P\big(\Bb(x_{0})P^\bot\ldots\Bb(x_{i-1})P^\bot\big)\hat D_z\Bb(x_i)\big(\Bb(x_{i+1})\ldots\Bb(x_{j_1})\big)\\
\times(1+\Lc)^{-1}\big(\Bb(x_{j_1+1})\ldots\Bb(x_{i'-1})\big)\hat D_z\Bb(x_{i'})\big(P^\bot\Bb(x_{i'+1})\ldots P^\bot\Bb(x_n)\big)P\Big)\,dz.
\end{multline}
By definition of the Malliavin derivative, cf.~\eqref{eq:D-expl}, starting from the representation~\eqref{eq:rep-AA_0reB}, we note that $\hat D\Bb$ can be explicitly computed as
\begin{equation}\label{eq:comput-DB}
\hat D_z\Bb(x)\,=\,c_0(z-x)B_0'(G(x)).
\end{equation}
To simplify notation, let us assume that $\kappa=1$, meaning that the underlying Gaussian field~$G$ is $\R$-valued (the general case is treated similarly up to keeping track of index contractions).
Using the above identity~\eqref{eq:decomp-PBPBP} to reformulate the different terms in~\eqref{eq:decomp-Sj-termn}, first appealing to Lemma~\ref{lem:CZ-eta} to express operators as limits of absolutely converging integrals, we are easily led to
\begin{multline}\label{eq:decomp-Sj-termn-re}
\langle \phi_x,P\Bb(\Kk P^\bot \Bb)^nP\phi_y\rangle
\,=\,\sum_{0\le i\le j_1<i'\le n}\,\sum_{m=0}^\infty\,\sum_{z_{j_1},z_{j_1+1},z_i,z_{i'}\in \Z^d}\mathds1_{2^m<|z_{j_1}-z_{j_1+1}|_\infty\le 2^{m+1}}\\
\times\int_{\R^d}\Big\langle \phi_x,P(\Bb\Kk_{2^{m}} P^\bot)^{i}(\hat D_z\Bb)\mathds1_{Q_R(z_{i})}(\Kk_{2^{m}} \Bb)^{j_1-i}\mathds1_{Q(z_{j_1})}\Kk \mathds1_{Q(z_{j_1+1})}\\
\times (1+\Lc)^{-1}(\Bb \Kk_{2^{m+1}})^{i'-j_1-1}(\hat D_z\Bb )\mathds1_{Q(z_{i'})}(\Kk_{2^{m+1}} P^\bot \Bb)^{n-i'} P\phi_y\Big\rangle\,dz.
\end{multline}
Next, we split each of the terms in~\eqref{eq:decomp-Sj-termn-re} into two parts, depending on whether $|z_{i}-z_{i'}|_\infty\ge \frac12|x-y|_\infty$ or $|z_{i}-z_{i'}|_\infty<\frac12|x-y|_\infty$. In the latter case, by the triangle inequality, we note that the condition implies
\[\max\Big\{\max_{0\le j< i}|z_{j}-z_{j+1}|_\infty,\max_{i'\le j< n}|z_{j}-z_{j+1}|_\infty\Big\}\,>\,\ell:=\tfrac1{2n}|x-y|_\infty\,\ge\,1.\]
In that case, we may further
condition with respect to the first interval $|z_j-z_{j+1}|_\infty$ with $0\le j<i$ or $i'\le j<n$ that exceeds the length $\ell=\frac1{2n}|x-y|_\infty$. The corresponding index is denoted by $j=j_2$.
We can then reformulate the identity~\eqref{eq:decomp-Sj-termn-re} as
\begin{equation}\label{eq:decomp-Sj-termn-re-T012}
\langle \phi_x,P\Bb(\Kk P^\bot\Bb)^nP\phi_y\rangle
\,=\,T_0^n(x,y)+T_1^n(x,y)+T_2^n(x,y),
\end{equation}
where:
\begin{enumerate}[---]
\item $T_0^n(x,y)$ is the contribution from $|z_{i}-z_{i'}|_\infty\ge \frac12|x-y|_\infty$,
\begin{multline*}
\quad T_0^n(x,y)\,:=\,\sum_{0\le i\le j_1<i'\le n}\,\sum_{m=0}^\infty\,\sum_{z_{j_1},z_{j_1+1},z_i,z_{i'}\in\Z^d}\mathds1_{2^m<|z_{j_1}-z_{j_1+1}|_\infty\le 2^{m+1}}\mathds1_{|z_i-z_{i'}|_\infty\ge \frac12|x-y|_\infty}\\
\times\int_{\R^d}\Big\langle \phi_x,P(\Bb\Kk_{2^{m}} P^\bot)^{i}(\hat D_z\Bb)\mathds1_{Q(z_{i})}(\Kk_{2^{m}} \Bb)^{j_1-i}\mathds1_{Q(z_{j_1})}\Kk \mathds1_{Q(z_{j_1+1})}\\
\times(1+\Lc)^{-1} (\Bb \Kk_{2^{m+1}})^{i'-j_1-1}(\hat D_z\Bb )\mathds1_{Q(z_{i'})}(\Kk_{2^{m+1}} P^\bot \Bb)^{n-i'} P\phi_y\Big\rangle\,dz,
\end{multline*}
\item $T_1^n(x,y)$ and $T_2^n(x,y)$ are the contributions from $|z_{i}-z_{i'}|<\frac12|x-y|_\infty$ after further conditioning on the next long interval: more precisely, we have set
\begin{multline*}
\quad T_1^n(x,y)\,:=\,\sum_{0\le j_2< i\le j_1<i'\le n}\,\sum_{m=0}^\infty\,\sum_{z_{j_1},z_{j_1+1},z_{j_2},z_{j_2+1},z_i,z_{i'}\in \Z^d}\\
\times\mathds1_{2^m<|z_{j_1}-z_{j_1+1}|_\infty\le 2^{m+1}}\,\mathds1_{|z_{j_2}-z_{j_2+1}|_\infty>\ell}\,\mathds1_{|z_i-z_{i'}|_\infty<\frac12|x-y|_\infty}\\
\times\int_{\R^d}\Big\langle \phi_x,P(\Bb\Kk_{2^{m}\wedge\ell} P^\bot)^{j_2}\Bb P^\bot\mathds1_{Q(z_{j_2})}\Kk\mathds1_{Q(z_{j_2+1})}\\
\times (P^\bot\Bb\Kk_{2^{m+1}} )^{i-j_2-1}P^\bot(\hat D_z\Bb)\mathds1_{Q(z_{i})}(\Kk_{2^{m}} \Bb)^{j_1-i}\mathds1_{Q(z_{j_1})}\Kk \mathds1_{Q(z_{j_1+1})}\\
\times (1+\Lc)^{-1}(\Bb \Kk_{2^{m+1}})^{i'-j_1-1}(\hat D_z\Bb )\mathds1_{Q(z_{i'})}(\Kk_{2^{m+1}} P^\bot \Bb)^{n-i'} P\phi_y\Big\rangle\,dz,
\end{multline*}
and $T_2^n(x,y)$ is the symmetric term obtained by rather picking $i'\le j_2<n$.
\end{enumerate}

\medskip
\step2 Estimation of $T_0^n(x,y)$.\\
The above-defined term $T_0^n(x,y)$ can already be estimated as such: taking the supremum over $\phi_x,\phi_y$, using~\eqref{eq:compo-Rp}, the pointwise decay of $\Kk$, using~\eqref{eq:comput-DB}, and recalling that
\[\|(1+\Lc)^{-1}\|_{\Ld^2(\Omega)\to\Ld^2(\Omega)}\le1,\]
we find
\begin{multline}\label{eq:estim-T0n}
\sup_{\phi_x,\phi_y}|T_0^n(x,y)|\,\lesssim\,\sum_{0\le i\le j_1<i'\le n}\,\sum_{m=0}^\infty\,(2^m)^{-d}\sum_{z_{j_1},z_{j_1+1},z_i,z_{i'}\in \Z^d}\\
\times\mathds1_{2^m<|z_{j_1}-z_{j_1+1}|_\infty\le 2^{m+1}}\mathds1_{|z_i-z_{i'}|_\infty\ge\frac12|x-y|_\infty}
\,\cro{(\Bb\Kk_{2^{m}} P^\bot)^{i}}_{q,2}(x,z_i)\\
\times\cro{(\Kk_{2^{m}} \Bb)^{j_1-i}}_{q,2}(z_i,z_{j_1})
\cro{(\Bb \Kk_{2^{m+1}})^{i'-j_1-1}}_{q,2}(z_{j_1+1},z_{i'})\\
\times\cro{(\Kk_{2^{m+1}} P^\bot \Bb)^{n-i'} }_{q,2}(z_{i'},y)
\bigg(\int_{\R^d}\Big(\sup_{Q(z_i)}|c_0(\cdot-z)|\Big)\Big(\sup_{Q(z_{i'})}|c_0(\cdot-z)|\Big)\,dz\bigg),
\end{multline}
where the supremum is implicitly understood to run over all functions $\phi_x,\phi_y$ satisfying~\eqref{eq:cond-phixy}.
The decay assumption~\eqref{eq:decay-c0} for $c_0$ allows to estimate the last factor as follows: for all~$a,b\in\R^d$,
\begin{equation}\label{eq:estim-c0c0-sup}
\int_{\R^d}\Big(\sup_{Q(a)}|c_0(\cdot-z)|\Big)\Big(\sup_{Q(b)}|c_0(\cdot-z)|\Big)\,dz
\,\lesssim\,\langle a-b\rangle^{-\gamma}.
\end{equation}
By the triangle inequality, we note that kernel restrictions in~\eqref{eq:estim-T0n} entail that the summation is restricted to
\[|x-z_i|_\infty,|z_i-z_{j_1}|_\infty,|z_{j_1+1}-z_{i'}|_\infty,|z_{i'}-y|_\infty\le n2^{m+1},\qquad 2^m\ge\tfrac1{2n}|x-y|_\infty.\]
Further appealing to the deterministic estimates of Lemma~\ref{lem:deterministic}, for all $0<\theta\le\frac{2d}{q'}$, we deduce
\begin{multline*}
\sup_{\phi_x,\phi_y}|T_0^n(x,y)|\,\le\,C^{n}\theta^{1-n}\langle x-y\rangle^{-\gamma}
\sum_{0\le i\le j_1<i'\le n}\,\sum_{m=0}^\infty\,(2^m)^{-d}\,\mathds1_{2^m\ge\frac1{2n}|x-y|_\infty}\\
\times\sum_{z_{j_1},z_{j_1+1},z_i,z_{i'}\in\Z^d}\mathds1_{|x-z_i|_\infty,|z_i-z_{j_1}|_\infty,|z_{j_1+1}-z_{i'}|_\infty,|z_{i'}-y|_\infty\le n2^{m+1}}\\
\times  \langle x-z_i\rangle^{\theta-d}\langle z_i-z_{j_1}\rangle^{\theta-d}
\langle z_{j_1+1}-z_{i'}\rangle^{\theta-d}\langle z_{i'}-y\rangle^{\theta-d},
\end{multline*}
and thus, estimating the sums, provided that $8\theta<d$,
\begin{eqnarray}
\sup_{\phi_x,\phi_y}|T_0^n(x,y)|
&\le&C^{n}\theta^{1-n}\langle x-y\rangle^{-\gamma}\sum_{m=0}^\infty\,(2^m)^{4\theta-d}\,\mathds1_{2^m\ge\frac1{2n}|x-y|_\infty}\nonumber\\
&\le&C^{n}\theta^{1-n}\langle x-y\rangle^{4\theta-d-\gamma}.\label{eq:estim-T0n*}
\end{eqnarray}

\medskip
\step3 Analysis of $T_1^n(x,y),T_2^n(x,y)$.\\
We turn to the estimation of the last two terms $T_1^n(x,y),T_2^n(x,y)$ in~\eqref{eq:decomp-Sj-termn-re-T012}. By symmetry, we can focus on $T_1^n(x,y)$. Before estimating it, we proceed to a further refinement of the path decomposition: similarly as in~\eqref{eq:decomp-Sj-termn-re}, we can use once again the covariance structure $PXP^\bot YP$ to extract additional couplings between variables.
More precisely, we use that the first factor
$(\Bb\Kk_{2^{m}\wedge\ell}P^\bot)^{j_2}\Bb P^\bot$ in the definition of $T_1^n(x,y)$ ends with $P^\bot$:
arguing similarly as for~\eqref{eq:decomp-Sj-termn-re}, we deduce for any random variable $X$,
\begin{multline*}
P(\Bb\Kk_{2^{m}\wedge\ell}P^\bot)^{j_2}\Bb P^\bot XP\\
\,=\,\sum_{0\le l\le j_2}\sum_{z_l\in\Z^d}\int_{\R^d}\Big(P(\Bb\Kk_{2^{m}\wedge\ell}P^\bot)^{l} (\hat D_w\Bb)\mathds1_{Q(z_l)}(\Kk_{2^{m}\wedge\ell}\Bb)^{j_2-l}(1+\Lc)^{-1} (\hat D_wX)P\Big)dw,
\end{multline*}
so that $T_1^n(x,y)$ becomes
\begin{multline}\label{eq:reexpress-T1n}
T_1^n(x,y)\,=\,\sum_{0\le l\le j_2< i\le j_1<i'\le n}\,\sum_{m=0}^\infty\,\sum_{z_{j_1},z_{j_1+1},z_{j_2},z_{j_2+1},z_i,z_{i'},z_l\in\Z^d}\\
\times\mathds1_{2^m<|z_{j_1}-z_{j_1+1}|_\infty\le 2^{m+1}}\,\mathds1_{|z_{j_2}-z_{j_2+1}|_\infty>\ell}\,\mathds1_{|z_i-z_{i'}|_\infty<\frac12|x-y|_\infty}\\
\times\iint_{\R^d\times\R^d}\Big\langle \phi_x,P(\Bb\Kk_{2^{m}\wedge\ell} P^\bot)^{l}(\hat D_w\Bb)\mathds1_{Q(z_l)}(\Kk_{2^{m}\wedge\ell}\Bb)^{j_2-l}\mathds1_{Q(z_{j_2})}\Kk\mathds1_{Q(z_{j_2+1})}\\
\times (1+\Lc)^{-1}(\hat D_w X_{j_1,j_2,i,i'}^{m})\phi_y\Big\rangle\,dzdw,
\end{multline}
with the short-hand notation
\begin{multline*}
X_{j_1,j_2,i,i'}^{m}\,:=\,
(\Bb\Kk_{2^{m+1}} P^\bot)^{i-j_2-1}(\hat D_z\Bb)\mathds1_{Q(z_{i})}(\Kk_{2^{m}} \Bb)^{j_1-i}\mathds1_{Q(z_{j_1})}\Kk \mathds1_{Q(z_{j_1+1})}\\
\times (1+\Lc)^{-1}(\Bb \Kk_{2^{m+1}})^{i'-j_1-1}(\hat D_z\Bb )\mathds1_{Q(z_{i'})}(\Kk_{2^{m+1}} P^\bot \Bb)^{n-i'} P.
\end{multline*}
By the chain rule, cf.~\eqref{eq:Mall-chain}, the Malliavin derivative of the latter quantity can be split into six contributions, depending on the index $l'$ of the variable on which the derivative falls: further using the commutator relation~\eqref{eq:commut-DL} in the form
\[\hat D_w(1+\Lc)^{-1}\,=\,(2+\Lc)^{-1}\hat D_w,\]
we find
\begin{equation*}
\hat D_w X_{j_1,j_2,i,i'}^{m}\,=\,A_1+A_2+A_3+A_4+A_5+A_6,
\end{equation*}
in terms of
\begingroup\allowdisplaybreaks
\begin{eqnarray*}
A_1&:=&\sum_{j_2<l'<i}\sum_{z_{l'}\in\Z^d}(\Bb\Kk_{2^{m+1}})^{l'-j_2-1}(\hat D_w\Bb)\mathds1_{Q(z_{l'})}\Kk_{2^{m+1}} P^\bot(\Bb\Kk_{2^{m+1}} P^\bot)^{i-l'-1}\\
&&\hspace{1cm}\times(\hat D_z\Bb)\mathds1_{Q(z_{i})}(\Kk_{2^{m}} \Bb)^{j_1-i}\mathds1_{Q(z_{j_1})}\Kk \mathds1_{Q(z_{j_1+1})}\\
&&\hspace{1cm}\times (1+\Lc)^{-1}(\Bb \Kk_{2^{m+1}})^{i'-j_1-1}(\hat D_z\Bb )\mathds1_{Q(z_{i'})}(\Kk_{2^{m+1}} P^\bot \Bb)^{n-i'} P,\\
A_2&:=&(\Bb\Kk_{2^{m+1}})^{i-j_2-1}(\hat D_w\hat D_z\Bb)\mathds1_{Q(z_{i})}(\Kk_{2^{m}} \Bb)^{j_1-i}\mathds1_{Q(z_{j_1})}\Kk \mathds1_{Q(z_{j_1+1})}\\
&&\hspace{1cm}\times(1+\Lc)^{-1}(\Bb \Kk_{2^{m+1}})^{i'-j_1-1} (\hat D_z\Bb )\mathds1_{Q(z_{i'})}(\Kk_{2^{m+1}} P^\bot \Bb)^{n-i'} P,\\
A_3&:=&\sum_{i<l'\le j_1}\sum_{z_{l'}\in\Z^d}(\Bb\Kk_{2^{m+1}})^{i-j_2-1}(\hat D_z\Bb)\mathds1_{Q(z_{i})}(\Kk_{2^{m}} \Bb)^{l'-i-1}\Kk_{2^{m}}\\
&&\hspace{1cm}\times(\hat D_w\Bb)\mathds1_{Q(z_{l'})}(\Kk_{2^{m}} \Bb)^{j_1-l'}\mathds1_{Q(z_{j_1})}\Kk \mathds1_{Q(z_{j_1+1})}\\
&&\hspace{1cm}\times(1+\Lc)^{-1}(\Bb \Kk_{2^{m+1}})^{i'-j_1-1} (\hat D_z\Bb )\mathds1_{Q(z_{i'})}(\Kk_{2^{m+1}} P^\bot \Bb)^{n-i'} P,\\
A_4&:=&\sum_{j_1<l'<i'}\sum_{z_{l'}\in\Z^d}(\Bb\Kk_{2^{m+1}})^{i-j_2-1}(\hat D_z\Bb)\mathds1_{Q(z_{i})}(\Kk_{2^{m}} \Bb)^{j_1-i}\mathds1_{Q(z_{j_1})}\Kk\mathds1_{Q(z_{j_1+1})} \\
&&\hspace{1cm}\times(2+\Lc)^{-1}(\Bb \Kk_{2^{m+1}})^{l'-j_1-1}(\hat D_w\Bb)\mathds1_{Q(z_{l'})} \Kk_{2^{m+1}}\\
&&\hspace{1cm}\times(\Bb \Kk_{2^{m+1}})^{i'-l'-1} (\hat D_z\Bb )\mathds1_{Q(z_{i'})}(\Kk_{2^{m+1}} P^\bot \Bb)^{n-i'} P,\\
A_5&:=&(\Bb\Kk_{2^{m+1}})^{i-j_2-1}(\hat D_z\Bb)\mathds1_{Q(z_{i})}(\Kk_{2^{m}} \Bb)^{j_1-i}\mathds1_{Q(z_{j_1})}\Kk \mathds1_{Q(z_{j_1+1})}\\
&&\hspace{1cm}\times(2+\Lc)^{-1}(\Bb \Kk_{2^{m+1}})^{i'-j_1-1} (\hat D_w\hat D_z\Bb )\mathds1_{Q(z_{i'})}(\Kk_{2^{m+1}} P^\bot \Bb)^{n-i'} P,\\
A_6&:=&\sum_{i'<l'\le n}\sum_{z_{l'}\in\Z^d}(\Bb\Kk_{2^{m+1}})^{i-j_2-1}(\hat D_z\Bb)\mathds1_{Q(z_{i})}(\Kk_{2^{m}} \Bb)^{j_1-i}\mathds1_{Q(z_{j_1})}\Kk\mathds1_{Q(z_{j_1+1})}\\
&&\hspace{1cm}\times (2+\Lc)^{-1}(\Bb \Kk_{2^{m+1}})^{i'-j_1-1} (\hat D_z\Bb )\mathds1_{Q(z_{i'})}(\Kk_{2^{m+1}} \Bb)^{l'-i'-1}\\
&&\hspace{1cm}\times\Kk_{2^{m+1}} (\hat D_w\Bb)\mathds1_{Q(z_{l'})}(\Kk_{2^{m+1}} P^\bot \Bb)^{n-l'} P.
\end{eqnarray*}
\endgroup
Inserting this into~\eqref{eq:reexpress-T1n} leads to a corresponding decomposition of $T_1^n(x,y)$ into six different terms,
\begin{equation}\label{eq:decomp-T1-6}
T_1^n(x,y)\,=\,T_{1,1}^n(x,y)+T_{1,2}^n(x,y)+T_{1,3}^n(x,y)+T_{1,4}^n(x,y)+T_{1,5}^n(x,y)+T_{1,6}^n(x,y).
\end{equation}
Note that this decomposition is similar to the distinction of the six cases in Step~3 of the proof of Theorem~\ref{th:exp-mix} in Section~\ref{sec:proof-exp-mix}, as illustrated through the diagrammatic representation~\eqref{eq:diagrams}.
We shall proceed to a direct estimation of the different terms.
For shortness, we focus on the estimation of~$T_{1,1}^n(x,y)$ and~$T_{1,4}^n(x,y)$, while the terms~$T_{1,2}^n(x,y)$ and~$T_{1,3}^n(x,y)$ (resp.\@ $T_{1,5}^n(x,y)$ and~$T_{1,6}^n(x,y)$) are actually similar to~$T_{1,1}^n(x,y)$ (resp.\@ $T_{1,4}^n(x,y)$).
Note that in the estimation of~$T_{1,2}^n(x,y)$ and~$T_{1,5}^n(x,y)$ we further need to use the following computation of the second Malliavin derivative of $\Bb$, instead of~\eqref{eq:comput-DB},
\begin{equation*}
\hat D_w\hat D_z\Bb(x)\,=\,c_0(z-x)c_0(w-x)B_0''(G(x)).
\end{equation*}
We start with the estimation of $T_{1,1}^n(x,y)$. By the triangle inequality, we note that the kernel truncations entail that the summations in that term are further restricted to
\begin{gather*}
|z_{j_1}-z_i|_\infty,|z_{j_1+1}-z_{i'}|_\infty,
|z_{j_2}-z_l|_\infty,|z_{j_2+1}-z_{l'}|_\infty
\,\le\, n2^{m+1},\\
2^{m}\ge\tfrac1{2n}|x-y|_\infty.
\end{gather*}
Arguing similarly as for~\eqref{eq:estim-T0n*}, taking the supremum over $\phi_x,\phi_y$, using~\eqref{eq:compo-Rp}, \eqref{eq:comput-DB}, the pointwise decay of $\Kk$, using Lemma~\ref{lem:deterministic},
using~\eqref{eq:estim-c0c0-sup} to estimate the integrals with respect to $z$ and $w$, and directly evaluating the sums over $z_{j_1},z_{j_1+1},z_{j_2},z_{j_2+1}$, we easily obtain for all $0<\theta\le\frac{2d}{q'}$, provided that~$5\theta<d$,
\begin{multline*}
\sup_{\phi_x,\phi_y}|T_{1,1}^n(x,y)|\,\le\,C^{n}\theta^{2-n}\langle x-y\rangle^{4\theta-2d}
\sum_{z_i,z_{i'},z_l,z_{l'}\in \Z^d}
\langle x-z_l\rangle^{\theta-d}
\langle z_l-z_{l'}\rangle^{-\gamma}\langle z_{l'}-z_{i}\rangle^{\theta-d}\\
\times
\mathds1_{|z_i-z_{i'}|_\infty<\frac12|x-y|_\infty}
\langle z_i-z_{i'}\rangle^{-\gamma}
\langle z_{i'}-y\rangle^{\theta-d}.
\end{multline*}
The remaining sums are now easily evaluated, and we are led to
\begin{equation*}
\sup_{\phi_x,\phi_y}|T_{1,1}^n(x,y)|\,\le\,C^{n}\theta^{2-n}\langle x-y\rangle^{7\theta-d-2(d\wedge\gamma)}.
\end{equation*}
The same holds for $T_{1,2}(x,y)$ and $T_{1,3}(x,y)$.
Next, we turn to the estimation of $T_{1,4}(x,y)$, for which a similar argument leads us instead to
\begin{multline*}
\sup_{\phi_x,\phi_y}|T_{1,4}^n(x,y)|\,\le\,C^n\theta^{2-n}\langle x-y\rangle^{4\theta-2d}\sum_{z_i,z_{i'},z_l,z_{l'}\in\Z^d}\langle x-z_l\rangle^{\theta-d}\langle z_l-z_{l'}\rangle^{-\gamma}\\
\times\langle z_{l'}-z_{i'}\rangle^{\theta-d}\langle z_{i'}-y\rangle^{\theta-d},
\end{multline*}
and thus, after evaluating the remaining sums,
\begin{equation*}
\sup_{\phi_x,\phi_y}|T_{1;4}^n(x,y)|\,\le\,C^n\theta^{2-n}\langle x-y\rangle^{7\theta-2d-d\wedge\gamma}.
\end{equation*}
The same holds for $T_{1,5}(x,y)$ and $T_{1,6}(x,y)$.
Combined with~\eqref{eq:decomp-Sj-termn-re-T012}, \eqref{eq:estim-T0n*}, and~\eqref{eq:decomp-T1-6}, these different estimates lead us to the claim~\eqref{eq:estim-bn-gauss}. This proves the main decay estimate~\eqref{eq:main2} in Theorem~\ref{th:correl}. H\"older continuity of the symbol follows as in Section~\ref{ssect:hoelder}.
\qed

\section{Weak correctors}\label{sect:weakcorrectors}
This section is devoted to the proof of Theorem~\ref{th:weak-cor}, for which we focus on the stretched exponential mixing setting~{\bf(H$_1$)}.
By definition~\eqref{eq:def-varphin} of correctors, arguing by duality, it suffices to prove for all $k<2d$, $x_0\in\R^d$, $R_0\ge1$, and for all test random variables $\zeta_{0}\in\Ld^\infty(\Omega)$ that are $\sigma(\Aa|_{Q_{R_0}(x_0)})$-measurable, provided that $\delta\ll1$ is small enough,
\begin{equation}\label{eq:zeta-dual-corr}
\expec{\zeta_{0}\Big((-\nabla\cdot P^\bot\Aa P^\bot\nabla)^{-1}\nabla\cdot P^\bot\Aa P\nabla p_k\Big)(0)}\,\lesssim\,\delta \|\zeta_0\|_{\Ld^1(\Omega)}R_0^k\big(1+\langle \tfrac{x_0}{R_0}\rangle^{C\delta+k-d}\big),
\end{equation}
where we have set for shortness $p_k(x):=(\ee\cdot x)^k$.
From now on, let $x_0\in\R^d$, $R_0\ge1$, and let a $\sigma(\Aa|_{Q_{R_0}(x_0)})$-measurable test random variable $\zeta_0\in\Ld^\infty(\Omega)$ be fixed.
In the perturbative regime~\eqref{eq:decomp-AB}, using a Neumann series expansion as in~\eqref{eq:barA}, we can write
\begin{eqnarray*}
\lefteqn{\expec{\zeta_{0}\Big((-\nabla\cdot P^\bot\Aa P^\bot\nabla)^{-1}\nabla\cdot P^\bot\Aa P\nabla p_k\Big)(0)}}\\
&=&\delta\sum_{n=0}^\infty\delta^n\expec{\zeta_{0}\Big(\Uu( P^\bot\Bb\Kk)^{n} P^\bot\Bb P\nabla p_{k}\Big)(0)},
\end{eqnarray*}
with the short-hand notation $\Uu:=(-\triangle)^{-1}\nabla\cdot$.
Equivalently, denoting by $G_0$ the Green's function of the Laplacian, $\Uu^*\delta_0=-\nabla G_0$, we find
\begin{eqnarray*}
\lefteqn{\expec{\zeta_0\Big((-\nabla\cdot P^\bot\Aa P^\bot\nabla)^{-1}\nabla\cdot P^\bot\Aa P\nabla p_k\Big)(0)}}\\
&=&-\delta\sum_{n=0}^\infty\delta^n\Big\langle\nabla G_0\,,\,P\zeta_{0}( P^\bot\Bb\Kk)^{n} P^\bot\Bb P\nabla p_{k}\Big\rangle\\
&=&-\delta\sum_{n=0}^\infty\delta^n\Big\langle \nabla p_{k}\,,\,P\Bb(\Kk P^\bot\Bb)^n (\nabla G_0)P^\bot\zeta_{0}\Big\rangle,
\end{eqnarray*}
For all $1\le p, q\le\infty$, noting that
\[\|\nabla p_k\|_{\Ld^{p'}_{q'}(Q_{R_0}(x))}\,\lesssim\, R_0^{k-1+\frac{d}{p'}}\langle\tfrac{x}{R_0}\rangle^{k-1},\]
we deduce
\begin{multline}\label{eq:zeta-dual-corr-pre0}
\Big|\,\expec{\zeta_{0}\Big((-\nabla\cdot P^\bot\Aa P^\bot\nabla)^{-1}\nabla\cdot P^\bot\Aa P\nabla(\ee\cdot x)^k\Big)(0)}\!\Big|\\
\,\lesssim\,\delta R_0^{k-1+\frac{d}{p'}}\sum_{n=0}^\infty\delta^n\sum_{x\in R_0\Z^d}\langle\tfrac{x}{R_0}\rangle^{k-1}\|P\Bb(\Kk P^\bot\Bb)^n (\nabla G_0)P^\bot\zeta_{0}\|_{\Ld^p_q(Q_{R_0}(x))}.
\end{multline}
In order to prove~\eqref{eq:zeta-dual-corr},
it is therefore sufficient to establish the following kernel decay estimate: for all $x\in\R^d$, $n\ge0$, $1<p\le q<\frac{d}{d-1}$, and $\e>0$, setting $\theta:=\frac{2d}{p'}$, provided that~$\theta\ll1$ is small enough,
\begin{multline}\label{eq:kernel-todo-zetax}
\|P\Bb(\Kk P^\bot\Bb)^{n}( \nabla G_0)P^\bot\zeta_{0}\|_{\Ld^p_q(Q_{R_0}(x))}\,\lesssim_q\,R_0^{1-\frac{d}{p'}}\big(\tfrac{|\!\log\e|}{\e}\big)^{\frac{3d+1}{\gamma}}\|\zeta_{0}\|_{\Ld^1(\Omega)}C^n\theta^{-n}\\
\times\big(\langle \tfrac{x}{R_0}\rangle\wedge\langle \tfrac{x-x_0}{R_0}\rangle\big)^{C\theta+\e-2d}\langle \tfrac{x}{R_0}\rangle^{C\theta+1-d}.
\end{multline}
Indeed, inserting this into~\eqref{eq:zeta-dual-corr-pre0}, choosing $\e$ small enough, and evaluating the sum, the bound~\eqref{eq:zeta-dual-corr} would follow.
Instead of~\eqref{eq:kernel-todo-zetax}, we claim that it actually suffices to prove the following result: for all $x\in\R^d$, $n\ge0$, $R\ge R_0$, we have for all $1<p\le q<\frac{d}{d-1}$ and $\e>0$, setting $\theta:=\frac{2d}{p'}$, provided that $\theta\ll1$ is small enough,
\begin{multline}\label{eq:kernel-todo-zetax-red}
\|P\Bb( \Kk P^\bot\Bb)^{n} (\nabla G_0)P^\bot \zeta_{0}\|_{\Ld^p_q(Q_R(x))}\,\lesssim_q\,R^{1-\frac{d}{p'}}\|\zeta_{0}\|_{\Ld^1(\Omega)}C^n\theta^{-n}\\
\times\Big(\big(\langle\tfrac{x}R\rangle\wedge\langle\tfrac{x-x_0}R\rangle\big)^{C\theta-2d}\langle \tfrac xR\rangle^{C\theta+1-d}
+n^n\e^{1-n}\big(C_0\exp(-\tfrac1{C_0}R^\gamma)\big)^\frac12\langle \tfrac{x}R\rangle^{\e+1-d}\Big).
\end{multline}
As the left-hand side in this estimate is bounded below by the left-hand side of~\eqref{eq:kernel-todo-zetax}, optimizing the right-hand side with respect to the parameter $R\ge R_0$, we indeed find that this estimate implies~\eqref{eq:kernel-todo-zetax}.
Now note that the pointwise bound $|\nabla G_0(z)|\lesssim|z|^{1-d}$ entails for all $R\ge1$ and $1\le p\le q<\frac{d}{d-1}$,
\[\|\nabla G_0\|_{\Ld^p_q(Q_R(y))}\lesssim_qR^{1-\frac{d}{p'}}\langle\tfrac yR\rangle^{1-d},\]
and thus
\begin{eqnarray*}
\lefteqn{\|P\Bb( \Kk P^\bot\Bb)^{n} (\nabla G_0)P^\bot \zeta_{0}\|_{\Ld^p_q(Q_R(x))}}\\
&\le&\sum_{y\in R\Z^d}\|\nabla G_0\|_{\Ld^p_q(Q_R(y))}\cro{P\Bb( \Kk P^\bot\Bb)^{n}P^\bot \zeta_{0}}_{p,q;R}(x,y)\\
&\lesssim_q&R^{1-\frac{d}{p'}}\sum_{y\in R\Z^d}\langle\tfrac yR\rangle^{1-d}\,\cro{P\Bb( \Kk P^\bot\Bb)^{n}P^\bot \zeta_{0}}_{p,q;R}(x,y).
\end{eqnarray*}
In order to prove~\eqref{eq:kernel-todo-zetax-red}, it is therefore sufficient to establish the following kernel decay estimate: for all $x,y\in\R^d$, $n\ge0$, $R\ge R_0$, we have for all $1<p\le q\le2$ and $\e>0$, setting $\theta:=\frac{2d}{p'}$, provided that $\theta\ll1$ is small enough,
\begin{multline}\label{eq:kernel-todo-zetax-red-red}
\cro{P\Bb( \Kk P^\bot\Bb)^{n}P^\bot \zeta_{0}}_{p,q;R}(x,y)\,\lesssim_q\,\|\zeta_{0}\|_{\Ld^1(\Omega)}C^n\theta^{-n}\\
\times\Big(\big(\langle\tfrac{x-y}R\rangle\wedge\langle\tfrac{x-x_0}R\rangle\big)^{C\theta-2d}\langle \tfrac{x-y}R\rangle^{C\theta-d}
+n^n\e^{1-n}\big(C_0\exp(-\tfrac1{C_0}R^\gamma)\big)^\frac12\langle \tfrac{x-y}R\rangle^{\e-d}\Big).
\end{multline}
By a slight modification of the proof of Lemma~\ref{lem:deterministic}, we first note that we can estimate for all $1<p\le q\le2$, setting~$\theta:=\frac{2d}{p'}$, provided that~$\theta\ll1$ is small enough,
\begin{equation*}
\cro{P\Bb( \Kk P^\bot\Bb)^{n}P^\bot \zeta_{0}}_{p,q;R}(x,y)
\,\le\,\|\zeta_0\|_{\Ld^1(\Omega)}C^n\theta^{-n}\langle\tfrac{x-y}R\rangle^{\theta-d}.
\end{equation*}
In the case when $|x|_\infty\le4nR$ or when $|x-x_0|_\infty\le4nR$, this already proves the desired estimate~\eqref{eq:kernel-todo-zetax-red-red}.
It only remains to prove~\eqref{eq:kernel-todo-zetax-red-red} in the case when $\mbox{$|x|_\infty$},\mbox{$|x-x_0|_\infty$}>4nR$, and we can restrict for that purpose to $x\in R\Z^d$.
Now this can be done by a straightforward modification of the proof of~\eqref{eq:red-B-1re} in Section~\ref{sec:proof-exp-mix}. We skip the detail for conciseness.
\qed

\section{Massive approximation}\label{sect:massive}
{This section is devoted to the proof of Corollary~\ref{cor:massive}.
For all $\mu>0$, let us consider the massive solution operator $\nabla(\mu-\nabla\cdot\Aa\nabla)^{-1}\nabla$ on $\Ld^2(\R^d\times\Omega)$. Repeating the proof of Lemma~\ref{lem:Sigal}, this leads to a corresponding convolution operator~$\bar\Ac_\mu(\nabla)$ on $\Ld^2(\R^d)^d$ such that
\[\E\big[\nabla(\mu-\nabla\cdot\Aa\nabla)^{-1}\nabla\big]\,=\,\nabla(\mu-\nabla\cdot\bar\Ac_\mu(\nabla)\nabla)^{-1}\nabla,\]
which is given by
\begin{eqnarray*}
\bar \Ac_\mu(\nabla)&:=&\E[\Aa(\Id+\Phi_\mu(\cdot,\nabla))],\\
\Phi_\mu(\cdot,\nabla)&:=&P^\bot\nabla(\mu-\nabla\cdot P^\bot\Aa P^\bot\nabla)^{-1}\nabla\cdot P^\bot\Aa P.
\end{eqnarray*}
We note that the operator $\bar\Ac_\mu(\nabla)$ converges strongly to $\bar\Ac(\nabla)$ on $\Ld^2(\R^d)^d$, and it remains to show the convergence of derivatives of the symbols. For shortness, we focus on the stretched exponential $\alpha$-mixing setting of Theorem~\ref{th:exp-mix}, but the argument can be immediately adapted to the correlated Gaussian setting of Theorem~\ref{th:correl} as well.
We split the proof into two steps.

\medskip
\step1 Uniform decay estimates: we can decompose
\[\bar\Ac_\mu(\nabla)\,=\,\Aa_0+\delta\,\bar\Bc_\mu(\nabla),\]
and the following kernel estimate holds for all $x,y\in \R^d$,
\begin{equation}\label{eq:unif-decay-L}
\|\mathds1_{Q(x)}\bar\Bc_\mu(\nabla)\mathds1_{Q(y)}\|_{\Ld^2(\R^d)^d\to\Ld^2(\R^d)^d}\,\lesssim\,\delta K\log(2+|x-y|)^K\langle x-y\rangle^{\delta K-3d}e^{-\frac\mu C|x-y|}.
\end{equation}
This follows by repeating the proof of Theorem~\ref{th:exp-mix} with a positive mass $\mu>0$, further noting that all kernels have an additional exponential decay $e^{-\frac\mu C|\cdot|}$ due to the mass; see also~\cite[Theorem~1.3]{Lemm-18}.

\medskip
\step2 Conclusion.\\
The bound~\eqref{eq:unif-decay-L} ensures that for any $\mu>0$ the symbol of $\bar\Ac_\mu(\nabla)$ is locally analytic on~$i\R^d$ and that it is bounded in $C_b^{2d-\delta K-}(i\R^d)$ uniformly with respect to $\mu>0$.
This smoothness allows to define all higher-order homogenized coefficients $\{\bar\Aa_\mu^n\}_{n\ge1}$ and to get for all $\xi\in\R^d$, $0<\eta\le1$, and~$0\le \ell\le 2d-\eta-\delta K$, uniformly with respect to $\mu>0$,
\begin{equation}\label{eq:BL-Taylor}
\Big|\bar\Ac_\mu(i\xi)-\sum_{n=0}^\ell(i\xi)^{\otimes n}_{j_1\ldots j_n}\bar\Aa_{\mu;j_1\ldots j_n}^n\Big|
\,\lesssim\,|\xi|^{\ell+\frac12\eta}.
\end{equation}
Comparing different values of $\mu$ and extracting $\bar\Aa_\mu^\ell$, we deduce for all $|e|=1$, $0<\kappa,\eta\le1$, $\mu,\mu'>0$, and $0\le\ell\le2d-\eta-\delta K$,
\begin{multline}\label{eq:estim-conv-baramu-induction}
|e^{\otimes \ell}_{j_1\ldots j_\ell}(\bar\Aa_{\mu;j_1\ldots j_\ell}^\ell-\bar\Aa_{\mu';j_1\ldots j_\ell}^\ell)|\,\lesssim\,\kappa^{-\ell}|\bar\Ac_{\mu}(i\kappa e)-\bar\Ac_{\mu'}(i\kappa e)|+\kappa^{\frac12\eta}\\
+\sum_{n=0}^{\ell-1}\kappa^{n-\ell}|e^{\otimes n}_{j_1\ldots j_n}(\bar\Aa_{\mu;j_1\ldots j_n}^n-\bar\Aa_{\mu';j_1\ldots j_n}^n)|.
\end{multline}
Since $\bar\Ac_\mu(\nabla)$ converges strongly to $\bar\Ac(\nabla)$ on $\Ld^2(\R^d)^d$ as $\mu\downarrow0$ and since symbols are continuous, we can deduce the pointwise convergence of symbols $\bar\Ac_\mu(i\xi)\to\bar\Ac(i\xi)$ for all~\mbox{$\xi\in\R^d$}. Passing successively to the limit $\mu,\mu'\downarrow0$ and $\kappa\downarrow0$ in~\eqref{eq:estim-conv-baramu-induction}, we can conclude by a direct induction over $\ell$ that for all $|e|=1$ and $0\le n< 2d-\delta K$ the limit
\begin{equation}\label{eq:conv-L-approx-bara}
e^{\otimes n}_{j_1\ldots j_n}\hat\Aa_{j_1\ldots j_n}^n\,:=\,\lim_{\mu\downarrow0}\,e^{\otimes n}_{j_1\ldots j_n}\bar\Aa_{\mu;j_1\ldots j_n}^n
\end{equation}
does actually exist in $\R$.
Then passing to the limit in~\eqref{eq:BL-Taylor}, we deduce for all $\xi\in\R^d$, $\eta>0$, and $0\le \ell\le 2d-\eta-\delta K$,
\begin{equation*}
\Big|\bar\Ac(i\xi)-\sum_{n=0}^\ell(i\xi)^{\otimes n}_{j_1\ldots j_n}\hat\Aa_{j_1\ldots j_n}^n\Big|
\,\lesssim\,|\xi|^{\ell+\frac12\eta}.
\end{equation*}
As the coefficients $\{\hat\Aa^n\}_n$ satisfying this estimate are necessarily unique, they must coincide with the coefficients $\{\bar\Aa^n\}_n$ constructed in Corollary~\ref{cor:applyDGL} (up to symmetrization): the convergence~\eqref{eq:conv-L-approx-bara} then yields the conclusion.
\qed

\section{Annealed Green's function asymptotics}\label{sect:corproofs}

This section is devoted to the proof of Corollaries \ref{cor:avgcasreg}  and \ref{cor:higherorder} on asymptotics of the annealed Green's function and its derivatives.
Let $d\ge3$ and let $\chi$ be a fixed frequency cut-off with Fourier transform $\hat \chi\in C^\infty_c(\R^d)$.

\subsection{Proof of Corollary~\ref{cor:avgcasreg}}\label{sect:corproofs-avgcasreg}
Recalling the definition~\eqref{eq:def-calG} of the annealed Green's function $\Gc$ as a tempered distribution for $d>2$, we find for all multi-indices $\alpha\ge0$,
\begin{equation}\label{eq:startasymp}
\chi\ast\nabla^\alpha\Gc(x)\,=\,\int_{\R^d}e^{ix\cdot\xi}\hat \chi(\xi)\,\frac{(i\xi)^\alpha}{m(\xi)}\,\frac{d\xi}{(2\pi)^{d/2}},\qquad m(\xi):=\xi\cdot\bar\Ac(i\xi)\xi.
\end{equation}
Recalling that $\bar\Ac(i\xi)=\bar\Aa^1+O(|\xi|)$, we are naturally led to compare with derivatives of the corresponding homogenized Green's function $\bar G(x):=(-\nabla\cdot\bar\Aa^1\nabla)^{-1}(x)$, which can be written as
\[\chi\ast\nabla^\alpha\bar G(x)\,=\,
\int_{\R^d}e^{ix\cdot\xi}\hat\chi(\xi)\,\frac{(i\xi)^\alpha}{m_0(\xi)}\,\frac{d\xi}{(2\pi)^{d/2}},\qquad m_0(\xi):=\xi\cdot\bar\Aa^1\xi.\]
Let us then consider the difference
\begin{equation}\label{eq:rep-Gc-barG-chi-rep}
\chi*(\nabla^\alpha\Gc-\nabla^\alpha\bar G)(x)\,=\,\int_{\R^d}e^{i x\cdot\xi}  \hat\chi(\xi)\,(i\xi)^\alpha\,R(\xi)\,\frac{d\xi}{(2\pi)^{d/2}},
\end{equation}
in terms of
\[R(\xi)\,:=\,
\frac{m_0(\xi)-m(\xi)}{m(\xi)m_0(\xi)}.\]
We split the proof into two steps.

\medskip
\step1 Proof that for all $x\in\R^d$ and $|\alpha|<d+1-\delta K$,
\begin{equation}\label{eq:estim-Gc-barG-chi}
\big|\chi\ast(\nabla^\alpha\Gc-\nabla^\alpha\bar G)(x)\big|
\,\lesssim_{\chi}\,\langle x\rangle^{1-d-|\alpha|}.
\end{equation}
We start from the representation~\eqref{eq:rep-Gc-barG-chi-rep} and proceed by dyadic decomposition to estimate the oscillatory integrals.
Let $\varphi:\R^d\to\R$ be a smooth radial cutoff function with $\varphi(\xi)=1$ for $|\xi|\le1$ and $\varphi(\xi)=0$ for $|\xi|>2$, and set $\psi(\xi):=\varphi(\xi)-\varphi(2\xi)$. We then define $\psi_0(\xi):=1-\varphi(\xi)$ and $\psi_l(\xi):=\psi(2^{l-1}\xi)$ for all $l\geq 1$. Note that this defines a partition of unity $\sum_{l\geq 0}\psi_l=1$ on $\R^d$. In these terms, let us decompose
\begin{equation}\label{eq:decomp-dyad-nabalphG}
\chi*(\nabla^\alpha\Gc-\nabla^\alpha\bar G)(x)\,=\,\sum_{l\ge0}\int_{\R^d}e^{i x\cdot\xi}  \hat\chi(\xi)\,(i\xi)^\alpha\,\psi_l(\xi)\,R(\xi)\,\frac{d\xi}{(2\pi)^{d/2}}.
\end{equation}
We separately analyze the cases $l=0$ and $l\ne0$, and we split the proof into three further substeps.

\medskip
\substep{1.1} Case $l=0$: proof that for all $r<2d-\delta K$,
\begin{equation}\label{eq:Rdyad-estim-l=0}
\bigg|\int_{\R^d}e^{i x\cdot\xi}  \hat\chi(\xi)\,(i\xi)^\alpha\,\psi_0(\xi)\,R(\xi)\,d\xi \bigg|\,\lesssim_{\chi,r}\,\langle x\rangle^{-r}.
\end{equation}
As by Theorem~\ref{th:exp-mix} the symbol $\xi\mapsto \bar\Ac(i\xi)$ belongs to $C_b^{2d-\delta K-}(\R^d)$, we get on any compact set $K\subset \R^d\setminus\{0\}$, for all~$r<2d-\delta K$,
\[\|R\|_{C^r(K)}\,\lesssim_{K,r}\,1.\]
As $\hat \chi$ is compactly supported and as $\psi_0$ is supported in $\R^d\setminus B$, this yields the claim~\eqref{eq:Rdyad-estim-l=0}.

\medskip
\substep{1.2} Case $l\geq 1$: proof that for all $l\ge1$ and $r<2d-\delta K$,
\begin{equation}\label{eq:Rdyad-estim-l>0}
\bigg|\int_{\R^d}e^{i x\cdot\xi}  \hat\chi(\xi)\,(i\xi)^\alpha\,\psi_l(\xi)\,R(\xi)\,d\xi \bigg|\,\lesssim_{\chi,r}\,(2^{-l})^{d+|\alpha|-1}\langle2^{-l} x\rangle^{-r}.
\end{equation}
As $\bar\Ac(0)=\bar\Aa^1$, we can write, using Einstein's summation convention,
\[R(\xi)\,=\,-\frac{\xi\cdot(\bar\Ac(i\xi)-\bar\Aa^1)\xi}{m(\xi)m_0(\xi)}\,=\,-i\Big(\int_0^1\nabla_j\bar\Ac_{kl}(it\xi)dt\Big)\frac{\xi_j\xi_k\xi_l}{m(\xi)m_0(\xi)}.\]
As we have $|m(\xi)|\simeq|m_0(\xi)|\simeq|\xi|^2$ and as by Theorem~\ref{th:exp-mix} the symbol $\xi\mapsto \bar\Ac(i\xi)$ belongs to $C_b^{2d-\delta K-}(\R^d)$, we can deduce for all~$r<2d-\delta K$,
\[\|R\|_{C^{r}(2^{-l}\le|\xi|\le2^{2-l})}\,\lesssim_r\,(2^{-l})^{-1-r}.\]
As $\psi_l$ is supported in $2^{-l}\le|\xi|\le2^{2-l}$ and satisfies $\|\psi_l\|_{C^r}\lesssim_r (2^{-l})^{-r}$ for all $r\ge0$, the claim~\eqref{eq:Rdyad-estim-l>0} follows.

\medskip
\substep{1.3} Proof of~\eqref{eq:estim-Gc-barG-chi}.\\
Using~\eqref{eq:Rdyad-estim-l=0} and~\eqref{eq:Rdyad-estim-l>0} to estimate the dyadic sum~\eqref{eq:decomp-dyad-nabalphG}, and distinguishing the cases $2^l>\langle x\rangle$ and $2^l\le\langle x\rangle$,
we get for all $x\in\R^d$ and $r<2d-\delta K$, provided $r\ne d+|\alpha|-1$,
\begin{eqnarray*}
\lefteqn{\big|\chi\ast(\nabla^\alpha\Gc-\nabla^\alpha\bar G)(x)\big|}\\
&\lesssim_{\chi,r}&\langle x\rangle^{-r}+\sum_{l\ge1:2^l\le\langle x\rangle}(2^{-l})^{d+|\alpha|-1-r}\langle x\rangle^{-r}+\sum_{l\ge1:2^l\ge\langle x\rangle}(2^{-l})^{d+|\alpha|-1}\\
&\lesssim_{r}&\langle x\rangle^{-r}+\langle x\rangle^{1-d-|\alpha|},
\end{eqnarray*}
thus proving~\eqref{eq:estim-Gc-barG-chi}.

\medskip
\step2 Proof that for all $x\in\R^d$ and $|\alpha|=d$ with $0<\eta<1-\delta K$,
\[[\chi\ast(\nabla^\alpha\Gc-\nabla^\alpha\bar G)]_{C^\eta(B(x))}\,\lesssim_{\chi,\eta}\,\langle x\rangle^{1-2d-\eta}.\]
Starting from~\eqref{eq:rep-Gc-barG-chi-rep}, the fractional differential quotient can be written as
\begin{multline*}
\frac{1}{|y|^\eta}\Big(\chi\ast(\nabla^\alpha\Gc-\nabla^\alpha\bar G)(x+y)-\chi\ast(\nabla^\alpha\Gc-\nabla^\alpha\bar G)(x)\Big)\\
\,=\,\int_{\R^d}e^{ix\cdot\xi}\hat\chi(\xi)\,\Big(\frac{e^{iy\cdot\xi}-1}{|y|^\eta}(i\xi)^\alpha\Big)\,R(\xi)\,\frac{d\xi}{(2\pi)^{d/2}}.
\end{multline*}
Noting that for all $\alpha,\beta\ge0$ and $|\xi|\le1$ we have
\[\sup_{0<|y|\le1}\Big|\nabla_\xi^\beta\Big(\frac{e^{iy\cdot\xi}-1}{|y|^\eta}(i\xi)^\alpha\Big)\Big|\,\lesssim\,\left\{\begin{array}{lll}
|\xi|^{|\alpha|-|\beta|+\eta}&:&\beta\le\alpha,\\
1&:&\beta>\alpha,
\end{array}\right.\]
the conclusion then follows by repeating the analysis of Step~1.\qed

\begin{rem}
In the case when $\Aa$ is symmetric in law, we have~\mbox{$\bar\Aa^2=0$}, cf.~\cite[Lemma~2.4]{DO1}. Hence, we deduce $|\bar\Ac(i\xi)-\bar\Aa^1|\lesssim|\xi|^2$, which yields one additional exponent of decay in~\eqref{eq:Rdyad-estim-l>0} and leads to the improvement claimed in the statement.
\end{rem}

\subsection{Proof of Corollary \ref{cor:higherorder}}\label{sect:corproofs-avgcasreg-2}
Iteratively solving~\eqref{eq:corr-Green}, the definition~\eqref{eq:corr-Green-0} of the corrected Green's functions can be reformulated as follows, for $1\le \ell\le2d$: $\bar G^\ell$ is defined as the tempered distribution with Fourier transform
\begin{multline*}
\Fc[\bar G^\ell](\xi)\,:=\,(\xi\cdot\bar\Aa^1\xi)^{-1}
+\sum_{n=1}^{\ell-1}\sum_{m\ge1,r_1,\ldots,r_m\ge1\atop r_1+\ldots+ r_m=n}(\xi\cdot\bar\Aa^1\xi)^{-1}(i\xi\cdot\bar\Aa^{r_1+1}(i\xi)i\xi)\\
\ldots (\xi\cdot\bar\Aa^1\xi)^{-1} (i\xi\cdot\bar\Aa^{r_m+1}(i\xi)i\xi)(\xi\cdot\bar\Aa^1\xi)^{-1},
\end{multline*}
where we use the short-hand notation $\bar\Aa^{n+1}(i\xi):=\bar\Aa^{n+1}_{j_1\ldots j_n}(i\xi)^n_{j_1\ldots j_n}$.
Note that
\[|\Fc[\bar G^\ell](\xi)|\,\lesssim\,|\xi|^{-2},\]
which ensures that the Fourier transform can be inverted and that $\bar G^\ell$ is well-defined as a tempered distribution for $d\ge3$ (and~$\nabla\bar G^\ell$ is further defined for $d=2$).
Next, we recall that by Theorem~\ref{th:exp-mix} and by Taylor's expansion we have for all $1\le\ell\le2d$ and $\eta>\delta K$,
\[\Big|\xi\cdot\bar\Ac(i\xi)\xi-\sum_{n=0}^{\ell-1}\xi\cdot \bar\Aa^{n+1}(i\xi)\xi\Big|\,\lesssim\,\left\{\begin{array}{lll}
|\xi|^{\ell+2}&:&\ell<2d,\\
|\xi|^{2d+2-\eta}&:&\ell=2d.
\end{array}\right.\]
Using a geometric series to expand, for any $N\ge0$,
\begin{eqnarray*}
(\xi\cdot\bar\Ac(i\xi)\xi)^{-1}
&=&(\xi\cdot\bar\Aa^1\xi)^{-1}\Big(1-(\xi\cdot\bar\Aa^1\xi)^{-1}i\xi\cdot(\bar\Ac(i\xi)-\bar\Aa^1)i\xi\Big)^{-1}\\
&=&(\xi\cdot\bar\Aa^1\xi)^{-1}\sum_{n=0}^N\Big((\xi\cdot\bar\Aa^1\xi)^{-1}i\xi\cdot(\bar\Ac(i\xi)-\bar\Aa^1)i\xi\Big)^n\\
&&+(\xi\cdot\bar\Ac(i\xi)\xi)^{-1}\Big((\xi\cdot\bar\Aa^1\xi)^{-1}i\xi\cdot(\bar\Ac(i\xi)-\bar\Aa^1)i\xi\Big)^{N+1},
\end{eqnarray*}
inserting the above expansion for $\bar\Ac$,
and comparing with the definition of $\Fc[\bar G^\ell]$,
we deduce that the discrepancy
\[R_\ell(\xi)\,:=\,(\xi\cdot\bar\Ac(i\xi)\xi)^{-1}-\Fc[\bar G^\ell](\xi)\]
satisfies for all $|\xi|\le1$, $1\le\ell\le 2d$, and $\eta>\delta K$,
\[|R_\ell(\xi)|\,\lesssim\,
\left\{\begin{array}{lll}
|\xi|^{\ell-2}&:&\ell<2d,\\
|\xi|^{2d-2-\eta}&:&\ell=2d.
\end{array}\right.\]
Similarly, taking derivatives, we find for all $|\xi|\le1$, $0\le k\le2d-1$, $1\le \ell\le2d$, and~$\eta>\delta K$,
\[|\nabla_\xi^kR_\ell(\xi)|\,\lesssim_\eta\,\left\{\begin{array}{lll}
|\xi|^{\ell-k-2}&:&\ell<2d,\\
|\xi|^{2d-k-2-\eta}&:&\ell=2d.
\end{array}\right.\]
Using the representation
\[\chi\ast(\nabla^\alpha\Gc-\nabla^\alpha\bar G^\ell)(x)\,=\,\int_{\R^d}e^{ix\cdot\xi}\hat\chi(\xi)\,(i\xi)^\alpha\,R_\ell(\xi)\,\frac{d\xi}{(2\pi)^{d/2}},\]
the conclusion now follows similarly as in the proof of Corollary~\ref{cor:avgcasreg}; we skip the details for shortness.
\qed

\subsection{Estimates without frequency cut-off}\label{sect:corproofs-cutoff}
In this section, we prove the different claims contained in Remark~\ref{rem:no-frequ-cutoff} on the possibility of removing the frequency cut-off in Corollary~\ref{cor:avgcasreg}.
We split the proof into three steps, separately proving items~(a), (b), and~(c).

\medskip
\step{1} Proof of~(a): for all $\alpha\ge0$ and $p>|\alpha|+d-2$,
\begin{equation}\label{eq:estim-cutoff-barG}
|\chi\ast\nabla^\alpha\bar G(x)-\nabla^\alpha\bar G(x)|\,\lesssim_{\chi,\alpha,p}\,|x|^{-p}.
\end{equation}
Given a radial cut-off function $g\in C^\infty_c(\R^d)$ with $g(\xi)=1$ for $|\xi|\le1$, we let $g_\e(\xi):=g(\e\xi)$ for $\e>0$. 
In these terms, by an approximation argument, we can represent for $|x|>0$,
\[\chi\ast\nabla^\alpha\bar G(x)-\nabla^\alpha\bar G(x)\,=\,
\lim_{\e\downarrow0}\int_{\R^d}e^{ix\cdot\xi}g_\e(\xi)(\hat\chi(\xi)-1)\,\frac{(i\xi)^\alpha}{m_0(\xi)}\,\frac{d\xi}{(2\pi)^{d/2}}.\]
For any integer $p\ge0$, integrating by parts, we may then estimate
\begin{equation}\label{eq:estim-cutoff-barG-0}
|\chi\ast\nabla^\alpha\bar G(x)-\nabla^\alpha\bar G(x)|\,\lesssim\,|x|^{-p}
\limsup_{\e\downarrow0}\int_{\R^d}\Big|\nabla_\xi^p\Big(g_\e(\xi)(\hat\chi(\xi)-1)\,\frac{(i\xi)^\alpha}{m_0(\xi)}\Big)\Big|\,d\xi,
\end{equation}
and thus, recalling that $(\hat\chi-1)(\xi)=0$ for $|\xi|\le1$, and noting that $|\nabla_\xi^ng_\e(\xi)|\lesssim_n|\xi|^{-n}$ and $|\nabla_\xi^n(\frac1{m_0})(\xi)|\lesssim_n|\xi|^{-n-2}$ for all $n\ge0$,
\[|\chi\ast\nabla^\alpha\bar G(x)-\nabla^\alpha\bar G(x)|\,\lesssim_{\chi,\alpha,p}\,|x|^{-p}
\int_{|\xi|\ge1}|\xi|^{|\alpha|-2-p}\,d\xi.\]
This proves the claim~\eqref{eq:estim-cutoff-barG} for all $p>d-2+|\alpha|$.

\medskip
\step{2} Proof of~(b): for all $|\alpha|\le1$, we have for $|x|\ge1$,
\begin{equation}\label{eq:remove-cutoff-CN}
|\chi\ast\nabla^\alpha\Gc(x)-\nabla^\alpha\Gc(x)|\,\lesssim_{\chi,\alpha,p}\,|x|^{-d-1},
\end{equation}
and in addition, for $|\alpha|=1$, we have for $|x|\ge2$,
\begin{equation}\label{eq:remove-cutoff-CN-re}
[\chi\ast\nabla^\alpha\Gc-\nabla^\alpha\Gc]_{C^{1-\eta}(B(x))}\,\lesssim_{\chi,\alpha,\eta,p}\,\left\{\begin{array}{lll}
|x|^{-d-1}&:&\eta>\frac{5d}{6d+6}+C\delta K,\\
|x|^{-d}&:&\eta>0.
\end{array}\right.
\end{equation}
Arguing as in~\eqref{eq:estim-cutoff-barG-0}, with the same cut-off $g_\e$,
we can estimate for any multi-index $\alpha\ge0$ and integer $p\ge0$,
\[|\chi\ast\nabla^\alpha\mathcal G(x)-\nabla^\alpha\mathcal G(x)|\,\lesssim\,
|x|^{-p}\limsup_{\e\downarrow0}\int_{\R^d}\Big|\nabla_\xi^p\Big(g_\e(\xi)(\hat\chi(\xi)-1)\,\frac{(i\xi)^\alpha}{m(\xi)}\,\Big)\Big|d\xi.\]
Recalling that $\hat\chi(\xi)-1=0$ for $|\xi|\le1$, that $\hat\chi$ is compactly supported, that we have $|\nabla_\xi^ng_\e(\xi)|\lesssim_n|\xi|^{-n}$ for all~$n\ge0$, and recalling the definition $m(\xi)=\xi\cdot\bar\Ac(i\xi)\xi$ and the uniform ellipticity of $\bar\Ac(i\xi)$, we deduce
\begin{equation*}
|\chi\ast\nabla^\alpha\mathcal G(x)-\nabla^\alpha\mathcal G(x)|\,\lesssim_{\chi,\alpha,g,p}\,
|x|^{-p}\sum_{n=0}^p\int_{|\xi|\ge1}|\xi|^{|\alpha|+n-2-p}|\nabla^{n}_\xi\bar\Ac(i\xi)|\,d\xi.
\end{equation*} 
In order to estimate the right-hand side, we appeal to the high-frequency weak integrability result for derivatives of the symbol~$\bar\Ac$ as established by Conlon and Naddaf in~\cite{conlon2000green}; see Appendix~\ref{ssect:CN}. More precisely, Lemma~\ref{lem:CN} yields $\nabla^n\bar\Ac\in\Ld^{d/n}_w(i\R^d)$ for $0\le n\le d-1$. As for $1\le n\le d-1$ and $p>d+|\alpha|-2$ the test function $\xi\mapsto\mathds1_{|\xi|\ge1}|\xi|^{|\alpha|+n-2-p}$ belongs to the dual Lorentz space $\Ld^{s}_w(\R^d)^*=\Ld^{s',1}(\R^d)$ with $s=\frac dn>1$, we deduce for any integer $p>d+|\alpha|-2$,
\begin{multline}\label{eq:remove-cutoff-CN-pre}
|\chi\ast\nabla^\alpha\mathcal G(x)-\nabla^\alpha\mathcal G(x)|\\[-2mm]
\,\lesssim_{\chi,\alpha,g,p}\,
|x|^{-p}
+\mathds1_{p\ge d}\,|x|^{-p}\sum_{n=d}^p\int_{|\xi|\ge1}|\xi|^{|\alpha|+n-2-p}|\nabla^{n}_\xi\bar\Ac(i\xi)|\,d\xi.
\end{multline}
To estimate the last sum, the Conlon--Naddaf integrability result of Lemma~\ref{lem:CN} needs to be properly combined with our regularity result of Theorem~\ref{th:exp-mix} by means of an interpolation argument.
On the one hand, as $\nabla^{d-1}\bar\Ac$ belongs to~$\Ld^\infty(i\R^d)$ by Theorem~\ref{th:exp-mix} (provided that~$\delta$ is small enough), the result of Lemma~\ref{lem:CN} yields by interpolation, for all~$0<\eta<1$ and~$q>\frac d{d-\eta}$,
\[\sup_{0<|y|\le1}\bigg\|\frac{\nabla^{d-1}\bar\Ac(\cdot+iy)-\nabla^{d-1}\bar\Ac}{|y|^{1-\eta}}\bigg\|_{\Ld^q(i\R^d)}\,\lesssim_{q,\eta}\,1,\]
which implies $\nabla^{d-1}\bar\Ac\in\dot W^{1-\eta,q}(i\R^d)$. 
On the other hand, Theorem~\ref{th:exp-mix} implies that $\bar\Ac$ belongs to $W^{s,\infty}(i\R^d)$ for all $s<2d-\delta K$, but also to $H^{s}(i\R^d)$ for all $s<\frac{5d}2-\delta K$, cf.~Remark~\ref{rem:reg-Hs}.
By interpolation, provided that $\delta$ is small enough, we can deduce
\begin{gather*}
\bar\Ac\in\dot W^{s,p}(i\R^d),\\
\text{for all $d\le s<d\tfrac{5-2\delta K/d}{2+\theta(3-2\delta K/d)}$, $\tfrac2{1+\theta}<p\le2$, and $0\le\theta\le1$.}\nonumber
\end{gather*}
In particular, this yields
\begin{equation}\label{eq:reg-interpol-Asymb}
\begin{array}{rl}
\nabla^d\bar\Ac\in\Ld^p(i\R^d),\quad&\text{for all~$p>1$},\\
\nabla^{d+1}\bar\Ac\in\Ld^p(i\R^d),\quad&\text{for all $p>\frac{6d+6}{6d+1}+C\delta K$}.
\end{array}
\end{equation}
Using this to control the right-hand side in~\eqref{eq:remove-cutoff-CN-pre} with $p=d+1$, the claim~\eqref{eq:remove-cutoff-CN} follows. The claim~\eqref{eq:remove-cutoff-CN-re} is deduced similarly by considering fractional differential quotients.

\medskip
\step{3} Proof of~(c):
if $\Aa$ is rotationally symmetric in law, then for all $|\alpha|<\frac{d+3}2$ we have for~$|x|\ge1$,
\[|\chi\ast\nabla^\alpha\Gc(x)-\nabla^\alpha\Gc(x)|\,\lesssim_{\chi,\alpha}\,|x|^{2-d-\frac{d+3}2}.\]
Setting for shortness $f_\e(\xi):=\frac{1}{m(\xi)}g_\e(\xi) (\hat\chi(\xi)-1)$, we start again with the following representation, for $|x|>0$,
\[\chi\ast\Gc(x)-\Gc(x)\,=\,\lim_{\e\downarrow0}\int_{\R^d}e^{i x\cdot\xi}  f_\e(\xi)\,d\xi.\]
If $\Aa$ is rotationally symmetric in law, we find that the symbol $m(\xi)=\xi\cdot\bar\Ac(i\xi)\xi$ is radial on~$\R^d$.
Without loss of generality, we can assume that the cut-off function $\chi$ is also radial.
Since $g_\e$ was taken radial, the function $f_\e$ is also radial.
Using radial variables $|x|\equiv r$ and $|\xi|\equiv k$, and using the abusive notation $f_\e(\xi)=f_\e(k)$, the Fourier transform of $f_\e$ takes the form
\[\int_{\R^d}e^{i x\cdot\xi}  f_\e(\xi)\,d\xi 
\,=\,c r^{-\nu} \int_0^\infty k^{\nu+1} \mathcal J_{\nu}(rk) f_\e(k)\,dk,\qquad \nu:=\tfrac{d-2}{2},\]
where $c$ is some universal constant and where $\mathcal J_\nu$ stands for the Bessel function of the first kind; see e.g.~\cite[Theorem~3.3]{SteinWeiss}.
Taking spatial derivatives, we may then deduce for $\alpha\ge0$,
\[|\chi\ast\nabla^\alpha\Gc(x)-\nabla^\alpha\Gc(x)|\,\lesssim_\alpha\,\sum_{n=0}^{|\alpha|}r^{n-|\alpha|-\nu}\limsup_{\e\downarrow0}\bigg|\partial_r^n\Big(\int_0^\infty k^{\nu+1} \mathcal J_{\nu}(rk) f_\e(k)\,dk\Big)\bigg|.\]
Noting that for a smooth function $h$ we have $\partial_r^n( h(rk))=(\frac{k}{r})^n\partial_k^n( h(rk))$ for all $n\ge0$, and integrating by parts, we are led to
\[|\chi\ast\nabla^\alpha\Gc(x)-\nabla^\alpha\Gc(x)|\,\lesssim_\alpha\,r^{-|\alpha|-\nu}\sum_{n=0}^{|\alpha|}\limsup_{\e\downarrow0}\bigg|\int_0^\infty \mathcal J_{\nu}(rk)\,\partial_k^n\big(k^{\nu+n+1}f_\e(k)\big)\,dk\bigg|.\]
Now recall the following identity for Bessel functions, which serves as the basis for harnessing oscillations: $\partial_z (z^{\lambda+1} \mathcal J_{\lambda+1}(z))=z^{\lambda+1} \mathcal J_{\lambda}(z)$. After rescaling, this gives
\[\mathcal J_{\lambda}(rk)
\,=\,r^{-1}k^{-\lambda-1}\partial_k (k^{\lambda+1} \mathcal J_{\lambda+1}(rk)).\]
Iteratively applying this identity and integrating by parts, we get for an  integer $p\ge0$ to be chosen later,
\begin{multline*}
|\chi\ast\nabla^\alpha\Gc(x)-\nabla^\alpha\Gc(x)|\\
\,\lesssim_\alpha\,r^{-p-|\alpha|-\nu}\sum_{n=0}^{|\alpha|}\limsup_{\e\downarrow0}\bigg|\int_0^\infty \mathcal J_{\nu+p}(rk)\,k^{\nu+p}(\partial_k k^{-1})^p(k^{-\nu}\partial_k^n)\big(k^{\nu+n+1}f_\e(k)\big)\,dk\bigg|.
\end{multline*}
Recalling as in Step~2 that we have for $n\ge0$,
\[|\partial_k^nf_\e(k)|\lesssim_n\sum_{m=0}^nk^{m-2-n}|\partial_\xi^m\bar\Ac(i\xi)|\mathds1_{k\ge1},\]
and recalling that Bessel functions satisfy the pointwise decay $|\mathcal J_\lambda(z)|\lesssim|z|^{-1/2}$ for all $z,\lambda\ge0$, we deduce for  integers~$p\ge0$,
\begin{equation*}
|\chi\ast\nabla^\alpha\Gc(x)-\nabla^\alpha\Gc(x)|\,\lesssim_\alpha\,r^{-p-|\alpha|-\nu-\frac12}\sum_{n=0}^{|\alpha|+p}\int_{|\xi|\ge1} |\xi|^{n+\nu-p-d-\frac12}|\partial_\xi^n\bar\Ac(i\xi)|\,d\xi.
\end{equation*}
Now appealing to the Conlon--Naddaf lemma in form of~\eqref{eq:reg-interpol-Asymb}, the conclusion follows by a direct computation.
\qed

\appendix
\section{Conlon--Naddaf--Sigal approach}\label{app:Sigal}
This appendix is devoted to the proof of Lemma~\ref{lem:Sigal},
which we split into three steps.

\medskip
\step1 Proof that $\Psi(\cdot,i\xi)$ as defined in~\eqref{eq:defin-k} is stationary and has vanishing expectation and finite second moments,
\begin{equation}\label{eq:estim-mom-Phi}
\E[\Psi(\cdot,i\xi)]\,=\,0,\qquad\E[|\Psi(\cdot,i\xi)e|^2]\,\le\,C_0^4|e|^2,\qquad\text{for all $e\in\R^d$}.
\end{equation}
The well-definiteness of $\Psi(\cdot,i\xi)$ as a stationary field with vanishing expectation and finite second moments is a consequence of the Lax--Milgram lemma as e.g.\@ in~\cite[Section~7.2]{JKO94}. We turn to the proof of the actual bound~\eqref{eq:estim-mom-Phi} on second moments.
By uniform ellipticity~\eqref{eq:ellipt-a} and by definition of~$\Psi(\cdot,i\xi)$, we find
\[\E[|\Psi(\cdot,i\xi)e|^2]\le C_0\E\big[(\overline{\Psi(\cdot,i\xi)e})\cdot \Aa\Psi(\cdot,i\xi)e\big]\,=\,-C_0\E[e\cdot \Aa\Psi(\cdot,i\xi)e],\]
and thus, by the Cauchy--Schwarz inequality,
\[\E[|\Psi(\cdot,i\xi)e|^2]\le C_0^2\E[|\Aa e|^2].\]
The claim~\eqref{eq:estim-mom-Phi} then follows from the boundedness of $\Aa$, cf.~\eqref{eq:ellipt-a}.

\medskip
\step2 Proof that the matrix $\bar\Ac(i\xi)$ as defined in~\eqref{eq:defin-barA-re} is uniformly elliptic and bounded in the sense of
\[e\cdot\bar\Ac(i\xi)e\,\ge\,\tfrac1{C_0}|e|^2,\qquad|\bar\Ac(i\xi)e|\le C_0^3|e|,\qquad\text{for all $e\in\R^d$}.\]
By definition of $\bar\Ac(i\xi)$ and $\Psi(\cdot,i\xi)$, we have
\[e\cdot\bar\Ac(i\xi)e\,=\,\E[e\cdot \Aa(e+\Psi(\cdot,i\xi)e)]
\,=\,\E\big[(\overline{e+\Psi(\cdot,i\xi)e})\cdot \Aa(e+\Psi(\cdot,i\xi)e)\big],\]
and thus, by uniform ellipticity~\eqref{eq:ellipt-a} and by Jensen's inequality,
\[e\cdot\bar\Ac(i\xi)e\,\ge\,\tfrac1{C_0}\E[|e+\Psi(\cdot,i\xi)e|^2]\,\ge\,\tfrac1{C_0}|e|^2.\]
For the upper bound, we start by noting that the same argument as in Step~1 yields
\begin{equation*}
\E[|e+\Psi(\cdot,i\xi)e|^2]
\,\le\,C_0\E\big[(\overline{e+\Psi(\cdot,i\xi)e})\cdot\Aa(e+\Psi(\cdot,i\xi)e)\big]
\,=\,C_0\E\big[e\cdot\Aa(e+\Psi(\cdot,i\xi)e)\big],
\end{equation*}
and thus
\[\E[|e+\Psi(\cdot,i\xi)e|^2]\,\le\,
C_0^2\E[|\Aa e|^2]\,\le\,C_0^4|e|^2,\]
which leads us to
\[|\bar\Ac(i\xi)e|\,=\,|\E[\Aa(e+\Psi(\cdot,i\xi)e)]|\,\le\,C_0\E[|e+\Psi(\cdot,i\xi)e|]\,\le\,C_0^3|e|,\]
as claimed.

\medskip
\step3 Proof that the ensemble average $\E[\nabla u_{\e,f}]=\nabla\bar u_{\e,f}$ satisfies the following well-posed pseudo-differential equation,
\[-\nabla\cdot\bar\Ac(\e\nabla)\nabla\bar u_{\e,f}=\nabla\cdot f,\]
and that fluctuations of $\nabla u_{\e,f}$ can be described through
\[\nabla u_{\e,f}-\E[\nabla u_{\e,f}]\,=\,\Psi(\tfrac\cdot\e,\e\nabla)\E[\nabla u_{\e,f}].\]
In terms of the projections $P=\E$ and $P^\bot=\Id-\E$ on $\Ld^2(\R^d\times\Omega)$, we consider the block decomposition~\eqref{eq:block-schur} of the elliptic operator $L=-\nabla\cdot\Aa\nabla$. By the Schur complement formula, this entails the following block decomposition of the solution operator on~\mbox{$\Ld^2(\R^d\times\Omega)$,}
\begin{eqnarray*}
P\nabla L^{-1}\nabla P
&=&\nabla\big(PLP-PLP^\bot(P^\bot LP^\bot)^{-1}P^\bot LP\big)^{-1}\nabla,\\
P^\bot\nabla L^{-1}\nabla P
&=&-\nabla(P^\bot LP^\bot)^{-1}P^\bot LP\big(PLP-PLP^\bot(P^\bot LP^\bot)^{-1}P^\bot LP\big)^{-1}\nabla,
\end{eqnarray*}
provided that the inverses do make sense. In terms of $\Psi$ and $\bar\Ac$, cf.~\eqref{eq:defin-K} and~\eqref{eq:defin-barA}, this precisely means
\begin{eqnarray*}
P\nabla L^{-1}\nabla P
&=&\nabla(-\nabla\cdot \bar\Ac(\nabla)\nabla)^{-1}\nabla P,\\
P^\bot\nabla L^{-1}\nabla P
&=&\Psi(\cdot,\nabla) \nabla(-\nabla\cdot \bar\Ac(\nabla)\nabla)^{-1}\nabla P.
\end{eqnarray*}
By Steps~1 and~2, these expressions both make sense on $\Ld^2(\R^d\times\Omega)^d$ and the conclusion then follows by $\e$-scaling.\qed

\section{Conlon--Naddaf lemma}\label{ssect:CN}
We recall the following result due to Conlon and Naddaf~\cite{conlon2000green}, which provides some frequency decay for derivatives of the symbol $\bar\Ac$. This high-frequency result is somehow orthogonal to the local regularity of the symbol that we are concerned with elsewhere in this work: it is unrelated to homogenization and it holds in the general stationary setting without any mixing assumption.
We refer to~\cite[Lemmas~3.9 and~3.10]{conlon2000green} for a proof in the discrete setting, which is easily generalized to the continuous setting as indicated in~\cite[Lemma~6.4 and eqn~(6.26)]{conlon2000green}.
A sketch of the proof is included below for the readers' convenience.

\begin{lem}[Conlon \& Naddaf~\cite{conlon2000green}]\label{lem:CN}
Let $\Aa$ be a stationary measurable random coefficient field satisfying the uniform ellipticity and boundedness assumptions~\eqref{eq:ellipt-a}, and let $\bar\Ac(\nabla)$ be the bounded convolution operator defined in Lemma~\ref{lem:Sigal}.
For all integers $n<d$ we have~$\nabla_\xi^n\bar\Ac\in \Ld^{d/n}_w(i\R^d)$, and in addition for all $0<\eta<1$,
\[\sup_{0<|y|\le1}\bigg\|\frac{\nabla_\xi^n\bar\Ac(\cdot+iy)-\nabla_\xi^n\bar\Ac}{|y|^\eta}\bigg\|_{\Ld^{\frac d{n+\eta}}_w(i\R^d)}\,\lesssim_\eta\,1.\]
\end{lem}

\begin{proof}[Sketch of the proof]
Let us focus on the proof that $\nabla_\xi\bar\Ac\in\Ld_w^d(i\R^d)$.
We recall that
\begin{gather*}
\bar\Ac(i\xi)\,=\,\E[\Aa(\Id+\Psi(\cdot,i\xi))],\qquad\Psi(\cdot,i\xi)=K_\xi \Aa,\\
K_\xi:=P^\bot\nabla_\xi\big(-\nabla_\xi\cdot P^\bot\Aa P^\bot\nabla_\xi\big)^{-1}\nabla_\xi\cdot P^\bot,\qquad\nabla_\xi:=\nabla+i\xi.
\end{gather*}
Taking the derivative of the above expression for $\bar\Ac(i\xi)$ with respect to $\xi$, we find
\begin{multline*}
\partial_{\xi_l}\bar\Ac(i\xi)\,=\,2\E\Big[\Aa P^\bot e_l\big(-\nabla_\xi\cdot P^\bot\Aa P^\bot\nabla_\xi\big)^{-1}\nabla_\xi\cdot P^\bot\Aa\Big]\\
+2\E\Big[\Aa P^\bot \nabla_\xi\big(-\nabla_\xi\cdot P^\bot\Aa P^\bot\nabla_\xi\big)^{-1}\big(e_l\cdot P^\bot\Aa P^\bot\nabla_\xi\big)\big(-\nabla_\xi\cdot P^\bot\Aa P^\bot\nabla_\xi\big)^{-1}\nabla_\xi\cdot P^\bot\Aa\Big],
\end{multline*}
which can be written as follows (without being too precise with matrix contractions),
\begin{equation}\label{eq:diff-A(xi)}
\nabla_{\xi}\bar\Ac(i\xi)
\,=\,2\E[\Aa  U_\xi K_\xi \Aa]
+2\E[\Aa K_\xi U_\xi \Aa K_\xi \Aa],
\end{equation}
in terms of the Riesz potential
\[U_\xi:=\nabla_\xi\Delta_\xi^{-1}P^\bot.\]
Note that the operator $K_\xi$ obviously satisfies for all $\xi$,
\[\|K_\xi\|_{\Ld^2(\Omega)\to\Ld^2(\Omega)}\,\le\,1,\] 
while on the contrary $U_\xi$ is not bounded on $\Ld^2(\Omega)$ for any fixed $\xi$.
To grasp a better understanding of $U_\xi$, we first note that it can be written as
\begin{equation}\label{eq:rep-Uxi}
U_\xi \phi\,=\,\int_{\R^d}e^{-ix\cdot\xi}\,\nabla\Delta^{-1}(x)\,P^\bot\phi(\tau_x\cdot)\,dx\,=\,h\ast P^\bot\hat\phi(\xi),\qquad h(\xi):=\xi|\xi|^{-2},
\end{equation}
in terms of the Fourier transform $\hat\phi(\xi):=\int e^{-ix\cdot\xi}\phi(\tau_x\cdot)\,dx$, which is defined almost surely in the sense of tempered distributions, where $x\mapsto\phi(\tau_x\cdot)$ stands for the stationary extension of $\phi$.
For fixed $\phi\in\Ld^2(\Omega)$, this motivates to consider the map $U_\phi:f\mapsto f\ast P^\bot\hat\phi$ defined on~$C^\infty_c(\R^d)$.
As by Bochner's theorem the Fourier transform~$\hat C_\phi$ of the covariance function $C_\phi(x):=\cov{\phi(\tau_x\cdot)}\phi$ is a positive measure with $\int_{\R^d}\hat C_\phi(dk)\le\|\phi\|_{\Ld^2(\Omega)}^2$, we can compute
\[\|(U_\phi f)(\xi)\|_{\Ld^2(\Omega)}\,=\,\Big(\int_{\R^d}|f(\xi+k)|^2\hat C_\phi(dk)\Big)^\frac12,\]
and we note that
\begin{eqnarray*}
\|U_\phi f\|_{\Ld^2(\R^d;\Ld^2(\Omega))}&\le&\|\phi\|_{\Ld^2(\Omega)}\|f\|_{\Ld^2(\R^d)},\\
\|U_\phi f\|_{\Ld^\infty(\R^d;\Ld^2(\Omega))}&\le&\|\phi\|_{\Ld^2(\Omega)}\|f\|_{\Ld^\infty(\R^d)}.
\end{eqnarray*}
By Hunt's interpolation, this entails
for all $2<p<\infty$,
\[\|U_\phi f\|_{\Ld^p_w(\R^d;\Ld^2(\Omega))}\,\le\,C_p\|\phi\|_{\Ld^2(\Omega)}\|f\|_{\Ld_w^p(\R^d)}.
\]
Applying this to~\eqref{eq:rep-Uxi} in form of $U_\xi\phi=(U_\phi h)(\xi)$ with $h\in\Ld^d_w(\R^d)$,
we deduce that the map $\phi\mapsto U\phi$ given by $U\phi(\xi):=U_\xi\phi$ satisfies
\[\|U\phi\|_{\Ld^d_w(\R^d;\Ld^2(\Omega))}\,\lesssim\,\|\phi\|_{\Ld^2(\Omega)}.\]
Using this and the boundedness of $K$ to estimate~\eqref{eq:diff-A(xi)}, we can deduce $\nabla_\xi\bar\Ac\in\Ld^d_w(\R^d)$.
As shown in~\cite{conlon2000green}, for higher derivatives, a careful (nontrivial) iteration of this argument is possible and yields the conclusion.
\end{proof}

\section*{Acknowledgements}
MD acknowledges financial support from the F.R.S.-FNRS, as well as from the European Union (ERC, PASTIS, Grant Agreement n$^\circ$101075879).\footnote{Views and opinions expressed are however those of the authors only and do not necessarily reflect those of the European Union or the European Research Council Executive Agency. Neither the European Union nor the granting authority can be held responsible for them.} The research of ML is partially supported by the DFG
through the grant TRR 352 -- Project-ID 470903074.
 FP was funded by the SNSF.

\bibliographystyle{plain}
\bibliography{biblio}

\def\cprime{$'$} \def\cprime{$'$} \def\cprime{$'$}
\begin{thebibliography}{10}

\bibitem{armstrong2019quantitative}
S.~Armstrong, T.~Kuusi, and J.-C. Mourrat.
\newblock {\em Quantitative stochastic homogenization and large-scale
  regularity}, volume 352.
\newblock Springer, 2019.

\bibitem{AS}
S.~N. Armstrong and C.~K. Smart.
\newblock Quantitative stochastic homogenization of convex integral
  functionals.
\newblock {\em Ann. Sci. \'Ec. Norm. Sup\'er. (4)}, 49(2):423--481, 2016.

\bibitem{BGO-17}
P.~Bella, A.~Giunti, and F.~Otto.
\newblock Quantitative stochastic homogenization: local control of
  homogenization error through corrector.
\newblock In {\em Mathematics and materials}, volume~23 of {\em IAS/Park City
  Math. Ser.}, pages 301--327. Amer. Math. Soc., Providence, RI, 2017.

\bibitem{BGO-20}
P.~Bella, A.~Giunti, and F.~Otto.
\newblock Effective multipoles in random media.
\newblock {\em Comm. Partial Differential Equations}, 45(6):561--640, 2020.

\bibitem{Benedek-Calderon-Panzone-62}
A.~Benedek, A.-P. Calder\'{o}n, and R.~Panzone.
\newblock Convolution operators on {B}anach space valued functions.
\newblock {\em Proc. Nat. Acad. Sci. U.S.A.}, 48:356--365, 1962.

\bibitem{BLP-78}
A.~Bensoussan, J.-L. Lions, and G.~Papanicolaou.
\newblock {\em Asymptotic analysis for periodic structures}, volume~5 of {\em
  Studies in Mathematics and its Applications}.
\newblock North-Holland Publishing Co., Amsterdam-New York, 1978.

\bibitem{Bourgain-18}
J.~Bourgain.
\newblock On a homogenization problem.
\newblock {\em J. Stat. Phys.}, 172(2):314--320, 2018.

\bibitem{Conlon-Giunti-Otto-17}
J.~G. Conlon, A.~Giunti, and F.~Otto.
\newblock Green's function for elliptic systems: existence and
  {D}elmotte-{D}euschel bounds.
\newblock {\em Calc. Var. Partial Differential Equations}, 56(6):Paper No. 163,
  51, 2017.

\bibitem{conlon2000green}
J.~G. Conlon and A.~Naddaf.
\newblock Green's functions for elliptic and parabolic equations with random
  coefficients.
\newblock {\em New York J. Math}, 6(153):225, 2000.

\bibitem{CN}
J.~G. Conlon and A.~Naddaf.
\newblock On homogenization of elliptic equations with random coefficients.
\newblock {\em Elec. Journal of Probability}, 5(Paper no. 9):1--58, 2000.

\bibitem{Conlon-Spencer-14a}
J.~G. Conlon and T.~Spencer.
\newblock Strong convergence to the homogenized limit of elliptic equations
  with random coefficients.
\newblock {\em Trans. Amer. Math. Soc.}, 366(3):1257--1288, 2014.

\bibitem{Delmotte-Deuschel-05}
T.~Delmotte and J.-D. Deuschel.
\newblock On estimating the derivatives of symmetric diffusions in stationary
  random environment, with applications to {$\nabla\phi$} interface model.
\newblock {\em Probab. Theory Related Fields}, 133(3):358--390, 2005.

\bibitem{Doukhan-94}
P.~Doukhan.
\newblock {\em Mixing}, volume~85 of {\em Lecture Notes in Statistics}.
\newblock Springer-Verlag, New York, 1994.
\newblock Properties and examples.

\bibitem{D-20a}
M.~Duerinckx.
\newblock {On the size of chaos via Glauber calculus in the classical
  mean-field dynamics}.
\newblock {\em Commun. Math. Phys.}, 2021.
\newblock {I}n press.

\bibitem{D-21a}
M.~Duerinckx.
\newblock Non-perturbative approach to the {B}ourgain--{S}pencer conjecture in
  stochastic homogenization.
\newblock {\em J. Math. Pures Appl.}, 176:183--225, 2023.

\bibitem{DGL}
M.~Duerinckx, A.~Gloria, and M.~Lemm.
\newblock A remark on a surprising result by {B}ourgain in homogenization.
\newblock {\em Comm. Part. Diff. Eq.}, 44(2):1345--1357, 2019.

\bibitem{DGO1}
M.~Duerinckx, A.~Gloria, and F.~Otto.
\newblock The structure of fluctuations in stochastic homogenization.
\newblock {\em Comm. Math. Phys.}, 377:259--306, 2020.

\bibitem{DO1}
M.~Duerinckx and F.~Otto.
\newblock Higher-order pathwise theory of fluctuations in stochastic
  homogenization.
\newblock {\em Stoch. Partial Differ. Equ. Anal. Comput.}, 8:625--692, 2020.

\bibitem{Gloria-Habibi-16}
A.~Gloria and Z.~Habibi.
\newblock Reduction in the resonance error in numerical homogenization {II}:
  {C}orrectors and extrapolation.
\newblock {\em Found. Comput. Math.}, 16(1):217--296, 2016.

\bibitem{GMa}
A.~Gloria and D.~Marahrens.
\newblock Annealed estimates on the {G}reen functions and uncertainty
  quantification.
\newblock {\em Ann. Inst. H. Poincar\'e Anal. Non Lin\'eaire},
  33(5):1153--1197, 2016.

\bibitem{GNO-reg}
A.~Gloria, S.~Neukamm, and F.~Otto.
\newblock A regularity theory for random elliptic operators.
\newblock {\em Milan J. Math.}, 88(1):99--170, 2020.

\bibitem{Gu-17}
Y.~Gu.
\newblock High order correctors and two-scale expansions in stochastic
  homogenization.
\newblock {\em Probab. Theory Related Fields}, 169(3-4):1221--1259, 2017.

\bibitem{HS-94}
B.~Helffer and J.~Sj\"ostrand.
\newblock On the correlation for {K}ac-like models in the convex case.
\newblock {\em J. Stat. Phys.}, 74(1-2):349--409, 1994.

\bibitem{JKO94}
V.~V. Jikov, S.~M. Kozlov, and O.~A. Ole{\u\i}nik.
\newblock {\em Homogenization of differential operators and integral
  functionals}.
\newblock Springer-Verlag, Berlin, 1994.

\bibitem{Lemm-Keller-21}
M.~Keller and M.~Lemm.
\newblock Asymptotic expansion of the annealed {G}reen's function and its
  derivatives.
\newblock Preprint, arXiv:2107.11583, 2021.

\bibitem{Lemm-18}
J.~Kim and M.~Lemm.
\newblock {On the Averaged Green's Function of an Elliptic Equation with Random
  Coefficients}.
\newblock {\em Arch. Ration. Mech. Anal.}, 234:1121--1166, 2019.

\bibitem{marahrens2015annealed}
D.~Marahrens and F.~Otto.
\newblock Annealed estimates on the green function.
\newblock {\em Probability theory and related fields}, 163:527--573, 2015.

\bibitem{MO-16}
J.-C. Mourrat and F.~Otto.
\newblock Correlation structure of the corrector in stochastic homogenization.
\newblock {\em Ann. Probab.}, 44(5):3207--3233, 2016.

\bibitem{naddaf1997homogenization}
A.~Naddaf and T.~Spencer.
\newblock On homogenization and scaling limit of some gradient perturbations of
  a massless free field.
\newblock {\em Communications in mathematical physics}, 183:55--84, 1997.

\bibitem{NP-book}
I.~Nourdin and G.~Peccati.
\newblock {\em Normal approximations with {M}alliavin calculus. From Stein's
  method to universality}, volume 192 of {\em Cambridge Tracts in Mathematics}.
\newblock Cambridge University Press, Cambridge, 2012.

\bibitem{Sigal}
I.~M. Sigal.
\newblock Homogenization problem.
\newblock Unpublished preprint.

\bibitem{Stein-70}
E.~M. Stein.
\newblock {\em Singular integrals and differentiability properties of
  functions}.
\newblock Princeton Mathematical Series, No. 30. Princeton University Press,
  Princeton, N.J., 1970.

\bibitem{SteinWeiss}
E.~M. Stein and G.~Weiss.
\newblock {\em Introduction to Fourier analysis on Euclidean spaces}, volume~1.
\newblock Princeton university press, 1971.

\bibitem{uchiyama1998green}
K.~Uchiyama.
\newblock Green's functions for random walks on $\mathbb z^n$.
\newblock {\em Proceedings of the London Mathematical Society}, 77(1):215--240,
  1998.

\end{thebibliography}

\end{document}